\documentclass[10pt]{amsart}

\usepackage{amssymb}
\usepackage{graphicx}
\usepackage{graphics}
\usepackage{url}
\usepackage{amscd}
\usepackage{textcomp}
\usepackage{enumitem}
\usepackage[all]{xy}
\usepackage{mathtools} 
\usepackage{xcolor}
\usepackage{xspace}
\usepackage[normalem]{ulem} 

\usepackage[margin=3cm, marginparwidth=2.5cm]{geometry}

\definecolor{darkblue}{rgb}{0,0,0.4} 
\usepackage[colorlinks=true, citecolor=darkblue, filecolor=darkblue, linkcolor=darkblue,urlcolor=darkblue]{hyperref} 
\usepackage[all]{hypcap}
\usepackage{xr-hyper}

\usepackage{bbm} 

\usepackage{pgffor}
\usepackage{ifthen}
\usepackage{tikz,ulem}
\usetikzlibrary{decorations.markings,decorations.pathreplacing,calc,arrows}
\usepackage{s2s1leetikz}  
\usepackage{s2s1leebigtikz}

\newtheorem{thm}{Theorem}[section]
\newtheorem{theorem}[thm]{Theorem}

\newtheorem{corollary}[thm]{Corollary}

\newtheorem{lemma}[thm]{Lemma}

\newtheorem{proposition}[thm]{Proposition}

\newtheorem{conjecture}[thm]{Conjecture}

\newtheorem{question}[thm]{Question}
\theoremstyle{definition}

\newtheorem{definition}[thm]{Definition}

\theoremstyle{remark}

\newtheorem{remark}[thm]{Remark}

\newtheorem{example}[thm]{Example}

\numberwithin{equation}{section}

\newcommand{\td}{\widetilde}

\renewcommand{\theta}{\vartheta}

\newcommand{\mi}{\mu}

\newcommand{\SSone}{S^1\times S^2} 
\newcommand{\SSr}{\#^r(S^1\times S^2)}  

\newcommand{\set}[1]{\left\{#1\right\}}

\newcommand{\ceiling}[1]{\left\lceil#1\right\rceil}

\DeclareMathOperator{\lk}{lk}
\DeclareMathOperator{\writhe}{wr}
\renewcommand{\wr}{\writhe}
\DeclareMathOperator{\br}{br}
\DeclareMathOperator{\selflinking}{sl}
\renewcommand{\sl}{\selflinking}

\DeclareMathOperator{\Hopf}{Hopf}
\DeclareMathOperator{\Wh}{Wh}
\newcommand{\tw}{{tw}}
\newcommand{\de}{\partial}
\newcommand{\sm}{\setminus}

\newcommand{\Z}{\mathbb{Z}}
\newcommand{\Q}{\mathbb{Q}}
\newcommand{\R}{\mathbb{R}}
\newcommand{\C}{\mathbb{C}}
\newcommand{\F}{\mathbb{F}}

\newcommand{\nbd}{\mathcal{N}} 
\newcommand{\del}{\partial}
\newcommand{\CP}{\mathbb{CP}^2} 
\newcommand{\bCP}{\overline{\mathbb{CP}^2}}
\newcommand{\intceil}[1]{\left\lceil#1\right\rceil}
\newcommand{\xistd}{\xi_{\mathit{std}}}
\mathchardef\mhyphen="2D

\newcommand{\BartoKh}{\mathrm{K}}
\newcommand{\OKCbmnobrackets}{\mathrm{KC}}
\newcommand{\OLC}[2][]{\mathrm{KC}_{\mathrm{Lee}}^{#1}(#2)}
\newcommand{\OLCnobrackets}{\mathrm{KC}_{\mathrm{Lee}}}
\newcommand{\OLCrrnnobrackets}{\mathrm{K}\widehat{\mathrm{C}}_{\mathrm{Lee}}}
\newcommand{\OdKC}[2][]{\mathrm{KC}'^{#1}( #2 )}
\newcommand{\OdKCp}[2][]{\mathrm{KC}'^{#1}\left( #2 \right)}
\newcommand{\OKC}[2][]{\mathrm{KC}^{#1}(#2)}
\newcommand{\OKCp}[2][]{\mathrm{KC}^{#1} \! \left( #2 \right)}
\newcommand{\CSharp}{\mathrm{C}^{\#}}
\newcommand{\CSharpSub}[1]{\mathrm{C}^{\#}_{#1}}
\newcommand{\CSharpLee}{\mathrm{C}^{\#}_{\mathrm{Lee}}}
\newcommand{\OCSharpLee}{\BartoKh\mathrm{C}^{\#}_{\mathrm{Lee}}}

\newcommand{\KC}[2][]{\mathrm{C}^{#1}(#2)} 
\newcommand{\LCp}[2][]{\mathrm{C}_{\mathrm{Lee}}^{#1} \left( #2 \right)} 
\newcommand{\LC}[2][]{\mathrm{C}_{\mathrm{Lee}}^{#1}(#2)} 
\newcommand{\LCnobrackets}{\mathrm{C}_{\mathrm{Lee}}} 
\newcommand{\wLCnobrackets}{\widehat{\mathrm{C}}_{\mathrm{Lee}}}
\newcommand{\LCrrn}[2][]{\widehat{\mathrm{C}}_{\mathrm{Lee}}^{#1} \left( #2 \right)} 
\newcommand{\LCrrnnobrackets}{\widehat{\mathrm{C}}_{\mathrm{Lee}}}
\newcommand{\dKC}[2][]{\mathrm{C}'^{#1}( #2 )} 
\newcommand{\dKCrrn}{\widehat{\mathrm{C}}'} 
\newcommand{\Kh}[2][]{\mathrm{Kh}^{#1}( #2 )} 
\newcommand{\Lh}[2][]{\mathrm{Kh}_{\mathrm{Lee}}^{#1} ( #2 )} 
\newcommand{\dKh}[2][]{\mathrm{Kh}'^{#1}( #2 )} 
\newcommand{\dKhnobrackets}{\mathrm{Kh}'} 
\newcommand{\KCbm}[2][]{\mathrm{KC}^{#1}\left( #2 \right) }  
\newcommand{\HH}[2][]{\mathit{HH}^{#1} \left( #2 \right) } 
\newcommand{\HHnobrackets}{\mathit{HH}}
\newcommand{\dKCbm}[2][]{\mathrm{KC}'^{#1}\left( #2 \right) }  
\newcommand{\dHH}[2][]{\mathit{HH}'^{#1} \left( #2 \right) } 
\newcommand{\LCbm}[2][]{\mathrm{KC}^{#1}_{\mathrm{Lee}} \left( #2 \right) }  
\newcommand{\LHH}[2][]{\mathit{HH}_{\mathrm{Lee}}^{#1} \left( #2 \right) } 

\newcommand{\TL}{\mathrm{TL}} 
\newcommand{\Kob}{\mathrm{Kob}} 
\newcommand{\dTL}{\mathrm{TL}'} 
\newcommand{\h}{\mathrm{h}} 
\newcommand{\q}{\mathrm{q}} 
\newcommand{\Lk}[1][D]{#1(\vec{k})} 
\newcommand{\Lzero}[1][D]{#1(\vec{0})} 
\newcommand{\thru}[1]{\mathrm{th}(#1)} 
\renewcommand{\d}{\mathfrak{s}} 
\newcommand{\e}{\varepsilon} 
\newcommand{\s}{\mathfrak{s}} 
\newcommand{\x}{\mathbf{x}} 
\newcommand{\y}{\mathbf{y}} 
\newcommand{\z}{\mathbf{z}} 
\newcommand{\w}{\mathbf{w}} 
\newcommand{\sav}{s_{\mathrm{av}}} 
\newcommand{\Pl}{\mathcal{P}} 
\newcommand{\id}{\operatorname{id}} 
\newcommand{\taut}{\tau_{\operatorname{top}}} 
\newcommand{\taub}{\tau_{\operatorname{bot}}}

\newcommand{\FT}{\mathcal{T}} 
\newcommand{\infFT}{\mathcal{T}^\infty} 
\newcommand{\m}[1]{-{#1}} 
\newcommand{\caps}{\mathfrak{C}} 
\newcommand{\caprefl}[1]{\widehat{#1}} 
\newcommand{\gSD}{g_{S^1 \times B^3}} 
\newcommand{\gDS}{g_{B^2 \times S^2}} 
\newcommand{\Fp}{F_{p,p}} 
\newcommand{\Fpp}[1]{F_{#1,#1}} 
\newcommand{\Seif}{\varsigma} 
\newcommand{\DT}{\sigma} 
\newcommand{\Diff}{\operatorname{Diff}} 

\begin{document}

\title[A generalization of Rasmussen's invariant]{A generalization of Rasmussen's invariant,\\with applications to surfaces in some four-manifolds}

\author[Manolescu]{Ciprian Manolescu}%
\address{Department of Mathematics, Stanford University, Stanford, CA 94305}%
\email{\href{mailto:cm5@stanford.edu}{cm5@stanford.edu}}
\thanks{CM was supported by NSF Grant DMS-1708320}

\author[Marengon]{Marco Marengon}%
\address{Alfr\'ed R\'enyi Institute of Mathematics, Budapest, Hungary}%
\thanks{MM was supported by NSF FRG Grant DMS-1563615.}
\email{\href{mailto:marengon@renyi.hu}{marengon@renyi.hu}}%

\author[Sarkar]{Sucharit Sarkar}
\address{Department of Mathematics, University of California, Los Angeles, CA 90095}
\thanks{SS was supported by NSF Grant DMS-1643401.}
\email{\href{mailto:sucharit@math.ucla.edu}{sucharit@math.ucla.edu}}

\author[Willis]{Michael Willis}
\address{Department of Mathematics, Stanford University, Stanford, CA 94305}%
\thanks{MW was supported by NSF FRG Grant DMS-1563615.}
\email{\href{mailto:msw188@stanford.edu}{msw188@stanford.edu}}

\begin{abstract} 
We extend the definition of Khovanov-Lee homology to links in connected sums of $S^1 \times S^2$'s, and construct a Rasmussen-type invariant for null-homologous links in these manifolds. For certain links in $S^1 \times S^2$, we compute the invariant by reinterpreting it in terms of Hochschild homology. As applications, we prove inequalities relating the Rasmussen-type invariant to the genus of surfaces with boundary in the following four-manifolds: $B^2 \times S^2$, $S^1 \times B^3$, $\CP$, and various connected sums and boundary sums of these. We deduce that Rasmussen's invariant also gives genus bounds for surfaces inside homotopy $4$-balls obtained from $B^4$ by Gluck twists. Therefore, it cannot be used to prove that such homotopy $4$-balls are non-standard.
\end{abstract}
\maketitle

\tableofcontents
\section{Introduction} \label{sec:Intro}

In \cite{Kh}, Khovanov developed his homology theory for links in $S^3$, as a categorification of the Jones polynomial. Lee \cite{Lee} constructed a deformation of Khovanov homology, which was later used by Rasmussen \cite{Rasmussen} to extract a numerical knot invariant, denoted $s$. Rasmussen showed that his invariant gives bounds on the slice genus, and thus managed to give a combinatorial proof of Milnor's conjecture about the slice genus of torus knots. 

An open question is to extend Khovanov homology to links inside other three-manifolds. So far, there are variants of Khovanov homology for links in $I$-bundles over surfaces \cite{APS, Gabrovsek}, in $\SSone$ \cite{Roz}, and in connected sums $\SSr$ \cite{MW}. See also \cite{MWW} for a general proposal.

In this paper we build on the work of the fourth author in \cite{MW} to extend Khovanov-Lee homology to links in $\SSr$. We can work over any ring $R$ where $2$ is invertible. For links in $S^3$, we keep track of Lee's deformation of the Khovanov complex through a variable $t$, as in \cite{Kh2}; the result is a complex over $R[t]$, whose homology is denoted $\dKhnobrackets$. For links $L\subset \SSr$, we represent them by diagrams $D$ in the plane with $r$ one-handles attached, as on the left of Figure~\ref{fig:InsertTwists}. Given such a diagram and a vector $\vec k = (k_1, \dots, k_r) \in \Z^r$, we get a diagram $D(\vec k)$ (for a link in $S^3$) by inserting $k_j$ copies of the full twist in place of the $j^{\text{th}}$ one-handle. Provided that the homology class of the link is $2$-divisible, one can show that the deformed Khovanov-Lee homology $\dKh{D(\vec k)}$, after a suitable grading shift, stabilizes in every given degree by taking $k_1, \dots, k_r$ sufficiently large. We define $\dKh{D}$ to be the stable limit of $ \dKh{D(\vec k)}$.

\begin{theorem}
\label{thm:deformedKh}
Let $D_1,D_2$ be two diagrams for an oriented link $L\subset M=\SSr$ such that $[L] \in H_1(M;\Z)$ is 2-divisible.  Then the homologies $\dKh{D_1}$ and $\dKh{D_2}$ are isomorphic up to grading shifts. Furthermore, if $[L]=0\in H_*(M;\Z)$, then these gradings shifts are zero, so the deformed Khovanov-Lee homology $\dKh{L}$ is a well-defined link invariant up to isomorphism, as a bi-graded module over $R[t]$.
\end{theorem}

\begin{figure}
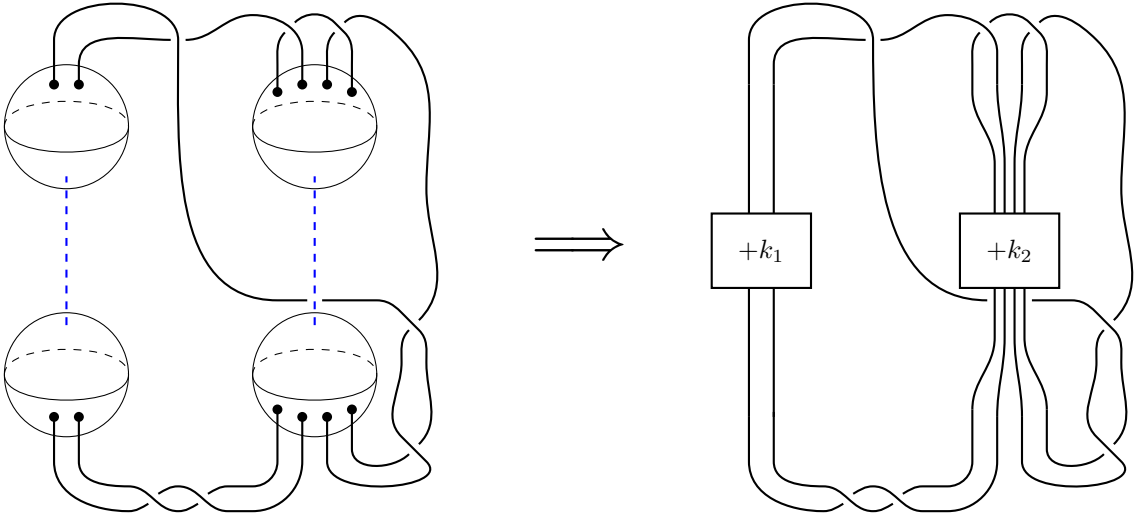

\centering
\LinMr
\hspace{.3in}
\resizebox{.5in}{!}{$\Longrightarrow$}
\hspace{.3in}
\LkinMr{+k_1}{+k_2}
\caption{On the left is an example of a diagram $D$ for a link
  $L \subset \#^2(\SSone)$.  The blue dashed lines indicate
  corresponding surgery spheres; a basepoint is chosen on each dashed line.  We
  convert this into the diagram $D(k_1, k_2)$ on the right by
  inserting several full twists in place of the
  dashed lines at the basepoints.}
\label{fig:InsertTwists}
\end{figure}

In the following, we set $R =\Q$ for simplicity. If we set $t=1$ in the deformed Khovanov-Lee complex, and then take homology, we obtain Lee's homology $\Lh{L}$. The Lee homology of any link $L\subset S^3$ was computed in \cite{Lee}: it contains one copy of $\Q$ for each possible orientation $o$ of $L$, and is in fact generated by cycles $\s_o$ corresponding to oriented resolutions of $L$. For links in $\SSr$, we have the following calculation.

 \begin{theorem}
 \label{thm:computeLee}
Let  $L\subset\SSr$ be a link such that $[L]$ is $2$-divisible. Suppose $L$ is represented by a diagram $D$. Let $O(L)$ denote the set of orientations of $L$ for which $L$ would be null-homologous in $\SSr$.  Then the Lee homology of $L$ is
\[\Lh{L}\cong \Q^{|O(L)|},\]
generated by cycles $\s_o$ corresponding to oriented resolutions of certain finite approximations $D(\vec{k})$ according to the orientations in $O(L)$.
\end{theorem}

Now let $L\subset\SSr$ be null-homologous, with orientation $o$ and opposite orientation $\overline{o}$. Let $D$ be a diagram for $L$. Following Rasmussen \cite{Rasmussen} and Beliakova-Wehrli \cite{beliakova-wehrli}, we define the \emph{$s$-invariant} of $D$ to be
\[s(D):= \frac{ \q([\s_o+\s_{\bar{o}}]) + \q([\s_o-\s_{\bar{o}}])}{2},\]
where $\q$ denotes the quantum filtration.

\begin{theorem}
\label{thm:s invt well defined}
If $D_1$ and $D_2$ are two diagrams for the same oriented, null-homologous link $L$ in $\SSr$, then $s(D_1)=s(D_2)$ and thus $s$ gives a well-defined invariant $s(L)$ of null-homologous links $L$ in $\SSr$.
\end{theorem}

The $s$-invariant of a link in $\SSr$ is closely related to that of links in $S^3$:
\begin{theorem}
\label{thm:fda}
\label{prop:s invt by finite approx}
Let $L\subset\SSr$ be an oriented, null-homologous link represented by a diagram $D$ with $n_D^+$ positive crossings. The $s$-invariant of $L$ can be computed via finite approximation as $s(L)=s(\Lk)$ where $\vec{k}=(k,\cdots,k)$ with $k \geq \intceil{\frac{n_D^++2}{2}}$.
\end{theorem}

Just as the usual $s$ gives constraints on the topology of link cobordisms in $I \times S^3$, our invariant $s$ does this for cobordisms in $I \times \SSr$.

\begin{theorem}
\label{thm:GenusBoundCylinders}
\label{cor:cob inequality}
Consider an oriented cobordism $\Sigma\subset I \times \SSr$ from a link $L_1$ to a second link $L_2$. Suppose that every component of $\Sigma$ has a boundary component in $L_1$. Then,  we have an inequality of $s$-invariants
\[s(L_2) - s(L_1) \geq  \chi(\Sigma),\]
where $\chi(\Sigma)$ denotes the Euler characteristic of $\Sigma$.
\end{theorem}
When $r=0$ (that is, for links in $S^3$), the inequality in Theorem~\ref{thm:GenusBoundCylinders}  was proved by Beliakova and Wehrli \cite{beliakova-wehrli}.

For links in a single copy of $S^1\times S^2$, there is an alternative method for calculating $s$, using Hochschild homology of Khovanov's tangle invariants. One could consider $(2p,2p)$-tangles in $I\times D^2$ which are being glued `end to end' to form a link in $S^1 \times D^2$. To such a tangle $Z\subset I\times D^2$, Khovanov \cite{KhTangles} defined a complex of bimodules $\OKCbmnobrackets(Z)$ over  the arc algebra $H_p$. Rozansky \cite{Roz} showed that the limiting complex $\OKCbmnobrackets(\infFT_{2p})$ of $H_p$-bimodules assigned to the infinite twist $\infFT_{2p}$ (the formal limit of finitely many full twists) gives a projective resolution of the identity bimodule $\OKCbmnobrackets(I)$.  As such, if $L\subset S^1\times S^2$ is formed by taking the closure of $Z\subset I\times D^2$ and collapsing $\partial D^2$, then 
\[\Kh{L}\cong \HHnobrackets(\OKCbmnobrackets(Z)),\]
where $\HHnobrackets$ denotes Hochschild homology. All of this continues to work for the deformed Khovanov-Lee complexes.

The Hochschild homology interpretation can be applied to compute the $s$-invariant of the link $\Fp \subset \SSone$ which is the union of $2p$ fibers $S^1 \times \{x_i\}$, for distinct $x_1, \dots, x_{2p} \in S^2$, where $p$ of the fibers are oriented in one direction and the other $p$ in the reverse direction. In other words, $\Fp$ is obtained from closing up the identity tangle in $I\times D^2$, which gives the identity bimodule over $H_p$. The link $\Fp$ is illustrated on the left side of Figure~\ref{fig:Fp}.

\begin{figure}
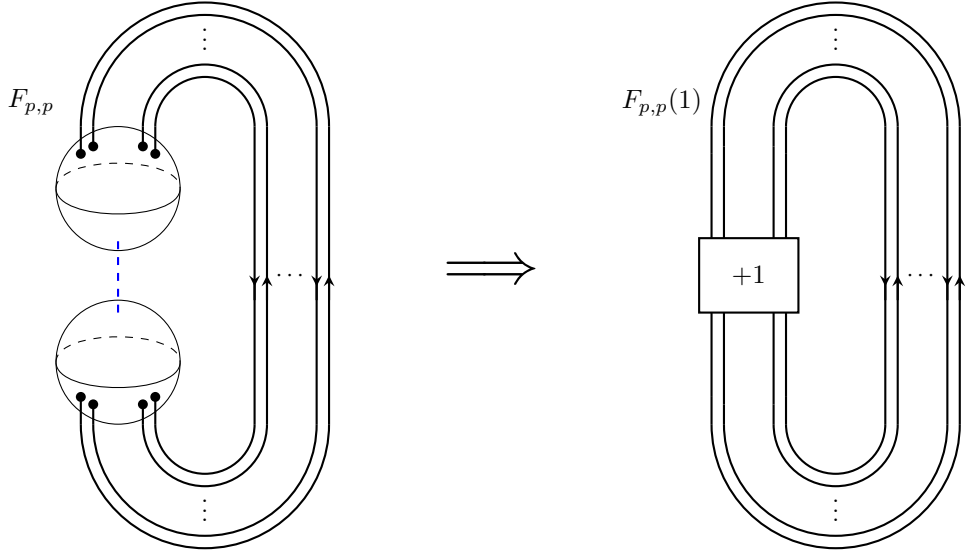

\centering
\FppinMr
\hspace{.5in}
\resizebox{.5in}{!}{$\Longrightarrow$}
\hspace{.25in}
\Fppone
\caption{The figure on the left shows the (null-homologous) link $\Fp\subset \SSone$. The figure on the right shows the link $\Fp(1) \subset S^3$.}
\label{fig:Fp}
\end{figure}

\begin{theorem}
\label{thm:s(Fp)}
For the link $\Fp \subset \SSone$, we have
\[s(\Fp)=1-2p.\]
\end{theorem}

The main idea in the proof of Theorem~\ref{thm:s(Fp)}  is to use Hochschild homology, and replace the infinite twist projective resolution by a more conventional `bar-like' projective resolution. The latter has a relatively simple map to the identity tangle that aids in computations.

Interestingly, by combining Theorems~\ref{thm:fda} and~\ref{thm:s(Fp)}, we are able to calculate the $s$-invariant of a link in $S^3$, which seems otherwise inaccessible. Precisely, we consider the link $\Fp(1)$, which is a finite approximation to $\Fp$, and is in fact the torus link $T(2p,2p)$ with $p$ strands oriented one way, and $p$ strands oriented the other way; cf. the right side of Figure~\ref{fig:Fp}. 

\begin{theorem}
\label{thm:sLpp}
For the link $\Fp(1) \subset S^3$, we have
\[
s(\Fp(1)) = 1-2p.
\]
\end{theorem}

Rasmussen's $s$-invariant has two close cousins: the Ozsv\'ath-Szab\'o $\tau$ invariant from knot Floer homology \cite{os-tau, RasmussenThesis}, and the Kronheimer-Mrowka $s^{\sharp}$ invariant from instanton homology \cite{KM, KMerr}. All of these are concordance invariants, but $\tau$ and $s^{\sharp}$ fit nicely  into more general theories (Heegaard Floer and instanton, respectively), which give invariants for $3$- and $4$-dimensional manifolds. In the absence of such a theory for Khovanov homology, we can still try to  probe four-dimensional properties of the $s$ invariant, with the hope that it would give us a clue about how Khovanov homology extends (or not) to more general manifolds. For example, Piccirillo's recent proof that the Conway knot is not slice \cite{Pic} relies essentially on the fact that the $s$ invariant does not behave as $\tau$ does with respect to $0$-surgeries on certain  knots.

Let $X$ be a closed oriented smooth $4$-manifold with no 1-handles and no 3-handles. 
Suppose we are interested in surfaces $\Sigma \subset X \setminus B^4$ with boundary $\del \Sigma = L \subset S^3$. We remove the co-cores of the $2$-handles and reverse the orientation of $\Sigma$ to obtain a cobordism in $S^3 \times I$ from the mirror reverse $\m L$ to a cable $J$ of attaching spheres of the $2$-handles. Thus, we can get constraints on the $s$-invariant of $L$ by applying Theorem~\ref{thm:GenusBoundCylinders}, provided we can calculate the $s$-invariants of all possible $J$. Computing $s$-invariants for infinite families of links is in general a difficult problem. However, for $X=\bCP$, if $[\Sigma] = 0$, then the cables of the attaching sphere in which we are interested are the torus links $\Fp(1)$, and we can use Theorem~\ref{thm:sLpp}.

The $\tau$ and $s^{\sharp}$ invariants satisfy adjunction-type
inequalities with respect to surfaces in negative-definite smooth
$4$-manifolds; see \cite{os-tau, KM}.  Using the idea of reduction to
cylinders outlined above, we obtain a similar inequality for $s$,
albeit we can only prove it for $\#^t \bCP$ instead of more general
negative definite smooth $4$-manifolds. However, note that all simply
connected, negative definite smooth $4$-manifolds are homeomorphic to
$\#^t \bCP$ by Freedman's work \cite{Freedman}, and it is unknown
whether $\#^t \bCP$ admits exotic smooth structures.

\begin{theorem}
\label{thm:weakadj}
Consider an oriented cobordism $\Sigma\subset Z = (\#^t\bCP) \setminus (B^4 \sqcup B^4)$ from a null-homologous link $L_1 \subset S^3$ to a second null-homologous link $L_2 \subset S^3$.
Suppose that $[\Sigma] = 0$ in $H_2(Z, \de Z)$ and that every component of $\Sigma$ has a boundary component in $L_2$.
Then, we have an inequality of $s$-invariants
\[s(L_1) - s(L_2) \geq \chi(\Sigma),\]
where $\chi(\Sigma)$ denotes the Euler characteristic of $\Sigma$.
\end{theorem}

By specializing Theorem \ref{thm:weakadj} to the case $L_1 = \varnothing$ (for which by convention we have $s(\varnothing) = 1$), we obtain a statement which is formally identical to the adjunction inequalitites for $\tau$ and $s^\sharp$~\cite{os-tau, KM}, in the case of nullhomologous $\Sigma$.

\begin{corollary}[Adjunction inequality for $s$]
\label{cor:adjunction}
Let $W = (\#^t \bCP) \setminus B^4$ for some $t \geq 0$. Let $L \subset \del W= S^3$ be a link, and $\Sigma \subset W$ a properly, smoothly embedded oriented surface with no closed components, such that $\del \Sigma =L$ and $[\Sigma]=0 \in H_2(W, \del W)$. Then 
\[
s(L) \leq 1-\chi(\Sigma).
\]
\end{corollary}

In particular, for knots we get that $s(K) \leq 2g$, where $g$ is the genus of $\Sigma$. Ozsv\'ath and Szab\'o \cite{os-tau} proved that the same inequality holds with $2\tau(K)$ instead of $s(K)$. While $s=2\tau$ for many families of knots, there are examples where the two invariants differ. The first such example, found by Hedden and Ording in \cite{HeddenOrding}, was the $2$-twisted positive Whitehead double of the right-handed trefoil, $\Wh^+(T_{2,3}, 2)$, for which $\tau=0$ but $s=2$. 

\begin{corollary}
\label{cor:D+}
The knot $\Wh^+(T_{2,3}, 2)$ does not bound a null-homologous smooth disk in $(\#^t \bCP) \setminus B^4$, for any $t \geq 0$.
\end{corollary}

Other obstructions to bounding null-homologous disks in negative
definite manifolds were given by Cochran, Harvey and Horn in
\cite[Proposition 1.2]{CHH}. These include the $d$-invariants of the
double branched cover and of $\pm 1$ surgeries on the knot, which are
$0$ for $\Wh^+(T_{2,3}, 2)$ by \cite[Theorems 3, 4, 8]{Tange}. Another
obstruction is the Tristram-Levine signature, which is again $0$ in
our case by \cite[Theorem 2]{Litherland}. So none of these invariants
can be used to prove Corollary~\ref{cor:D+}.

Another application of Theorem~\ref{thm:weakadj} involves the notion of
generalized crossing change, which was introduced by Cochran and
Tweedy in~\cite{Positive}. Given a diagram $D$ of a link $L$ in $S^3$,
a \emph{generalized crossing change from positive to negative} (also
called \emph{adding a generalized negative crossing}) consists of
adding a positive full twist on $D$, so that the number of strands
going up and the number of strands going down are the same. See Figure
\ref{fig:GCC}. Note that when only two strands are involved
(one going up and one going down) this operation is the usual crossing
change from positive to negative (combined with a Reidemeister II
move).

Here we warn the reader against a possible source of confusion. Note that a \emph{negative} crossing change (i.e., from positive to negative) can be realized by inserting a \emph{positive} full twist on two oppositely oriented strands. (It can also be realized by inserting a \emph{negative} full twist on two coherently oriented strands, but that is irrelevant here.) Analogously, adding a generalized \emph{negative} crossing change can be realized by inserting a \emph{positive} full twist on $2n$ strands, $n$ oriented one way and $n$ oriented the other way.

\begin{figure}
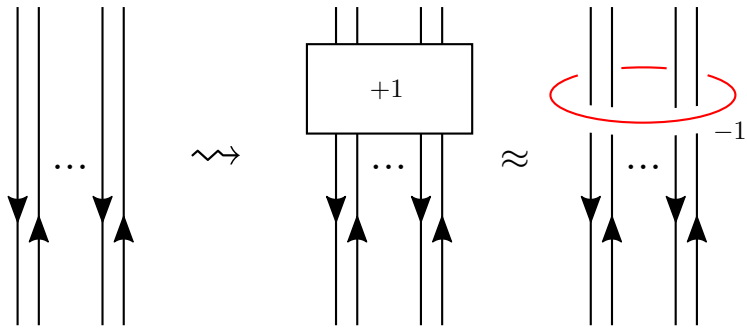
\caption{Adding a generalized negative crossing.}
\label{fig:GCC}
\end{figure}

\begin{theorem}
\label{thm:addFT}
Suppose that a link $L^{\tw} \subset S^3$ is obtained from $L$ by adding a generalized negative crossing. Then $s(L^{\tw}) \leq s(L)$.
\end{theorem}

Consider now the untwisted positive Whitehead double of the left-handed trefoil, $\Wh^+({T_{2,-3}},0)$. It is still an open problem to determine whether $\Wh^+({T_{2,-3}},0)$ is slice or not, since all gauge-theoretical invariants vanish for it. This is a very special case of a question  \cite[Problem 1.38]{Kirby} asking whether $\Wh^+(K,0)$ is slice only when $K$ is slice.

One can use a computer program to check that $s=0$ for $\Wh^+({T_{2,-3}},0)$. Hence, $s$ cannot detect whether this knot is slice. Theorem~\ref{thm:addFT} gives a more conceptual proof that $s=0$ in this case. In fact, this holds for a more general family of Whitehead doubles:

\begin{theorem}
\label{thm:sWh+=0}
Let $K \subset S^3$ be a knot that can be turned into the unknot by doing only crossing changes from negative to positive. (For example, $K$ could be a negative knot.) Then, $s(\Wh^+(K, 0)) = 0$.
\end{theorem}

Freedman, Gompf, Morrison, and Walker \cite{mnm} proposed the following approach to disproving the smooth four-dimensional Poincar\'e conjecture: Find a homotopy $4$-sphere $X$ and a knot $K\subset S^3$ that is slice in $X$ (i.e., bounds a smoothly embedded disk in $X \setminus B^4$), and prove that is not slice in $S^4$ by showing that $s(K) \neq 0$. More generally, one could try to find a link $L$ that is strongly slice in $X$ (i.e., bounds a collection of disjoint, smoothly embedded disks in  $X \setminus B^4$), and show that $s(L) \neq 1-|L|$, where $|L|$ is the number of components of $L$; such a link cannot be strongly slice in $S^4$.

It follows from Corollary \ref{cor:adjunction} that this strategy cannot work for a large class of potential counterexamples to the smooth 4D Poincar\'e conjecture, namely those obtained by Gluck twists. 
\begin{corollary}
\label{cor:Glucktwist}
Let $X$ be a homotopy 4-sphere obtained by a Gluck twist on a $2$-knot in $S^4$. If a link $L \subset S^3$ is strongly slice in $X$, then $s(L) = 1-|L|$.
\end{corollary}

A key ingredient in the proof of Corollary~\ref{cor:Glucktwist} is that, by a property of Gluck twists, we have $X \# \CP \cong \CP$ and $X \# \bCP \cong \bCP$. Thus, we can use what we know about the behavior of $s$ with respect to surfaces in $\bCP \setminus B^4$. 

Many other properties of the $s$-invariants in $S^3$ and in $\SSr$ are established in Sections~\ref{sec:sS3} and \ref{sec:sMr}. Let us mention a few of them here. For a null-homologous link $L$ in $ \#^r(\SSone)$, we define $s_-(L) = s(L)$ and $s_+(L) = -s(\m L)$. (Note that for knots $K \subset S^3$, we have $s_-(K)=s_+(K)$ by \cite[Proposition 3.9]{Rasmussen}.) 

First, we have a connected sum formula:
\begin{theorem}
\label{thm:ConnSumSSr}
If $L_1$ and $L_2$ be null-homologous links in $\#^{r_1}(\SSone)$ and $\#^{r_2}(\SSone)$, respectively, each with a single basepoint. Then
$$ s_\pm(L_1 \# L_2) = s_\pm(L_1) + s_\pm(L_2).$$
where $L_1\# L_2$ is the null-homologous link in $\#^{r_1+r_2}(\SSone)$ obtained by connected summing $L_1$ and $L_2$ at the basepoints.
\end{theorem}

For knots in $S^3$, this was proved in \cite[Proposition 3.11]{Rasmussen}. However, even the case of links in $S^3$ appears to be new.

Second, the $s_-$ and $s_+$ invariants give genus bounds for surfaces in the boundary connected sums $\natural^r(S^1 \times B^3)$ and $\natural^r(B^2 \times S^2)$. For a null-homologous link $L \subset \SSr$, we denote by $\gSD(L)$ the minimum genus of a properly embedded surface $\Sigma \subset \natural^r(S^1 \times B^3)$ with $\de \Sigma = L \subset M$. We define $\gDS$ similarly, but using surfaces $\Sigma \subset \natural^r(B^2 \times S^2)$.

\begin{theorem}
\label{thm:genus bounds}
Let $L\neq\varnothing$ be a null-homologous $\ell$-component link in $ \#^r(\SSone)$. Then
\begin{align*}
s_-(L) &\leq 2\gDS(L) + \ell - 1\\
\rotatebox[origin=c]{-90}{$\leq$} \hspace{2ex} & \hspace{10ex} \rotatebox[origin=c]{-90}{$\leq$} \\
s_+(L) &\leq 2\gSD(L) + \ell - 1
\end{align*}
\end{theorem}
Hedden and Raoux~\cite{HR} proved an analogue of Theorem \ref{thm:genus bounds} for the $\tau$ invariants from Heegaard Floer homology.

Finally, let $\xistd$ be the standard (tight) contact structure in $\SSr$. We will use our generalized $s$-invariant to give a combinatorial proof of the slice Bennequin inequality in the case for null-homologous transverse links in $(\SSr, \xistd)$, relative to the Stein filling $\natural^r(S^1 \times B^3)$. This is a particular case of the general slice Bennequin inequality in Stein fillings due to Lisca-Mati\'c \cite{SB2}. Whereas Lisca-Mati\'c used gauge theory, our proof is similar to Shumakovitch's combinatorial proof of the slice-Bennequin inequality in $S^3$ \cite{SB3}, which is based on the $s$-invariant.

\begin{theorem}
\label{thm:sliceBennequin}
Let $L \neq \varnothing$ be a null-homologous transverse link in $(\SSr, \xistd)$. Then
\[
\sl(L)+1 \leq s_+(L).
\]
\end{theorem}

\begin{corollary}[Slice Bennequin inequality in $\SSr$]
\label{cor:sliceBennequin}
Let $L$ be a null-homologous transverse link in $(\SSr, \xistd)$, and let $\Sigma$ be a properly embedded surface in $\natural^r(S^1 \times B^3)$ with no closed components and $\de \Sigma = L$. Then
\[
\sl(L) \leq -\chi(\Sigma).
\]
\end{corollary}

\medskip
{\bf Organization of the paper.} In Section~\ref{sec:Lee complex and s-invt} we define the deformed Khovanov-Lee complexes for links in $\SSr$ and prove Theorems~\ref{thm:deformedKh} and  \ref{thm:computeLee}. In Section~\ref{sec:s_invt} we define the $s$-invariant of links in $\SSr$, and prove Theorems~\ref{thm:s invt well defined}, \ref{thm:fda}, and \ref{thm:GenusBoundCylinders}. In Section~\ref{sec:Hochschild} we give the reinterpretation in terms of Hochschild homology. We use this in Section~\ref{sec:Lpp} to prove Theorems~\ref{thm:s(Fp)} and \ref{thm:sLpp}, about the $s$-invariants of the links $\Fp$ and $\Fp(1)$.
In Section~\ref{sec:adj} we prove Theorem \ref{thm:weakadj}, the adjunction inequality (Corollary \ref{cor:adjunction}), its application to Gluck twists (Corollary~\ref{cor:Glucktwist}). In Section~\ref{sec:sS3} we establish several properties of the invariants $s_+$ and $s_-$ for links in $S^3$; these include Theorems~\ref{thm:addFT} and \ref{thm:sWh+=0}. In Section~\ref{sec:sMr} we give several properties of the $s$-invariants in $\SSr$; this is where we prove Theorems~\ref{thm:ConnSumSSr}, \ref{thm:genus bounds} and \ref{thm:sliceBennequin}, as well as Corollary~\ref{cor:sliceBennequin}. Lastly, in Section~\ref{sec:Problems} we discuss some open problems.

\medskip
{\bf Acknowledgements.} We would like to thank Paolo Ghiggini, Matthew Hedden, Ko Honda, Peter Lambert-Cole, Robert Lipshitz, JungHwan Park, and Katherine Raoux for helpful conversations. We would also like to thank the referees for their helpful comments and suggestions.
\section{Lee's deformation in $\SSr$} \label{sec:Lee complex and s-invt}

\subsection{Notation and conventions}\label{sec:notation}
We will use the following notation throughout this section.
\begin{figure}
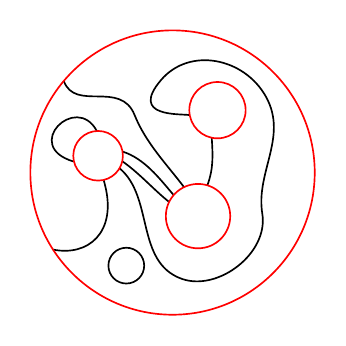
\caption{A planar tangle, giving an operation in a planar algebra. The asterisks denote the basepoints.}
\label{fig:planar}
\end{figure}

\begin{itemize}
\item $R$ will denote a commutative, unital ring with 2 being invertible. All our invariants will depend on $R$ but, to avoid clutter, we will drop it from the notation.
\item Our complexes $C^*$ will have differentials that increase the homological grading by one, so they are cochain complexes. However, by a slight abuse of terminology (which appears in some of the literature on Khovanov homology), we will refer to them as {\em chain complexes}, and talk about their {\em homology}.
\item Recall from  \cite{JonesPlanar} or \cite[Section 5]{BN} that a planar algebra $\Pl$ consists of a collection of $R$-modules $\Pl_n$ for each $n \geq 0$, and multilinear operations
\[
Z_{T}: \Pl_{n_1} \otimes \dots \otimes \Pl_{n_k} \to \Pl_{n_0},
\]
one for each planar tangle $T$ as in Figure~\ref{fig:planar}. Here,
$T$ is a picture consisting of a totally ordered collection of input
disks inside an output disk---with the boundary of each disk
decorated with a basepoint---and non-intersecting strings between
them, with $2n_1, \dots, 2n_k$ endpoints on the input disks, and
$2n_0$ on the output disk. Planar tangles are considered up to
basepoint-preserving isotopy. The operations $Z_T$ are required to
satisfy a composition axiom.
\item $\Kob$ will denote Bar-Natan's planar algebra (p.1465 from
  \cite{BN}). The $R$-module $\Kob_n$ (which is $\Kob(2n)$ in
  Bar-Natan's notation) consists of complexes of (formal linear
  combinations of) crossingless tangles in a disk with $2n$
  endpoints. The morphisms in the complexes are given by matrices of
  (formal linear combinations of) cobordisms between the tangles,
  modulo the relations in \cite[Section 4.1.2]{BN}. The Bar-Natan
  relations admit the following well-known simplification using dotted
  cobordisms and a formal variable $t$: A dot on a surface represents
  a one-handle attached to the surface near the dot, divided by $2$;
  the formal variable $t$ represents a sphere with three dots (that
  is, one-eighth of a genus-three surface); these are illustrated in
  the first row of Figure~\ref{fig:relations}. The original Bar-Natan
  relations then translate to the three relations in the second row of
  Figure~\ref{fig:relations}---the sphere relation, the dotted sphere
  relation, and the neck-cutting relation. A useful consequence of
  these relations is that any two dots on the same connected component
  of a surface may be removed at the cost of multiplying the surface
  by $t$; therefore, we can view our planar algebra to be over $R[t]$.
\item The generators (crossingless diagrams) of a complex $C^*$ in
  $\Kob$ will be denoted as $\delta\in C^*$. For such a diagram, the
  notations $\h_C(\delta)$ and $\q_C(\delta)$ will denote the
  homological and quantum gradings of $\delta$ in $C^*$; the
  subscripts will be omitted if there is no cause for confusion.
\item The notation $\h^a$ and $\q^b$ will indicate upward shifts in homological and $\q$-degree.
\item We define the deformed Temperley-Lieb category $\dTL_{n}$ to be $\Kob_{n}$, when we want to think of the diagrams as being inside a rectangle with $n$ points on the top and $n$ on the bottom. By setting $t=0$ we obtain the category $\TL_n$ of morphisms from $n$ to $n$ in the 2-category $\mathbb{TL}$ from \cite[Section 2.3]{KhTangles}.
\item For a generator $\delta \in \dTL_{n}$, the notation $\thru{\delta}$ will denote the through-degree of the diagram, that is, the number of strands that pass from the top boundary to the bottom boundary.
\item Given an oriented tangle $T$, the notation $\dKC{T}$ will indicate the universal Khovanov complex over $R[t]$ of $T$ in $\Kob$ as constructed by Bar-Natan in \cite{BN}. By specifying $t=0$ in $\dKC{T}$ we get a complex denoted $\KC{T}$, and by specifying $t=1$ we get a complex denoted $\LC{T}$.
\item We use $\BartoKh$ to denote the universal Khovanov functor from
  $\Kob_0$ to the category of complexes of $R[t]$-modules, taking the
  empty diagram to the ground ring $R[t]$. Due to the delooping
  isomorphism \cite{Nao-kh-universal}, $\BartoKh$ sends a diagram
  $\delta$ (with $h(\delta)=q(\delta)=0$) to
  \[
    \bigotimes_{C\in\delta}\big(\q R[t]\oplus\q^{-1}R[t]\big),
  \]
  where the tensor product is over the circles in $\delta$.  We denote
  the generator of $\q R[t]$ by $1$ (denoted $v_+$ in
  \cite{Kh,Rasmussen}) and the generator of $\q^{-1} R[t]$ by $x$
  (denoted $v_-$ in \cite{Kh,Rasmussen}).
\item For a diagram $D$ representing a link $L \subset S^3$, the
  homology of $\OdKC{D}$ will be denoted by $\dKh{L}$. If we set $t=0$ in $\OdKC{D}$ we obtain the usual Khovanov complex $\BartoKh \KC{D}$, whose homology is the Khovanov homology $\Kh{L}$. By setting $t=1$ in $ \OdKC{D}$ we get the Lee complex $\BartoKh \LC{D}$, whose homology is the Lee homology $\Lh{L}$.

\item Let $D$ be a diagram for a link $L \subset S^3$. If $o$ is an
  orientation of $L$, we let $\delta_o \in \dKC{D}$ be the oriented
  resolution of $D$. We let $\s_o \in \OdKC{D}$ be the Lee generator
  associated to $\delta_o$ by the procedure in \cite{Lee}.  That is,
  we assign to each oriented circle $C$ in $\delta_o$ the number of
  circles which enclose $C$.  To this number we add one if $C$ is
  oriented counter-clockwise.  If we call the resulting number $z(C)$,
  then the circle $C$ is assigned the element 
  \[
    g_C:=(-1)^{z(C)}1+x\in \q R[t]\oplus\q^{-1}R[t].
  \]
  The Lee generator $\s_o$ for $\delta_o$ is then the tensor
  product $$\s_o:=\bigotimes_{C\in\delta_o} g_C.$$
\item The notation $\FT_n$ will be used to denote a right-handed full twist braid on $n$ strands.  Multiple full twists are denoted $\FT_n^k$.  We will also use the notation $\infFT_n$ to indicate a formal infinite full twist.
\item A link diagram $D$ for $L \subset \SSr$ is a picture of the link in the standard Kirby diagram for $\SSr$ with $r$ one-handles, as on the left of Figure \ref{fig:InsertTwists}. Given such a diagram $D$ and a vector $\vec{k}\in\Z^r$, the notation $\Lk$ indicates the diagram (for a link in $S^3$) built from $D$ by connecting the $n_i$ endpoints on each attaching sphere and inserting a copy of the full twists $\FT_{n_i}^{k_i}$ as illustrated in Figure \ref{fig:InsertTwists}.   
\end{itemize}

 \begin{figure}
   \centering
   \begin{tikzpicture}[scale=0.7]

     \begin{scope}
       \draw[thick] (0,0) circle (1);
       \begin{scope}[yscale=0.35]
         \draw[thick,dashed] (1,0) arc (0:180:1);
         \draw[thick] (-1,0) arc (180:360:1);
       \end{scope}
       \node at (1.6,0) {$=0$};
     \end{scope}

     \begin{scope}[xshift=5cm]
       \draw[thick] (0,0) circle (1);
       \begin{scope}[yscale=0.35]
         \draw[thick,dashed] (1,0) arc (0:180:1);
         \draw[thick] (-1,0) arc (180:360:1);
       \end{scope}
       \node at (0,0.7) {$\bullet$};
       \node at (1.6,0) {$=1$};
     \end{scope}

     \begin{scope}[xshift=13cm,yshift=4cm]
       \draw[thick] (0,0) circle (1);
       \begin{scope}[yscale=0.35]
         \draw[thick,dashed] (1,0) arc (0:180:1);
         \draw[thick] (-1,0) arc (180:360:1);
       \end{scope}
       \node at (-0.3,0.7) {$\bullet$};
       \node at (0,0.7) {$\bullet$};
       \node at (0.3,0.7) {$\bullet$};
       \node at (-1.6,0) {$t:=$};
     \end{scope}

     \begin{scope}[xshift=10cm,yscale=0.35]
       \foreach \i in {0,3,6}{
         \draw[thick] (\i,5) circle (1);
         \draw[thick,dashed] (1+\i,-5) arc (0:180:1);
         \draw[thick] (-1+\i,-5) arc (180:360:1);
       }
       \foreach \i in {3,6}{
         \draw[thick] (1+\i,-5) to[looseness=7,out=90,in=90] (-1+\i,-5);
         \draw[thick] (1+\i,5) to[looseness=7,out=-90,in=-90] (-1+\i,5);
       }
       \node at (3,2) {$\bullet$};
       \node at (6,-2) {$\bullet$};
       \draw[thick] (-1,-5) --++(0,10);
       \draw[thick] (1,-5) --++(0,10);
       \node at (1.5,0) {$=$};
       \node at (4.5,0) {$+$};
     \end{scope}
       
     \begin{scope}[xshift=2cm,yshift=4cm,yscale=0.8]
       \foreach \i in {0,5}{
         \draw[thick] (-2+\i,-1) -- (1+\i,-1) -- (2+\i,1) -- (-1+\i,1) -- cycle;
       }
       \node at (0,0) {$\bullet$};
       \node at (2.5,0) {$:=\frac{1}{2}$};
       \draw[thick] (5+0.3,0) arc (0:180:0.3);
       \draw[thick] (5+0.6,0) arc (0:180:0.6);

     \end{scope}
     
   \end{tikzpicture}
 \caption{The top row describes the short-hand notations: a dot for a
   one-handle divided by $2$, and $t$ for a triply dotted sphere. The
   bottom row describes the Bar-Natan's relations in terms of dotted
   cobordisms.}
\label{fig:relations}
\end{figure}
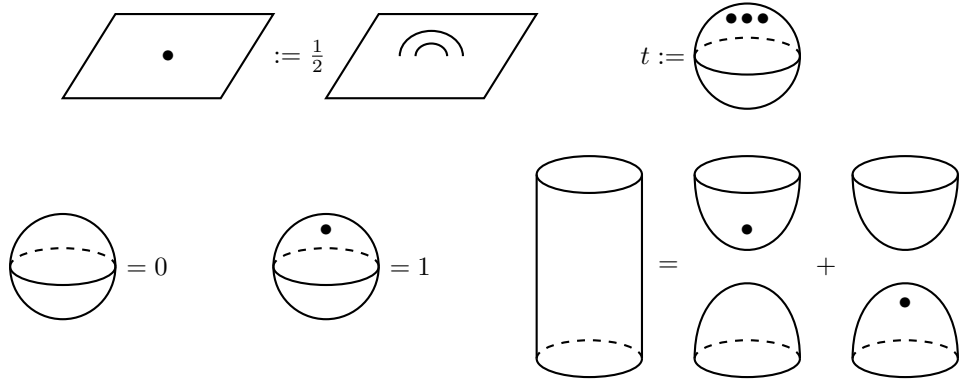

\subsection{The deformed complex and Lee homology}\label{sec:constructing deformed cx}
In \cite{MW}, the fourth author generalized the results of Rozansky in \cite{Roz} to define Khovanov complexes for diagrams $D$ of links $L$ in $M:=\SSr$, provided that $[L] \in H_1(M;\Z)$ is 2-divisible.  The construction goes as follows. We pass from $D$ to link diagrams $\Lk\subset S^3$ as in Figure \ref{fig:InsertTwists}, consider a suitably renormalized version of $\KC{\Lk}$, and take a limit as $k\rightarrow\infty$.  The key point is the construction of a well-defined complex $\KC{\infFT_n}\in\TL_n$ for the infinite full twist $\infFT_n$ on any even number $n=2p$ of strands.  Such complexes can then be inserted into the proper places in the diagram for $L$ using the planar algebra operations. This gives a complex $\KC{D}$. If $D_1$ and $D_2$ are two diagrams for the same link $L \subset M$, then the complexes $\OKC{D_1}$ and $\OKC{D_2}$ have isomorphic homologies.

Now for a diagram $D$ of a link $L$ in $S^3$, the Lee homology
$\Lh{L}$ and the spectral sequence relating it with $\Kh{L}$ (as well
as the $s$-invariants for $L$ when the ground ring is a field) are
determined by constructing the deformed Khovanov complex $\dKC{D}$.
Thus the main step towards defining a Lee homology and $s$-invariant
for links in $M$ is to construct deformed Khovanov complexes $\dKC{D}$
for their diagrams. In turn, this relies on constructing a complex
$\dKC{\infFT_n}$ for the infinite twist in $\dTL_n$ over $R[t]$.

\begin{theorem}\label{thm:KC'(infFT)}
Fix an even integer $n=2p$.  Let $\FT_{n,o}$ denote the full twist $\FT_n$ equipped with an orientation $o$, and let $\eta_o$ denote the difference between the number of upward and downward pointing strands in $\FT_{n,o}$.
\begin{itemize}
\item For any $k>0$, there exists a complex $\CSharp(\FT_n^k)\in\dTL_n$ (independent of orientations), supported in non-positive homological grading, that is chain homotopy equivalent to a shifted Khovanov complex
\[\CSharp(\FT_n^k) \simeq \dKCrrn(\FT_{n,o}^k) := \h^{-\frac{k}{2}\eta_o^2} \q^{-\frac{3k}{2}\eta_o^2} \dKC{\FT_{n,o}^k}.\]
Furthermore, all diagrams $\delta$ in the truncated complex $\CSharp(\FT_n^k)_{\geq 1-2k}$ are split $(\thru{\delta}=0)$ and contain no disjoint circles.
\item There exists a well-defined semi-infinite complex $\dKC{\FT_n^\infty}$ (independent of orientations) that satisfies, for any fixed homological degree $d\leq 0$,
\[\dKC{\FT_n^\infty}_{\geq d} = \CSharp(\FT_n^k)_{\geq d}\]
for all $k \geq \intceil{\frac{1-d}2}$.  In particular, all diagrams $\delta\in\dKC{\infFT_n}$ satisfy $\h(\delta)\leq 0$ and $\thru{\delta}=0$.
\end{itemize}
\end{theorem}
\begin{proof}
The only difference between $\TL_n$ and $\dTL_n$ is the value assigned to two dots on a connected component of a cobordism.  A careful reading of the construction in $\TL_n$ in \cite[Section 2]{MW} or \cite[Section 8]{Roz} shows that this value is never used, and so the exact same construction works here.  We provide a short summary below.

The complexes $\CSharp(\FT_n^k)$ are obtained from $\dKCrrn(\FT_{n,o}^k)$ by a sequence of multi-cone simplifications utilizing delooping isomorphisms \cite{Nao-kh-universal}  and crossing removing Reidemeister I and II moves.  Note that such Reidemeister I and II moves keep producing very strong deformation retracts \cite[Definition 2.9]{MW} in $\dTL_n$ (closed spheres are still zero), while Naot's delooping isomorphism also holds true in $\dTL_n$ (closed spheres with two dots are still zero).  The complex $\dKC{\FT_n^\infty}$ is built as a limiting complex of the $\CSharp(\FT_n^k)$; as $k\rightarrow\infty$, the complexes stabilize, allowing for truncations to be approximated via complexes associated to finite twists.  The details are worked out in \cite[Section 2]{MW}.  The grading shifts indicated there are easily shown to match the ones here with a careful count of positive and negative crossings within a full twist.
\end{proof}

\begin{corollary}\label{cor:KC'(L) finite approx}
Let $D$ be a diagram for an oriented link $L\subset M=\SSr$ such that $[L]\in H_1(M;\Z)$ is 2-divisible.  Let $n^+_D$ denote the number of positive crossings in $D$.  For any $k>0$, let $\vec{k}=(k,\dots,k)$ and define $\CSharp(\Lk)\in\Kob$ to be the complex resulting from inserting $\CSharp(\FT_{n_i}^k)$ in place of each $\FT_{n_i}^k$ in the diagram for $\Lk$ (see Figure \ref{fig:InsertTwists}). Then:
\begin{itemize}
\item The complex $\CSharp(\Lk)\in\Kob$ is chain homotopy equivalent to a shifted Khovanov complex
\[\CSharp(\Lk)\simeq \dKCrrn(\Lk) := \h^{-\sum \frac{k}{2}\eta_i^2} \q^{-\sum\frac{3k}{2}\eta_i^2} \dKC{\Lk}\]
where the sums are taken over all attaching spheres, with each $\eta_i$ denoting the algebraic intersection number of $L$ with the corresponding sphere.
\item There exists a well-defined semi-infinite complex $\dKC{D}\in\Kob$ that satisfies, for any fixed homological degree $d$,
\[\dKC{D}_{\geq d} = \CSharp(\Lk)\]
for any $\vec{k}=(k,\dots,k)\in\Z^r$ with
\[k\geq\intceil{\frac{n^+_D+1-d}{2}}.\]
\end{itemize}
In particular, the homology
\[
\dKh{D} := H^*(\OdKC{D})
\]
can be computed in any finite homological degree $d$ by computing the (shifted) Khovanov homology of a corresponding link $\Lk$ for $k \geq \intceil{\frac{n^+_D+1-d}{2}}$.
\end{corollary}
\begin{proof}
This is the construction of \cite[Section 3]{MW}, applied now to the deformed complexes over $R[t]$.  As $k\rightarrow\infty$, we limit towards inserting $\dKC{\FT_{n_i}^\infty}$ into each spot, building the semi-infinite complex $\dKC{D}$.  The shifts are taking into account all of the full twists that have been inserted.  The minimal value for $k$ is determined by noting that the right-most homological grading available when we consider only crossings in $\Lzero$ (ignoring the full twists) is $n^+_D$.  The limiting procedure for Theorem \ref{thm:KC'(infFT)} fixes right-most degrees as homological degree zero in the (properly simplified and shifted) complexes for the full twists.  See Figure \ref{fig:basic stab diag} for an illustration.
\end{proof}

\begin{remark}\label{rmk:no shifts for nullhomologous}
Note that if $L$ is nullhomologous in $\SSr$, then all of the shifts in Corollary \ref{cor:KC'(L) finite approx} are zero and there is no need to renormalize any of the finite approximation complexes $\dKC{\Lk}$ when studying $\dKh{D}$.
\end{remark}

\begin{figure}
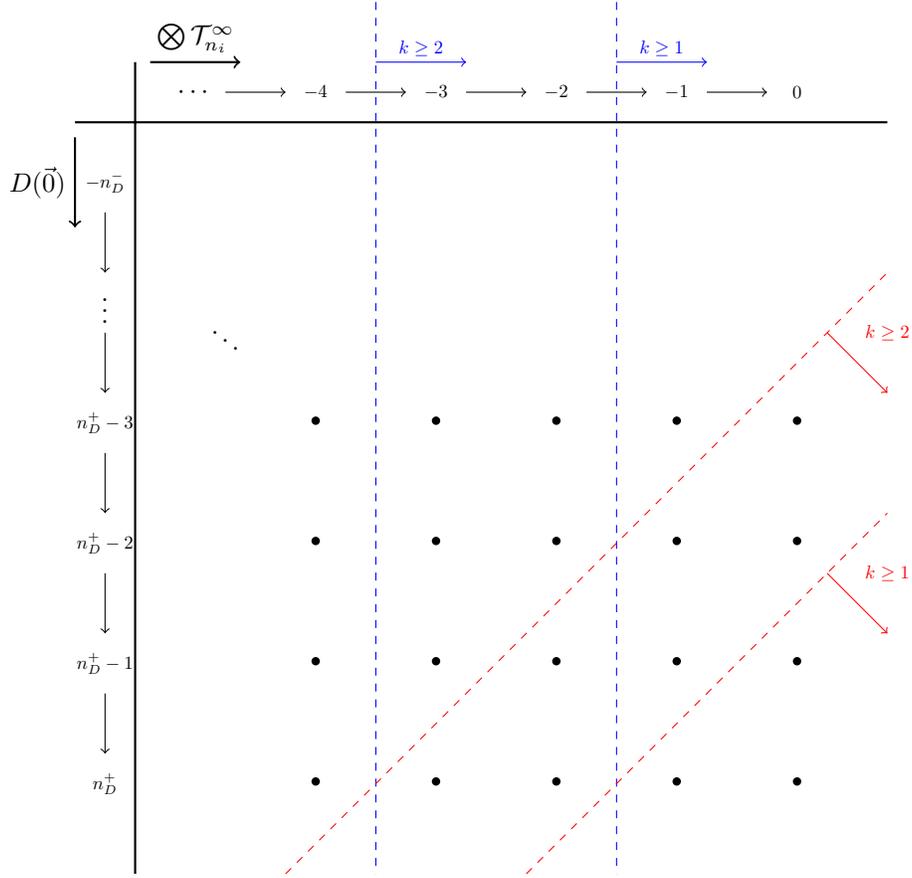

\BasicStabDiag
\caption{The diagram above represents $\dKC{D}$ as a bicomplex, with the horizontal direction representing the differentials coming from the infinite twists, and the vertical direction representing the differentials coming from the other crossings of $D$ (i.e., the crossings present in $\Lzero$).  The vertical dashed lines indicate how the finite full twists $\FT^k$ stabilize for various values of $k$ according to Theorem \ref{thm:KC'(infFT)}, which impose the stabilization for the full complex of $\Lk$ indicated by the diagonal dashed lines.}
\label{fig:basic stab diag}
\end{figure}

{
\renewcommand{\thethm}{\ref{thm:deformedKh}}
\begin{theorem}
Let $D_1,D_2$ be two diagrams for an oriented link $L\subset M=\SSr$ such that $[L] \in H_1(M;\Z)$ is 2-divisible.  Then the homologies $\dKh{D_1}$ and $\dKh{D_2}$ are isomorphic up to grading shifts. Furthermore, if $[L]=0\in H_*(M;\Z)$, then these gradings shifts are zero, so the deformed Khovanov-Lee homology $\dKh{L}$ is a well-defined link invariant up to isomorphism, as a bi-graded module over $R[t]$. 
\end{theorem}
\addtocounter{thm}{-1}
}
\begin{proof}
Once more, a careful reading of the invariance proof from \cite[Section 3.3]{MW} reveals that the value of two dots is never used.  The symbols $\eta_i$ appearing in the degree shifts in that paper represent the algebraic intersection numbers of $L$ with the various attaching spheres. When $[L] =0$ in $H_*(M;\Z)$, such grading shifts vanish.
\end{proof}

\begin{remark}
\label{rmk:deformedKh unnatural}
The isomorphisms of Theorem \ref{thm:deformedKh} might depend on a choice of finite approximations (which in turn determine finite approximation diagrams $D_1(k),D_2(k)$).  See also Remark \ref{rmk:Lee may be unnatural} and Section \ref{sec:naturality}.
\end{remark}

In view of Theorem~\ref{thm:deformedKh}, when $[L]$ is $2$-divisible, we can use the notation 
\[
\dKh{L}:=\dKh{D},
\]
where $D$ is any diagram for $L$. When $[L] \neq 0$, we do this with the understanding that $\dKh{L}$ is only well-defined up to isomorphism and possible grading shift.

The Lee complex $\LC{D}$ is defined by taking $\dKC{D}$ and setting
the variable $t=1$ in the ground ring $R[t]$ as usual.  Similarly, for diagrams of the form $\Lk$ we have a shifted Lee complex 
\[\LCrrn{\Lk}:= \h^{-\sum \frac{k}{2}\eta_i^2} \q^{-\sum\frac{3k}{2}\eta_i^2} \LC{\Lk},\]
and a simplified complex $\CSharpLee(\Lk)\simeq \LCrrnnobrackets(\Lk)$ as in Corollary \ref{cor:KC'(L) finite approx}.  The Lee homology is denoted 
\[ 
  \Lh{D} := H^*(\OLC{D}).
\]
The same proof as that of Theorem \ref{thm:deformedKh} shows that $\Lh{D}$ is independent of the diagram up to isomorphism. This justifies the notation $\Lh{L}$, which we will use from now on.

The following result subsumes Theorem~\ref{thm:computeLee} announced in the Introduction.

\begin{theorem}\label{thm:Lee is copies of Q}
Suppose our ground ring $R$ is a field $\F$ of characteristic not equal to two.  Let $D$ be a diagram for an oriented link $L$ in $M=\SSr$ with $n^+_D$ positive crossings, such that $[L] \in H_1(M;\Z)$ is 2-divisible. Let $O(L)$ denote the set of orientations $o$ of $L$ for which the re-oriented link $L_o$ (having diagram $D_o$) would be null-homologous.  Let $a:=\max_{o\in O(L)}(n^+_{D_o})$.  Then
\[
\Lh{D}\cong\F^{|O(L)|}\cong {H^*\bigl(\OLCrrnnobrackets(\Lk)\bigr)}_{\geq n^+_D-a} 
\]
for $\vec{k}=(k,\dots,k)$ with $k\geq\intceil{\frac{a+2}{2}}$.
\end{theorem}

\begin{proof}
According to Corollary \ref{cor:KC'(L) finite approx}, $\dKh{L}$ can be approximated in finite homological range by the shifted complex $\dKCrrn(\Lk)$ for some $\vec{k}\in\Z^r$; since the Lee complex is determined by setting $t=1$ in $\dKC{D}$, the same statement holds for $\Lh{L}$.  Since $\Lk$ is a diagram for a link in $S^3$, Lee's results \cite{Lee} show that $\Lh{\Lk}\cong\F^{b}$, where we have one copy of $\F$ for each possible orientation of $\Lk$.  The question becomes which of these orientations give rise to copies of $\F$ that are within the stable homological range.  Noting that an orientation of $\Lk$ is equivalent to an orientation for $L$, we claim that the only surviving copies of $\F$ are those for which the correspondingly oriented $L$ was null-homologous in $M$.

To show this, we fix an orientation $o$ of $L$ and a single attaching sphere $S$, and let $n=2p$ and $\eta$ be the geometric and algebraic intersection numbers of the oriented $L$ with the chosen sphere $S$, respectively. For any $\vec{k}$, Lee's arguments find a generator $[\s_o]$ of $\Lh{\Lk}$ coming from the oriented resolution of $\Lk$. When restricted to the full twists corresponding to $S$, this oriented resolution gives a diagram $\delta_{o, S} \in\LCrrnnobrackets(\FT_n^k)$, where $k$ is the entry of $\vec{k}$ for the chosen sphere.  Since the oriented resolution of a tangle occurs in homological degree zero,
after shifts we see
%
\[\h_{\LCrrn{\FT_n^k}}(\delta_{o, S}) = -\frac{k}{2}\eta^2.\]

Meanwhile, the complex $\LCrrnnobrackets(\FT_n^k)$ approximates $\LC{\infFT_n}$ only in homological degrees strictly greater than $1-2k$.  From this, some simple algebra reveals that $k$ was irrelevant, and that $\delta_{o, S}$ could be found in the stable range if and only if $|\eta|<2$, or equivalently $\eta=0$ and $\h_{\LCrrnnobrackets(\FT_n^k)}(\delta_{o, S})=0$.

Thus we see that, for any orientation $o\notin O(L)$, the Lee generator $[\s_o]$ of $\Lh{\Lk}$ will never sit in the stable homological range, being `pushed out towards $-\infty$' in the limit as our twists become infinite.  Meanwhile, for any orientation $o\in O(L)$, the Lee generator $[\s_o]$ has corresponding diagrams $\delta_{o, S}$ sitting in homological grading zero
in each of the renormalized complexes $\LCrrnnobrackets(\FT_n^k)$ for the twists, regardless of $n,k$, or the original orientation of $L$.  When these complexes are stitched together in our planar algebra with the finite complex coming from the crossings in $D$, the oriented resolution generator $\delta_o\in\wLCnobrackets(\Lk)$ will have 
\[
\h_{\LCrrn{\Lk}}(\delta_o)=n^+_D - n^+_{D_o}.
\]
This is in the stable range for computing homology as soon as the complex has stabilized through homological degree $n^+_D - n^+_{D_o} - 1$; if we wish to find the smallest $k$ that guarantees stability for all such $o\in O(L)$, we define $a$ as above and use the bound in Corollary \ref{cor:KC'(L) finite approx}.
\end{proof}

\begin{remark}
\label{rmk:Lee may be unnatural}
The proof above requires a choice of finite approximation $k$ to provide the first isomorphism of Theorem \ref{thm:Lee is copies of Q}; this is similar to the situation in Theorem \ref{thm:deformedKh} and Remark \ref{rmk:deformedKh unnatural}. However, we will show that this isomorphism is independent of $k$ in Section \ref{sec:naturality}.
\end{remark}

Finally, we record a simple property of our homology theories.  Given an oriented link $L\subset\SSr$, we let $r(L)$ denote the \emph{reverse} of $L$, obtained from $L$ by reversing the orientation on each component.

\begin{proposition}\label{prop:Kh(r(L))}
For any $L\subset\SSr$, we have
\[\dKh{r(L)} \cong \dKh{L},\quad \Lh{r(L)}\cong\Lh{L}.\]
\end{proposition}
\begin{proof}
Both homology theories are built combinatorially from a link diagram without regard for orientation, except for the counts of negative and positive crossings which are preserved under reversal.
\end{proof}


\subsection{Behavior under diffeomorphisms}
Let $\phi_1: S^1 \to S^1$ be the conjugation (reflection across the horizontal axis), and $\phi_2: S^2 \to S^2$ be reflection across the equatorial plane. Consider the following diffeomorphisms of $\SSone$:
\begin{itemize}
\item the Dehn twist along some $\{*\} \times S^2$, denoted $\sigma$;
\item the composition $\rho = (\phi_1 \times \id) \circ (\id \times \phi_2)$. In the standard surgery diagram, this is isotopic to rotation by $\pi$ about an axis perpendicular to the plane of the projection;
\item the reflection $R = \id \times \phi_2$. In the standard surgery diagram, this corresponds to a reflection fixing the plane of the projection. 
\end{itemize}
For $M=\SSone$, it is well-known that the group $\pi_0(\Diff^+(M))$ of orientation-preserving diffeomorphisms (up to isotopy) is $\Z_2\times\Z_2$, generated by $\sigma$ and $\rho$. If we include all (not necessarily orientation preserving) diffeomorphisms, the group $\pi_0(\Diff(M))$ is $\Z_2\times\Z_2\times \Z_2$, generated by $\sigma, \rho,$ and $R$.

More generally, for $M=\SSr$, the mapping class group $\pi_0(\Diff^+(M))$ can be understood using the methods in \cite[Section 2]{HW}. The natural map $\pi_0(\Diff^+(M)) \to \operatorname{Aut}(\pi_1(M, x))$ is surjective, and its kernel is generated by Dehn twists along two-spheres. Furthermore, $\pi_1(M, x)$ is the free group on $r$ generators, and its automorphism group is generated by Nielsen transformations \cite{Nielsen}. We conclude that $\pi_0(\Diff^+(M))$ is generated by the following:
\begin{itemize}
\item the Dehn twist $\sigma_i$ in one of the summands $(\SSone)_i$;
\item the rotation $\rho_i$ in one of the summands $(\SSone)_i$, obtained as follows. Since $\SSone$ is connected, we can isotope $\rho$ to fix a point $x$, and in fact to fix pointwise a ball $B$ around $x$. We let $B$ be the ball where we connect $(\SSone)_i$ to the rest of the manifold, and extend the diffeomorphism by the identity to the other summands;
\item for every permutation of the summands $\SSone$, a diffeomorphism inducing that permutation;
\item handle slides, i.e. viewing $\SSr$ as the boundary of a four-dimensional handlebody made of a zero-handle and $r$ one-handles, we slide a one-handle over another. 
\end{itemize}
The group $\pi_0(\Diff(M))$ is generated by the above, together with the orientation-reversing diffeomorphism $m$ (for ``mirror'') induced by reflecting via $R$ in each summand.  To get a diagram for $m(L)$ in our standard picture of $\SSr$, we change all the crossings in a standard diagram for $L$.  We will also use the notation $-L := r(m(L))$ for the mirror reverse of $L$.

\begin{theorem}\label{thm:dKC preserved by even diffeos}
Let $L\subset M=\SSr$ be a link that is 2-divisible in homology.  Then $\dKh{L}$ and $\Lh{L}$ are preserved under orientation-preserving self-diffeomorphisms up to grading shifts that vanish if $[L]=0$ in $H_1(M)$.  That is, for any $\Phi\in\pi_0(\Diff^+(M))$ of $M$, we have
\[ \dKh{\Phi(L)} \cong \h^a\q^b \dKh{L}, \quad \Lh{\Phi(L)} \cong \h^a\q^b \Lh{L}\]
with both $a,b=0$ in the case that $L$ was null-homologous in $M$.
\end{theorem}
\begin{proof}
Let $D$ denote a diagram for $L$ and let $\vec{k}=(k,\dots,k)$ for $k$ sufficiently large.
With the help of Corollary \ref{cor:KC'(L) finite approx}, it is enough to check how the generators of $\pi_0(\Diff^+(M))$ affect the homology of the finite approximation diagrams $\Lk$.

In the case of a rotational generator $\Phi=\rho_i$, or a transposition $\Phi=\tau_{i,i+1}$ (which is enough to generate all permutations), we have an obvious isotopy in $S^3$ between the diagrams of the finite approximations of $L$ and $\Phi(L)$ as illustrated in Figures \ref{fig:rot} and \ref{fig:trans}.
Meanwhile, a Dehn twist $\sigma_i$ simply adds a full twist onto the link diagram near the attaching sphere, effectively changing $k$ to $k\pm 1$.  The stabilization of the \emph{shifted} homology of $\Lk$ shows that this has no effect besides a possible grading shift, which is zero when the algebraic intersection number of $L$ with the $i^{\text{th}}$ attaching sphere is zero.

\begin{figure}
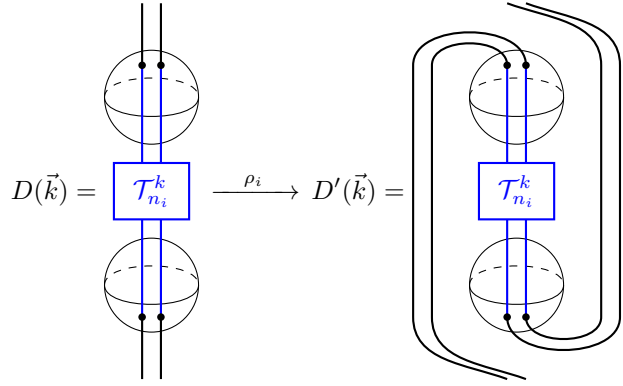

\[\Lk=\LkPreRot \, \xrightarrow{\quad\rho_i\quad} \, \Lk[D']=\LkPostRot\]
\caption{If $\Lk$ above is the finite approximation diagram computing $\dKh{L}$ in some degree, then $\Lk[D']$ is the finite approximation diagram computing $\dKh{\Phi(L)}$ for a rotation $\Phi=\rho_i$.  Note that the rest of the link can be isotoped away from this local picture.  There is a clear isotopy in $S^3$ relating $\Lk$ and $\Lk[D']$, since any full twist diagram is preserved under a rotation by $\pi$.
In this and other figures in this subsection we have drawn the surgery spheres, but these are not meant to be identified, i.e.~the diagrams are for links in $S^3$.
}
\label{fig:rot}
\end{figure}

\begin{figure}
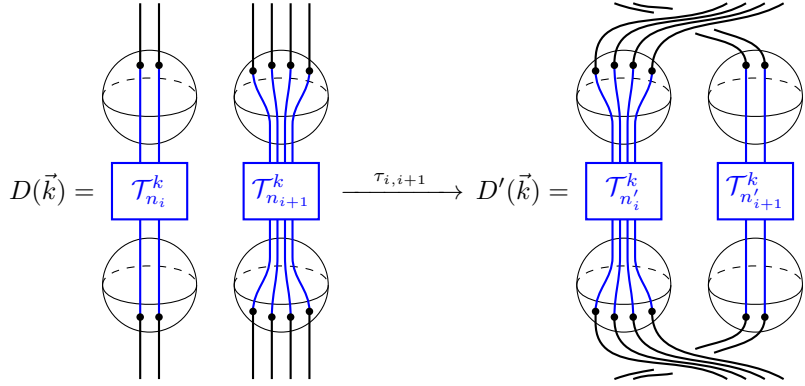

\[\Lk=\LkPreTrans \,\xrightarrow{\quad\tau_{i,i+1}\quad}\, \Lk[D']=\LkPostTrans\]
\caption{If $\Lk$ above is the finite approximation diagram computing $\dKh{L}$ in some degree, then $\Lk[D']$ is the finite approximation diagram computing $\dKh{\Phi(L)}$ for a transposition $\Phi=\tau_{i,i+1}$ (again, the rest of the link has been isotoped away from this local picture).  Note that $\tau_{i,i+1}$ has interchanged the roles of $i$ and $i+1$ in the indexing of the handles, so that $n'_i=n_{i+1}$ and vice-versa.  There is a clear isotopy in $S^3$ relating $\Lk$ and $\Lk[D']$.}
\label{fig:trans}
\end{figure}

Finally, we have the case where $\Phi$ denotes a handle slide.  Up to a permutation, we can assume the handle slide is between the first and second handles in our standard diagram, and we can begin with an isotopy moving other strands of $L$ away from the path of the handle slide.  In this way we can illustrate $\Phi$ as in Figure \ref{fig:hslide1}.

\begin{figure}
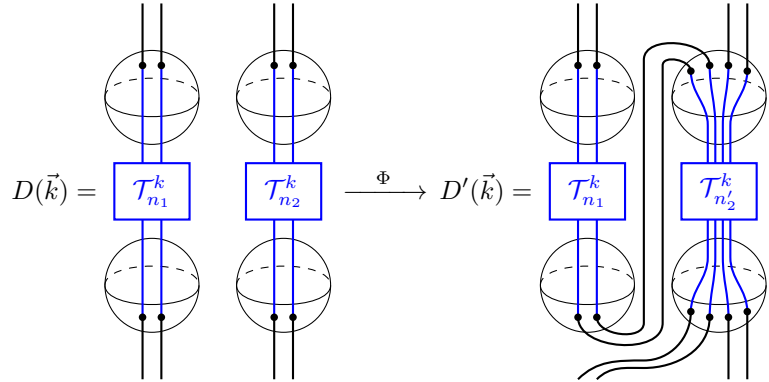

\[\Lk=\LkPreHslide \, \xrightarrow{\quad \Phi \quad} \, \Lk[D']=\LkPostHslide\]
\caption{If $\Lk$ above is the finite approximation diagram computing $\dKh{L}$ in some degree, then $\Lk[D']$ is the finite approximation diagram computing $\dKh{\Phi(L)}$ for a handle slide $\Phi$ (again, the rest of the link has been isotoped away from this local picture).  We see that $n_1$ has been preserved but $n'_2=n_1+n_2$.}
\label{fig:hslide1}
\end{figure}

From this point forward our goal is to manipulate the complex $\dKC{\Lk[D']}$ in order to show that its truncation has the same homology as the corresponding truncation of $\dKC{\Lk}$. 
On the left, we replace $\dKC{\FT_{n_1}^k}$ with the chain homotopic simplified complex $\CSharp(\FT_{n_1}^k)$.  Meanwhile on the right, we use the isotopy

\[
\begin{tikzpicture}[baseline={([yshift=-.7ex]current bounding box.center)},x=1.5em,y=-2em]
 \foreach \i in {1,...,4} 
  {
     \draw[thick] 
         (\i-1,0) -- (\i-1,0+.1)
         (\i-1,0+2) -- (\i-1,0+1.9);
  }
 \draw (1-1-.1,0+.1) rectangle (4-1+.1,0+1.9);
 \node at (2.5-1,0+1){$\FT_{n_1+n_2}^k$};
\end{tikzpicture}
\cong
\begin{tikzpicture}[baseline={([yshift=-.7ex]current bounding box.center)},x=1.5em,y=-2em]
\BboxOnly[0]{1}{2}{$\FT_{n_1}^k$}
\BboxOnly[0]{3}{4}{$\FT_{n_2}^k$}
\Bbox[1]{4}{1}{4}{$c_{12}\mhyphen\FT_2^k$}
\end{tikzpicture},
\]
where $c_{12}\mhyphen\FT_2^k$ denotes $k$ full twists on two strands which are then cabled (with blackboard framing) with $n_1$ and $n_2$ strands, respectively.  Finally, we isotope the copy of $\FT_{n_2}^k$ further away from the local picture and focus on the deformed Khovanov complex assigned to the rest of the diagram.  See Figure \ref{fig:hslide2}. 
The local picture (ignoring the far right $\FT_{n_2}^k$) will be denoted $D''$. The resulting complex for the local picture will be denoted $\CSharpSub{1}(D'')$, where the subscript $1$ indicates that we're using the $\CSharp$ complex for $\FT_{n_1}^k$ only.

\begin{figure}
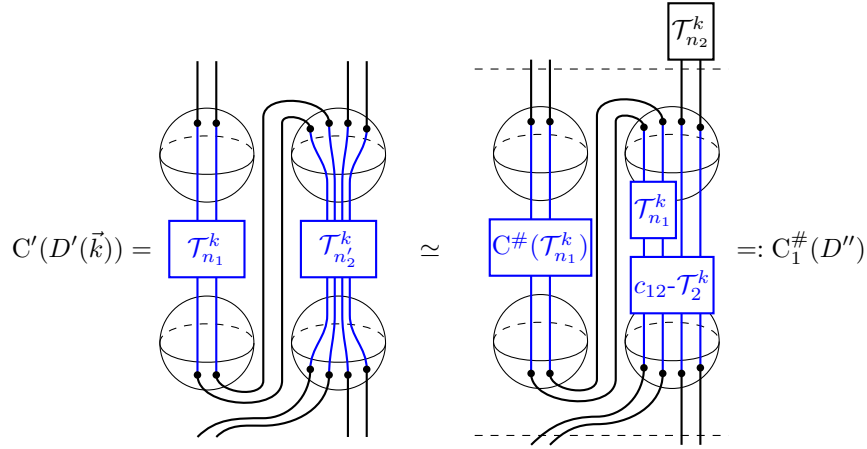

\[\dKC{\Lk[D']} = \LkPostHslide \quad\simeq\quad \LkPostPostHslide =: \CSharpSub{1}(D'') \]
\caption{We simplify the complex of $\dKC{\Lk[D']}$, obtained from Figure \ref{fig:hslide1}, by simplifying $\dKC{\FT_{n_1}^k}$ into $\CSharp(\FT_{n_1}^k)$ and performing some isotopies on $\FT_{n_1+n_2}^k$.  Abusing notation slightly, we omit the $\dKC{\cdot}$ notation in the pictures while allowing $\CSharp(\cdot)$ to indicate that part of the complex has been simplified.  The twists $\FT_{n_2}^k$ are drawn further out to indicate that, from here, we will simplify the complex assigned to the diagram within the dotted lines which we denote $\CSharpSub{1}(D'')$.}
\label{fig:hslide2}
\end{figure}

Now $\CSharpSub{1}(D'')$ is a complex in $\TL_{n_1+n_2}$ which can be used to compute $\dKh{\Phi(L)}$ in certain homological gradings.  By construction (see \cite[Section 2]{MW}), the complex $\CSharp(\FT_{n_1})$ appearing within Figure \ref{fig:hslide2} contributes only split diagrams $\delta\in\TL_{n_1}$ with $\thru{\delta}=0$ to the relevant homological gradings.  Thus we can view the truncation $\CSharpSub{1}(D'')_{tr}$ as a multicone over the truncated $\CSharp(\FT_{n_1}^k)$ where every term in the multicone involves a diagram with a split $\delta$ in the place of $\CSharp(\FT_{n_1}^k)$ in Figure \ref{fig:hslide2}.

However, any such diagram can clearly be simplified; namely, $\delta$
is split into a top portion and a bottom portion, say $\delta^{top}$
and $\delta^{bot}$, and the bottom portion $\delta^{bot}$ can be carried
along the path of the handle slide, unwinding both $\FT_{n_1}^k$ and
$(c_{12}\mhyphen\FT_2)^k$ (using only crossing-removing Reidemeister I
and II moves).  The resulting diagram has \emph{no crossings
  remaining}, as illustrated in Figure \ref{fig:hslide3}.

\begin{figure}
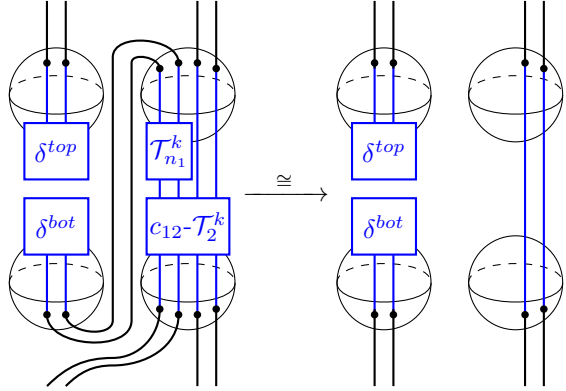

\[\LkPostPostHslideDelta \xrightarrow{\quad \cong \quad} \LkDelta\]
\caption{Every diagram $\delta$ in the truncated complex $\CSharp(\FT_{n_1}^k)$ is split into two halves, say $\delta^{top}$ and $\delta^{bot}$.  This $\delta$ contributes a term in the multicone decomposition for $\CSharpSub{1}(D'')_{tr}$ according to the diagram on the left, which can be simplified into the diagram on the right having no remaining crossings.}
\label{fig:hslide3}
\end{figure}


The goal now is to use the simplification of Figure \ref{fig:hslide3} on each individual term in the multicone $\CSharpSub{1}(D'')_{tr}$, while simultaneously leaving the multicone maps (dotted cobordisms) between such terms unchanged.  The general statement allowing such a simplification is contained in \cite[Corollary 2.14]{MW}.  In short, if our dotted cobordisms between terms commute with our intended Reidemeister simplifications up to isotopy, then Bar-Natan's functoriality arguments guarantee that we can perform the simplification up to possible sign changes in the multicone maps.  If in addition we can assign signs to our intended Reidemeister simplifications in a suitable manner, the sign changes can all be cancelled.

Because our diagrams are split with no disjoint circles, it is not hard to see that our Reidemeister simplifications commute with any of our dotted cobordisms up to isotopy as desired (see also \cite[Lemma 2.26]{MW}).  The central idea behind fixing the signs stems from Rozansky \cite[Section 7.3]{Roz} and is expanded upon via several lemmas in \cite{MW}.  First, we fix a constant ``bottom'' set of crossingless matchings 
\[\gamma:=
\begin{tikzpicture}[baseline={([yshift=-.7ex]current bounding box.center)},x=.25cm,y=-.22cm]
\foreach \n in {0,2,5}
{
\draw[thick]
(\n,0) -- (\n,1) to[out=-90,in=-90] (\n+1, 1) -- (\n+1,0);
}
\draw
(4.1,0.7) node[scale=.7] {\dots};
\draw[decorate,decoration={brace}]
(-.3,-.3)--(6.3,-.3) node [midway,above] {$n_1$};
\end{tikzpicture}
\]
to append to the bottom left part of $D''$.  Having this bottom set of matchings gives us two choices on how to perform the Reidemeister simplifications---we may either ``pull $\delta^{bot}$ downwards'' or ``pull $\gamma$ upwards.''  Both choices involve untwisting the $n_1$ left-most strands, together with unwrapping them about the $n_2$ right-most strands.  The difference between these two choices provides the desired signs that will cancel the sign changes coming from our initial simplifications.  The details for verifying this cancellation are described in the proofs of \cite[Lemmas 2.25, 2.26]{MW} for the untwisting, together with \cite[Lemmas 3.11, 3.12]{MW} for the unwrapping (note that in the unwrapping case, the two choices are only isotopic up to ``looping'' disjoint circles around the $n_2$ strands; fortunately, such loopings induce identity maps on our complexes).

Thus according to \cite[Corollary 2.14]{MW}, the entire truncated complex $\CSharpSub{1}(D'')_{tr}$ is homotopy equivalent to the truncated $\CSharp(\FT_{n_1}^k)$ together with $n_2$ extra strands on the right.  When we stitch this complex back together with the rest of the diagram, we will recover the $\FT_{n_2}^k$ that we left waiting ``above'' our complex and thus will build a complex homotopy equivalent to a truncation of $\dKC{\Lk[D']}$, as desired (up to degree shifts which vanish for null-homologous links; compare the Finger Move in \cite[Section 3]{MW}).

Finally, all of these manipulations were done in the deformed complex without regard for the value of $t$, and thus will work just as well for the Lee homology as well.
\end{proof}

We note here that Theorem \ref{thm:dKC preserved by even diffeos} is not sufficient to conclude that $\pi_0(\Diff^+(M))$ acts on $\dKh{L}$ in a well-defined manner, due once again to the choices of finite approximations $k$ required throughout.

\subsection{Lee generators in $\OLC{\infFT_n}$ and $\OLC{D}$}
\label{sec:naturality}

In this section we will show that, for any diagram $D$ of a link $L\subset\SSr$ (that is 2-divisible in $H_1(\SSr)$), the isomorphism $\Lh{D}\cong\F^{|O(L)|}$ of Theorem \ref{thm:Lee is copies of Q} is natural.

Recall that $O(L)$ is defined to be the set of orientations for which $L$ would be null-homologous in $\SSr$.  We will actually show a stronger and more crucial statement: that for any orientation $o\in O(L)$, the summand $\F_o$ of $\Lh{D}\cong\F^{|O(L)|}$ is generated by a specific cycle $\s_o\in\LC{D}$ which is determined by an oriented diagram $\delta_o$ corresponding to an oriented resolution of $D$.  We refer to such cycles $\s_o$ as Lee generators.

Furthermore, for any finite approximation $\LCrrnnobrackets(\Lk)$ used to compute this homology in $S^3$, the summand $\F_o=\langle \s_o\rangle$ of $\Lh{D}$ corresponds to the summand $\F_{o,k}=\langle \s_{o,k} \rangle$ of $\Lh{\Lk}$, where $\s_{o,k}$ is the Lee generator coming from the actual oriented resolution of $\delta_{o,k}$ of $\Lk$.

Since $\Lh{D}$ is built by inserting limiting complexes of infinite twists into a diagram, we will show that the Lee generators are `preserved' in a certain sense by the types of simplifications used during the limiting process.  Because we will be simplifying multicones, we need to keep track not only of the chain maps involved in simplifying a local picture, but also the homotopies that make these chain maps into equivalences.

\begin{lemma}\label{lem:simps on Lee}
Suppose $D$ is a link diagram (for a link in $S^3$) with orientation $o$, and corresponding oriented resolution $\delta_o$ and Lee generator $\s_o$.  Let $\psi:\LC{D}\rightarrow \psi(\LC{D})$ be a chain homotopy equivalence (with chain homotopy inverse $\td\psi$) induced by one of three types of local simplifications:
\begin{enumerate}[label=(\Roman*)]
\item \label{it:R1} a Reidemeister I move which eliminates a negative crossing;
\item \label{it:R2} a crossing-removing Reidemeister II move;
\item \label{it:FT2} a simplification of the full twist on two strands, as indicated (ignoring grading shifts)
\[
\LCp{
\begin{tikzpicture}[baseline={([yshift=-.7ex]current bounding box.center)},x=.33cm,y=-.4cm]
  \Bsigma[0]{2}{1}
  \Bsigma[1]{2}{1}
 \end{tikzpicture}
}
\simeq
\left(
 \begin{tikzpicture}[baseline={([yshift=-.7ex]current bounding box.center)},x=.33cm,y=-.4cm]
  \Bsigma[0]{2}{0}
  \Bsigma[1]{2}{0}
 \end{tikzpicture}
\longrightarrow
 \begin{tikzpicture}[baseline={([yshift=-.7ex]current bounding box.center)},
  x=.33cm,y=-.4cm]
  \Bcupcap[0.5]{2}{1}
  \draw[thick]
  (0,0)--(0,0.5)
  (0,1.5) -- (0,2)
  (1,0)--(1,0.5)
  (1,1.5)--(1,2);
  \end{tikzpicture}
\longrightarrow
 \begin{tikzpicture}[baseline={([yshift=-.7ex]current bounding box.center)},
  x=.33cm,y=-.4cm]
  \Bcupcap[0.5]{2}{1}
  \draw[thick]
  (0,0)--(0,0.5)
  (0,1.5) -- (0,2)
  (1,0)--(1,0.5)
  (1,1.5)--(1,2);
 \end{tikzpicture}
\right).
\]
\end{enumerate}
Then $\psi(\delta_o)$ is an oriented diagram which we shall call $\delta_{o'}$, with corresponding Lee generator $\s_{o'}$.  In case $\ref{it:FT2}$, $\delta_{o'}$ sits at the far left end of the simplified complex if the strands were oriented the same way, and at the far right end if they were oppositely oriented.

Furthermore, $\psi$ satisfies the following properties:
\begin{enumerate}[label=(\roman*)]
\item \label{it:Lee to Lee} on $\OLC{D}$, $\psi$ `preserves Lee subspaces' in the sense that $\psi(\s_o)$ is a unit multiple of $\s_{o'}$;
\item \label{it:no horiz diff on Lee} any component of the differential on $\delta_{o'}$ in the simplified complex $\psi(\LC{D})$ is a saddle;
\item \label{it:htpy is cob} we can assume that each non-zero component of the homotopy on $\LC{D}$ making $\td\psi\circ\psi\simeq I$ is induced by an oriented planar cobordism.
\end{enumerate}
\end{lemma}
\begin{proof}
For Reidemeister simplifications of the form \ref{it:R1} and \ref{it:R2}, the orientation $o'$ is clearly inherited form $o$.  The relevant cobordism maps $\psi$ are defined in \cite{BN} where Property \ref{it:htpy is cob} is also shown (in fact in both cases, the homotopies are simply births and deaths of a single disjoint circle up to a sign---this fact is then used in \cite{MW} to prove that such maps $\psi$ are so-called very strong deformation retracts).  Property \ref{it:Lee to Lee} is checked by Rasmussen in \cite{Rasmussen}, while Property \ref{it:no horiz diff on Lee} is clear since the simplified complexes are still complexes for link diagrams, which have only saddles as differentials.


We illustrate simplifications of the form \ref{it:FT2} with the following diagram.
\begin{center}
\LeegenFTtwo
\end{center}
The oriented resolution $\delta_o$ on the top row is indicated in blue on the left when the strands are similarly oriented, and in red on the right when they are oppositely oriented.  The resulting $\delta_{o'}$ on the bottom row uses the indicated resolution while maintaining $o$ away from this local picture.  The simplification of the complex utilizes Naot's delooping isomorphism \cite[Definition 3.2]{Nao-kh-universal} and Gaussian elimination as in \cite[Lemma 3.2]{BNfast}.  The reader can check that the blue map is just the identity on $\s_o$, while the red map induces multiplication by $\lambda=\pm 2$, which we have demanded to be a unit in our ground ring (this is the only check that requires the use of a sum of cobordisms, so that one cannot use Rasmussen's arguments immediately).  Property \ref{it:no horiz diff on Lee} is clear in both cases (the red case has no local outward differentials from $\delta_{o'}$, while the blue case has only a saddle; all differentials from crossings away from the simplification are also saddles), and Property \ref{it:htpy is cob} can be checked by hand---the only non-zero homotopy is the death cobordism on the disjoint circle (up to a sign).
\end{proof}

Lemma \ref{lem:simps on Lee} characterizes how certain chain homotopy equivalences treat Lee generators, as well as how the homotopies involved treat them. With all of this in place, we can state and prove the main theorem of this section.

\begin{theorem}\label{thm:oriented res is correct}
Let $D$ be a diagram of a null-homologous link $L$ in $M=\SSr$.  Then the orientation $o$ of $D$ gives rise to a well-defined oriented resolution diagram $\delta_o\in\dKC[0]{D}$ satisfying the following properties.
\begin{itemize}
\item The Lee generator $\s_o$ corresponding to $\delta_o$ generates a summand $\F_o$ in $\Lh[0]{D}$ when $t=1$.  If we expand $\LC{D}$ as a multicone along the infinite twist complexes involved, then the oriented diagram giving rise to $\delta_o$ sits in the far right end of the multicone (that is to say, $\h_{\LC{\infFT}}(\delta_o)=0$).
\item When $\vec{k}=(k,\dots,k)$ with $k>0$, the oriented resolution diagram $\delta_{o,k}\in\dKC[0]{\Lk}$ has corresponding Lee generator $\s_{o,k}$ that gets mapped to a unit multiple of a Lee generator $\s_{o,k}^\#\in \OLCnobrackets^\#(\Lk)$ corresponding to a stable copy of $\delta_o$ sitting in $(\CSharp)^0(\Lk)$.  When $k\geq\intceil{\frac{n^+_D + 2}{2}}$ so that we are in the stable range for computing homology, the finite approximation Lee homology subspace $\langle[\s_{o,k}^\#]\rangle\subset \Lh{\Lk}$ is identified with the Lee homology subspace $\langle[\s_o]\rangle\subset \Lh{D}$.
\end{itemize}
See Figure \ref{fig:oriented res is correct} for an illustration of the situation.
\end{theorem}

\begin{figure}
\[
\begin{tikzpicture}[y=27em]
\node[draw,rectangle](A) at (0,0) {\oriresKClim};
\node[draw,rectangle](B) at (0,-1) {\oriresLeelim};
\draw[->] (A) -- (B) node[midway,right] {Apply $\BartoKh$ and set $t=1$};
\end{tikzpicture}
\]
\caption{Given a diagram $D$ for an oriented link $L\subset\protect\SSr$, every diagram $\Lk$ has an oriented resolution $\delta_{o,k}\in\protect\dKC{\Lk}$ with corresponding Lee generator $\protect\s_{o,k}$.  We also have an oriented resolution $\delta_{o,k}^\#$ (with corresponding Lee generator $\s_{o,k}^\#$) in the equivalent complex $\CSharp(\Lk)$ via replacing each $\FT_{n_i}^k$ with an oriented diagram $\epsilon_{n_i,o}$; this is independent of $k$, and can be denoted $\delta_o$ (with corresponding $\s_o$).  The equivalence maps $\s_{o,k}\mapsto\s_{o,k}^\#$ (up to a unit multiple, omitted for clarity).  Finally, the truncations $\CSharp(\Lk)_{\geq d}$ limit to produce $\dKC{D}$.  Since all of the $\delta_{o,k}$ are in homological grading zero, they can be found in the truncated complex, and become well-defined as soon as $k=\intceil{\frac{n^+_D + 2}{2}}$.}
\label{fig:oriented res is correct}
\end{figure}

\begin{proof}
The proof will be handled in two phases, corresponding to the two boxes of Figure \ref{fig:oriented res is correct} which should be referenced throughout.  First we will see how an orientation of $D$ determines a certain oriented diagram $\delta_o$ via the limiting process defining $\LC{D}$.  Then we will see how this limiting process treats Lee generators to see how the corresponding $\s_o\in\OLC{D}$ generates a summand of homology in $\Lh{D}$, and how this summand relates to summands in the finite approximations $\Lh{\Lk}$.

Recall that $\LC{D}$ is built by inserting infinite twist complexes $\LC{\infFT_{n_i}}$ into the diagram $D$.  These infinite twist complexes are themselves built by performing multicone simplifications on the complexes $\LC{\FT_{n_i}^k}$ to transform them into simplified complexes $\CSharpLee(\FT_{n_i}^k)$, which then fit into a sequence of inclusions (after truncation)
\begin{equation}\label{eq:FT limiting process}
\CSharpLee(\FT_{n_i})_{\geq -1} \hookrightarrow \CSharpLee(\FT_{n_i}^2)_{\geq -3} \hookrightarrow \CSharpLee(\FT_{n_i}^3)_{\geq -5} \hookrightarrow \cdots
\end{equation}
which limit to give $\LC{\infFT_{n_i}}$.  We claim that an orientation $o$ on $\FT_{n_i}$ determines a specific oriented diagram $\epsilon_{n_i,o}\in\CSharpLee(\FT_{n_i})$ that is then preserved throughout this limiting process.  We then define $\delta_o\in\LC{D}$ to be the diagram formed by replacing each infinite twist complex by $\epsilon_{n_i,o}$ and taking the oriented resolution of the other crossings in $D$.

To see how each $\epsilon_{n_i,o}$ is defined, we fix an attaching sphere $i$ and consider the simplification 
\[\LC{\FT_{n_i}} \xrightarrow{\Psi} \CSharpLee(\FT_{n_i}).\]
As described in \cite[Section 2.4]{MW}, $\Psi$ is achieved by simplifying complexes within multicones via maps $\psi$ of the forms \ref{it:R1},\ref{it:R2},\ref{it:FT2} from Lemma \ref{lem:simps on Lee} (see \cite[Section 2.4]{MW}).  According to Lemma \ref{lem:simps on Lee}, if we focus on simplifying only complexes coming from oriented diagrams in our multicone (starting from the initial oriented $\FT_{n_i}$), the single complex maps $\psi$ will compose to reach an oriented diagram, which is unaffected by simplifying the remainder of the diagrams in the multicone to reach $\CSharpLee(\FT_{n_i})$.  We define $\epsilon_{n_i,o}\in(\CSharpLee)^0(\FT_{n_i})$ to be this diagram; abusing notation slightly, we can write this as
\[\epsilon_{n_i,o}:=\psi(\delta_{n_i,o}),\]
where $\delta_{n_i,o}\in\LC{\FT_{n_i}}$ is the usual oriented resolution.  We know that $\h_{\CSharpLee(\FT_{n_i})}(\epsilon_{n_i,o})=0$ because the homological grading is always zero for the oriented resolution, and for null-homologous links there are no grading shifts in the simplification/limiting procedure.  This in turn implies that the through-degree of $\epsilon_{n_i,o}$ is always zero, since every resolution in the stable range has through-degree $0$ (see Theorem \ref{thm:KC'(infFT)}).
See Figure \ref{fig:simp FT4} for an example of this simplification process in the case of $n_i=4$ strands with a specific orientation.

\begin{figure}
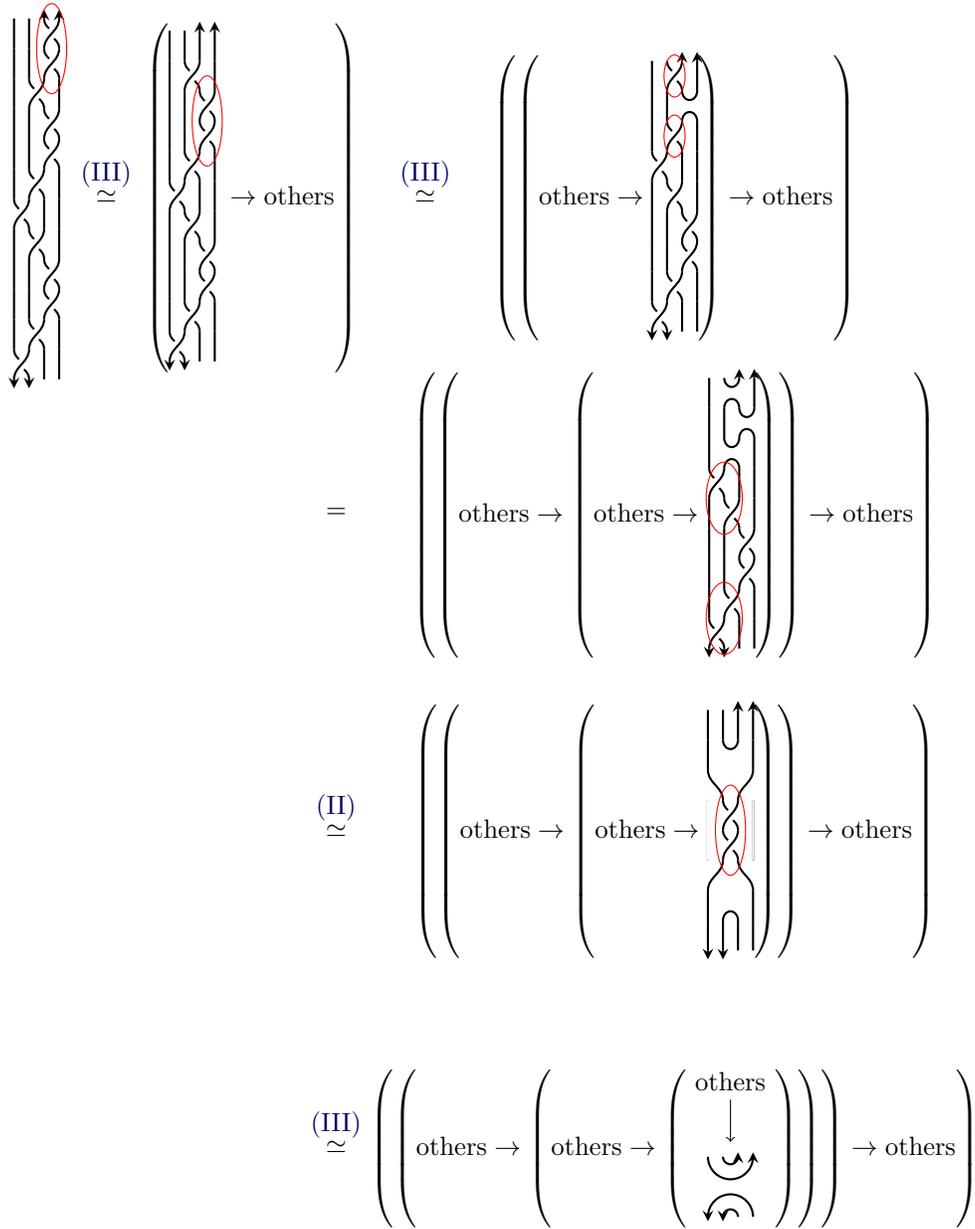

\[\FTfourEx\]
\caption{We keep track of the various oriented diagrams arrived at throughout the simplification $\protect\LC{\FT_4}\rightarrow\protect\CSharpLee(\FT_4)$ for the given orientation.  At each step we circle the crossings about to be locally simplified in the multicone.  The map types (in the notation of Lemma \ref{lem:simps on Lee}) are indicated, except for an equality where we have expanded the complex along the two chosen crossings in the typical way (0-resolution $\rightarrow$ 1-resolution).  
Having performed these simplifications, we can proceed to simplify the `others' without affecting the oriented diagram which we denote $\epsilon_{4,o}$.}
\label{fig:simp FT4}
\end{figure}

All of this was for a local full twist; for the full diagram we let $\delta^\#_{o,1}\in\CSharpLee(D(\vec{1}))$ denote the oriented diagram formed by placing $\epsilon_{n_i,o}$ in place of each $\FT_{n_i}$ in the diagram $D(\vec{1})$, and taking the oriented resolution of the rest of the crossings in $D$.  Abusing notation once more, we can write this as
\[\delta^\#_{o,1}:=\psi(\delta_{o,1}),\]
where $\delta_{o,1}\in\LCnobrackets(D(\vec{1}))$ is the usual oriented resolution.


Now for the limiting procedure, since $\h_{\CSharpLee(\FT_{n_i})}(\epsilon_{n_i,o})=0$, we have that $\epsilon_{n_i,o}$ is already in the stable range after one full twist, and thus it remains constant throughout the sequence of inclusions \eqref{eq:FT limiting process}.  Put another way, the simplifications for $\LC{\FT_{n_i}^k}=\LC{\FT_{n_i}^{k-1}}\otimes\LC{\FT_{n_i}}$ are inductively defined so that $\LC{\FT_{n_i}^{k-1}}$ is simplified first, and any diagram there with through-degree zero (such as $\epsilon_{n_i,o}$) is concatenated with $\FT_{n_i}$, and this concatenation is further simplified via an `untwisting' of the strands using maps $\psi$ of types \ref{it:R1} and \ref{it:R2} from Lemma \ref{lem:simps on Lee} to arrive back at the same diagram (i.e. $\epsilon_{n_i,o}\cdot\FT_{n_i}$ is simplified back to $\epsilon_{n_i}$).  See \cite[Section 2.5]{MW} for further details.  Again passing to the full diagram and taking oriented resolutions of the other crossings of $D$, we have a constant sequence of oriented diagrams $\delta^\#_{o,k}\in\CSharpLee(\FT_{n_i}^k)$ which limits to give the oriented diagram $\delta_o\in\LC{D}$.

We now pass on to phase two, where we investigate how all of these multicone simplifications and limiting procedures affect Lee homology.  If we continue to use the symbol $\Psi$ to denote the simplification $\OLCnobrackets(\Lk)\xrightarrow{\simeq}\OCSharpLee(\Lk)$ for any $k$, and we let $\s_{o,k}$ and $\s_{o,k}^\#$ denote the Lee generators assigned to the oriented diagrams $\delta_{o,k}$ and $\delta_{o,k}^\#$ respectively, then we claim that
\begin{equation}\label{eq:Phi(s)}
\Psi(\s_{o,k})=\lambda_k \s_{o,k}^\#,
\end{equation}
where $\lambda_k\in R$ is some unit that may depend on $k$.  Recall however that $\delta_{o,k}^\# = \delta_o$ for all $k$, so in fact this claim will imply all of the statements in the theorem.

To prove Equation \eqref{eq:Phi(s)}, recall that $\Psi$ (for any $k$) is a composition of multicone maps induced by single-complex simplifications $\psi$ of types covered by Lemma \ref{lem:simps on Lee}.  At any stage of this process, the chain homotopy equivalence from one multicone to the next consists of the corresponding map $\psi$ together with certain alternating combinations of homotopies and multicone differentials (see \cite[Proposition 2.10]{MW}).  Lemma \ref{lem:simps on Lee} shows firstly that the single-complex maps $\psi$ will send Lee generators to unit multiples of corresponding Lee generators throughout (item \ref{it:Lee to Lee}), and secondly that the combinations of homotopies and multicone differentials will always induce the zero map on any Lee generator because our homotopies are single oriented cobordisms (item \ref{it:htpy is cob}) which then must be composed with \emph{unoriented} saddles (item \ref{it:no horiz diff on Lee}) to give unorientable cobordisms which must induce zero maps on Lee generators.
(The fact that unorientable cobordisms induce zero maps in Lee homology follows from \cite[Proposition 4.1]{Rasmussen}. While the result is stated there only for orientable cobordisms, the same proof works verbatim for unorientable cobordisms too.)
Thus the composition $\Psi$ satisfies Equation \eqref{eq:Phi(s)} and the proof is concluded.\end{proof}


\begin{remark}\label{rmk:Lee is copies of Q naturally}
Although Theorem \ref{thm:oriented res is correct} is stated in terms of a single orientation on $L$, it is not hard to see that every orientation $o\in O(L)$ can be treated in the same way, each giving rise to an oriented resolution $\delta_o$ and corresponding Lee generator $\s_o$.  Because the span of these Lee generators is preserved under the finite approximation isomorphism, we can immediately conclude that the set $\{[\s_o]\,|\,o\in O(L)\}$ is linearly independent and provides a natural isomorphism $\Lh{L}\cong \F^{|O(L)|}$ as in Theorem \ref{thm:Lee is copies of Q}.
\end{remark}

The upshot of Theorem \ref{thm:oriented res is correct} and Remark \ref{rmk:Lee is copies of Q naturally} is that we can treat the Lee homology of a link in $\SSr$ in much the same way that we treat the Lee homology of links in $S^3$.  That is to say, the Lee homology for a link diagram $D$ is broken into summands generated by cycles corresponding to certain ``oriented resolutions'' of $D$.  Furthermore, this arrangement respects the finite approximation procedure of Corollary \ref{cor:KC'(L) finite approx} (up to unit multiples), so that we can study these generators by studying the behavior of genuine Lee generators corresponding to oriented resolutions of the link diagrams $\Lk$ for links in $S^3$.
\section{A generalization of Rasmussen's invariant}
\label{sec:s_invt}

\subsection{Defining the $s$-invariant for links in $\SSr$}
Let $R=\F$ be a field of characteristic not equal to $2$. Given a diagram $D$ of a null-homologous link $L\subset \SSr$, the deformed complex $\dKC{D}$ is $\q$-graded if we assign the variable $t$ in the ground ring a $\q$-grading of $\q(t):=-4$.
In this way, when $t$ is set equal to $1$ the Lee complex $\LC{D}$ is a $\q$-filtered complex, and this filtration descends to the homology $\Lh{L}\cong \F^{|O(L)|}$.
We denote the filtration level of $x \in \LC{D}$ by $\q(x)$. If $x$ is a cycle we also denote the filtration level of the homology class $[x]$ by $\q([x]) \geq \q(x)$.
According to Theorem \ref{thm:oriented res is correct}, the orientation of $L$ gives rise to a specific diagram $\delta_o\in\LC[0]{D}$ whose corresponding Lee generator $\s_o\in\OLC[0]{D}$ generates one summand $\F_o$; similarly the opposite orientation $\bar{o}$ for $L$ (for which $L$ would still be null-homologous) also gives rise to a diagram $\delta_{\bar{o}}$ and Lee generator $\s_{\bar{o}}$ generating $\F_{\bar{o}}$.  As in the work of Beliakova and Wehrli \cite[Section 7.1]{beliakova-wehrli} generalizing the arguments in \cite{Rasmussen}, we can make the following definition.

\begin{definition}\label{def:s invt}
Given a diagram $D$ of a null-homologous link $L\subset \SSr$ with orientation $o$ and opposite orientation $\bar{o}$, we define the \emph{$s$-invariant} of $D$ (with coefficients in $\F$) to be
\[s_{\F}(D):= \frac{ \q([\s_o+\s_{\bar{o}}]) + \q([\s_o-\s_{\bar{o}}])}{2}.\]
\end{definition}

When $\F=\Q$, we simply write $s(D)$ for $s_{\Q}(D)$. From now on, for ease of notation, we will restrict to the case $\F=\Q$. However, all the results in this paper still hold for the other invariants $\s_{\F}$.

\begin{remark}
\label{rem:new definition of s}
As described in the proof of Theorem \ref{thm:oriented res is correct}, the diagram $\delta_o$ for $D$ can be viewed as an oriented resolution of a link diagram in $S^3$ obtained from $\Lk$ by replacing each $\FT_{n_i}^k$ with $\epsilon_{n_i,o}$.  The same can clearly be said for $\delta_{\bar{o}}$, and note that $\epsilon_{n_i,\bar{o}}=\epsilon_{n_i,o}$.  Thus we have that $\q([\s_o+\s_{\bar{o}}])$ and $\q([\s_o-\s_{\bar{o}}])$ differ by two, as they do for links in $S^3$; cf.~\cite[Proposition 3.3]{Rasmussen} and \cite[Section 7.1]{beliakova-wehrli}.  Following their notation we will denote the lesser of these two quantities by $s_{\min}$, and the larger by $s_{\max}$.
Thus we have
\begin{equation}
\label{eq:new definition of s}
\q([\s_o]) = \q([\s_{\bar{o}}]) = s_{\min}.
\end{equation}
\end{remark}

\begin{theorem}
\label{thm:s invt finite approx}
Given a diagram $D$ of a null-homologous oriented link $L\subset \SSr$, the $s$-invariant of $D$ can be computed via finite approximation as $s(D)=s(\Lk)$ where $\vec{k}=(k,\cdots,k)$ with $k\geq\intceil{\frac{n_D^++2}{2}}$.
\end{theorem}
\begin{proof}
As described in Theorem \ref{thm:oriented res is correct} (see Figure \ref{fig:oriented res is correct}), $\delta_o$ and $\delta_{\bar{o}}$ can be identified with oriented resolutions of link diagrams $\Lk\subset S^3$ in the finite approximation for $\Lh{L}$.  The finite approximation isomorphism preserves Lee subspaces, and so must preserve the $s$-invariant.\end{proof}

{
\renewcommand{\thethm}{\ref{thm:s invt well defined}}
\begin{theorem}
If $D_1$ and $D_2$ are two diagrams for the same oriented null-homologous link $L$ in $M = \SSr$, then $s(D_1)=s(D_2)$ and thus $s$ gives a well-defined invariant $s(L)$ of oriented null-homologous links $L$ in $M$.
\end{theorem}
\addtocounter{thm}{-1}
}
\begin{proof}
Let $o$ be the orientation of $L$, inducing orientations $o_1$ and $o_2$ on the diagrams $D_1$ and $D_2$ respectively.  As in \cite{beliakova-wehrli}, one must ensure that there exists an isomorphism $\Lh{D_1}\cong\Lh{D_2}$ which sends $[\s_{o_1}]$ to a unit scalar multiple of $[\s_{o_2}]$, i.e.~that link isotopies preserve Lee subspaces in homology.

Any two such diagrams are related by a sequence of moves as described in Section 3.1 in \cite{MW}.  For almost every type of move, we pass to a finite approximation and then use Reidemeister moves for diagrams describing links in $S^3$.  If we set the following values
\[
n^+_{\max}:=\max(n^+_{D_1},n^+_{D_2}), \qquad k:=\intceil{\frac{n^+_{\max}+2}{2}}, \qquad d:= n^+_{\max} - 2k + 1,
\]
this amounts to the following diagram between the $k^\text{th}$ rows of the Lee complexes for $D_1$ and $D_2$ in Figure \ref{fig:oriented res is correct} (again unit multiples have been omitted).
\[
\oriresLeeIsotopy
\]
The Reidemeister moves induce a chain homotopy equivalence $\rho^*$ which maps $\s_{o_1,k}$ to a unit multiple of $\s_{o_2,k}$ modulo boundaries (omitted from the diagram for clarity, as in Figure \ref{fig:oriented res is correct}) via the work in \cite{Rasmussen} and \cite{beliakova-wehrli}.  This in turn induces a map $\rho^\#_{\geq d}$ between the truncated, simplified complexes which maps $\s_{o_1}$ to a unit multiple of $\s_{o_2}$ modulo boundaries.  Although $\rho^\#_{\geq d}$ may no longer be a chain homotopy equivalence, it preserves our oriented resolutions up to unit multiples which is all we need.  Any such choice of $k$ allows us to compose $\rho^\#_{\geq d}$ with the finite approximation isomorphisms (which also preserve Lee generators up to unit multiples via Theorem \ref{thm:oriented res is correct}) to choose an isomorphism
\[\Lh{D_1} \cong \Lh{\Lk[D_1]} \cong \Lh{\Lk[D_2]} \cong \Lh{D_2}\]
which preserves Lee subspaces as required (the unit multiples involved may depend on $k$).

The only move that is not handled so easily is the so-called surgery-wrap move, discussed in Theorem 3.15 in \cite{MW}.  A careful reading of Lemma 3.12 in that paper shows that surgery-wrap moves require working with the truncated finite approximation (where all of the diagrams coming from the full twists are split) and simplifying multicones using only transformations of the type \ref{it:R2} from Lemma \ref{lem:simps on Lee}.  Thus the proof that surgery-wrap moves send Lee generators to unit multiples of each other follows precisely along the proof of Theorem \ref{thm:oriented res is correct}.   Note that, although there are typically shifts in both homological degree and $\q$-degree for such moves, both of those shifts vanish when the link is null-homologous. \end{proof}

\begin{remark}
Theorem~\ref{thm:fda} from the Introduction is basically a restatement of Theorem~\ref{thm:s invt finite approx}, taking into account Theorem~\ref{thm:s invt well defined}.
\end{remark}

\begin{remark}\label{rmk:lots of maps}
In the proof of Theorem \ref{thm:s invt well defined}, we use a specific value for $k$ to consider the finite approximation, but of course any larger value of $k$ could be chosen.  This means that we actually have infinitely many maps which can be used to represent the isomorphism $\Lh{D_1}\cong\Lh{D_2}$, one for each large enough $k$.  The proof shows that any one of these maps would send $[\s_{o_1}]$ to a unit multiple of $[\s_{o_2}]$, but the precise unit may be different for different values of $k$.
Nevertheless, just the existence of one such map is enough to show that the Lee subspace, and thus the value of $s$, is preserved.  The situation is similar for self-diffeomorphisms in Theorem \ref{thm:s invt diffeo and rev} below.
\end{remark}

\begin{theorem}\label{thm:s invt diffeo and rev}
Given a null-homologous link $L\subset\SSr$, we have
\[s(r(L))=s(L)=s(\Phi(L)),\]
where $r(L)$ denotes the reverse of $L$ and $\Phi$ denotes any orientation-preserving self-diffeomorphism of $\SSr$.
\end{theorem}
\begin{proof}
For reversals, the diagrammatic description of $\Lh{L}$ makes the statement clear (see Proposition \ref{prop:Kh(r(L))}).  

For $\Phi$, we proceed as in the proof of Theorem \ref{thm:s invt well defined} above, except now we follow along the arguments in the proof of Theorem \ref{thm:dKC preserved by even diffeos} instead of those from \cite[Section 3]{MW}. For rotations and transpositions, the invariance of $\Lh{L}$ used isotopies of finite approximation diagrams, which work just as in the proof of Theorem \ref{thm:s invt well defined}. For Dehn twists we have the re-indexing, which preserves Lee generators without shifts in the null-homologous case.  Finally for handle slides, a careful reading of the proof of Theorem \ref{thm:dKC preserved by even diffeos} shows that we use multicone simplifications that incorporate only crossing-removing Reidemeister I and II moves (both included in Lemma \ref{lem:simps on Lee}), allowing the proof of the invariance under surgery-wrap moves (by following along the lines of the proof of Theorem \ref{thm:oriented res is correct}) to carry over to handle slides without serious modification.
\end{proof}

\subsection{Disjoint unions}
\label{sec:disj union}

\begin{proposition}\label{prop:disj union}
Consider null-homologous links $L_i \subset \#^{k_i} (\SSone)$, for $i=1, \dots, m$, and let
$$ L = \coprod_i L_i \subset \#^{\sum k_i}  (\SSone).$$
Then
$$ s(L) = \sum s(L_i) - m.$$
\end{proposition}
\begin{proof}
This is a simple consequence of Theorem \ref{prop:s invt by finite approx}, which states that we can compute the relevant $s$-invariants via links in $S^3$.  The disjoint union described here is equivalent to taking the disjoint union of these finite approximation links, for which \cite[Equation (8)]{beliakova-wehrli} provide the desired formula.
\end{proof}

\subsection{Cobordisms in $I \times \SSr$}
In this section we analyze the effects that cobordisms of links in $I\times\SSr$ have on Lee homology, in order to conclude (following \cite{Rasmussen} and \cite{beliakova-wehrli}) that certain cobordisms between links $L_1$ and $L_2$ lead to bounds on the difference $s(L_2)-s(L_1)$.

\begin{theorem}\label{thm:cob map on Lee}
Let $M=\SSr$, and let $\Sigma\subset I \times M$ denote an oriented cobordism between oriented  links $L_1\subset \{0\}\times M$ and $L_2\subset \{1\}\times M$.  Fix an orientation $o_1$ on $L_1$ (not necessarily the original orientation) making $L_1$ null-homologous in $M$, and let $\s_{o_1}$ denote the corresponding Lee generator.  Denote by $O(\Sigma,o_1)$ the set of orientations $o$ of $\Sigma$ whose induced orientation $o|_1$ on $L_1$ is $\bar{o}_1$.

Then the cobordism $\Sigma$ induces a (possibly not natural) map $\phi_\Sigma:\Lh{L_1}\rightarrow\Lh{L_2}$, of filtration degree $\chi(\Sigma)$ (the Euler characteristic of $\Sigma$), such that
\begin{equation}\label{eq:cob map on Lee}
\phi_{\Sigma}([\s_{o_1}]) = \sum_{o\in O(\Sigma,o_1)} a_o [\s_{o|_2}]
\end{equation}
where for each $o\in O(\Sigma,o_1)$, $a_o$ is a unit in $R$, and $o|_2$ is the orientation on $L_2$ induced by $o$.
Here $\s_{o_1} \in \BartoKh\LC{D_1}$ and $\s_{o|_2} \in \BartoKh\LC{D_2}$ denote the Lee generators associated to $o_1$ and $o|_2$, for some diagrams $D_1$ and $D_2$ for the links $L_1$ and $L_2$.

Furthermore, when $L_1$ and $L_2$ are null-homologous, the map $\phi_{\Sigma}$ is filtered, of filtration degree $\chi(\Sigma)$ (the Euler characteristic of $\Sigma$).
\end{theorem}
\begin{proof}
This is a simple generalization of Rasmussen's arguments in \cite{Rasmussen}.  We break our cobordism $\Sigma$ into elementary pieces $\Sigma_i$ between diagrams $D_i$ and $D_{i+1}$, and then check that each piece can be used to build a map satisfying
\[\phi_{\Sigma}([\s_{o_i}]) = \sum_{o\in O(\Sigma,o_i)} a_o [\s_{o|_{i+1}}].\]
If $\Sigma_i$ corresponds to a local isotopy, we can use a corresponding map from the proof of Theorem \ref{thm:s invt well defined}. (Note that in this case, $\Sigma_i$ is a cylinder with $\chi(\Sigma_i)=0$, and there is only one orientation in $O(\Sigma_i,o_i)$.)  As in Remark \ref{rmk:lots of maps}, there are actually many choices of such a map, but we need only choose one to label as $\phi_{\Sigma_i}$.  If $\Sigma_i$ corresponds to a Morse move (a cup, cap, or saddle), we can still choose $\phi_{\Sigma_i}$ by choosing a finite approximation diagram to act upon as in the proof of Theorem \ref{thm:s invt well defined}, and such maps have already been checked in \cite{Rasmussen}.  When we glue such cobordisms together, we compose maps leading to a double sum
\begin{align*}\phi_{\Sigma_{i+1}\cup \Sigma_i} ([\s_{o_i}]) &= \phi_{\Sigma_{i+1}}\circ\phi_{\Sigma_i} ([\s_{o_i}])\\
&= \sum_{o\in O(\Sigma_i,o_i)} a_o \phi_{\Sigma_{i+1}}([\s_{o|_{i+1}}])\\
&= \sum_{o\in O(\Sigma_i,o_i)} a_o \sum_{o'\in O(\Sigma_{i+1},o|_{i+1})} b_{o'} [\s_{o'|_{i+2}}]\\
&= \sum_{o'' \in O(\Sigma_{i+1}\cup \Sigma_i, o_i)} c_{o''} [\s_{o''|_{i+2}}]
\end{align*}
where the last line follows from using $o$ and $o'$ agreeing at $D_{i+1}$ to uniquely define an orientation $o''$ on all of $\Sigma_{i+1}\cup\Sigma_i$, while $c_{o''}=a_ob_{o'}$ is a unit since each of $a_o,b_{o'}$ were.
\end{proof}

\begin{remark}\label{rmk:lots of maps for cobs}
The reason we added the phrase ``possibly not natural'' in the statement of Theorem \ref{thm:cob map on Lee} is the following.
As in Remark \ref{rmk:lots of maps}, our elementary cobordisms only induce maps up to a choice of finite approximation.
Different choices may lead to different unit multiples during the proof of Theorem \ref{thm:cob map on Lee}.
Moreover, we have not checked invariance of the maps under movie moves.
\end{remark}


%
 
We now deduce Theorem \ref{thm:GenusBoundCylinders} along the lines of \cite{Rasmussen} and \cite{beliakova-wehrli}.

{
\renewcommand{\thethm}{\ref{thm:GenusBoundCylinders}}
\begin{theorem}
Consider an oriented cobordism $\Sigma\subset I \times \SSr$ from a link $L_1$ to a second link $L_2$. Suppose that every component of $\Sigma$ has a boundary component in $L_1$. Then,  we have an inequality of $s$-invariants
\[s(L_2) - s(L_1) \geq  \chi(\Sigma),\]
where $\chi(\Sigma)$ denotes the Euler characteristic of $\Sigma$.
\end{theorem}
\addtocounter{thm}{-1}
}
\begin{proof}
Let $o_1$ and $o_2$ denote the orientations of $L_1$ and $L_2$, respectively. Since every component of $\Sigma$ has a boundary component in $L_1$, there can be at most one orientation $o$ on $\Sigma$ that induces the orientation $o_1$ on $L_1$. Thus, the sum in Equation \eqref{eq:cob map on Lee} contains only a single term, sending the Lee generator $[\s_{o_1}]\in\Lh{L_1}$ to a unit multiple of $[\s_{o_2}]\in\Lh{L_2}$.
Since $\phi_{\Sigma}$ is a filtered map of degree $\chi(\Sigma)$, and $s(L_i)$ is defined via the filtration level coming from $[\s_{o_i}]$ using Equation \eqref{eq:new definition of s}, the bound follows.
\end{proof}

\begin{remark}\label{rmk:BW wrong}
The opposite bound $s(L_1)-s(L_2)\geq \chi(\Sigma)$ is false as can be seen from the annulus cobordism $(\chi=0)$ from $L_1=U_2$, the two-component unlink ($s=-1$), to $L_2=\varnothing$, the empty link ($s=1$). (This false opposite bound was originally claimed in \cite{beliakova-wehrli}, but later withdrawn.)
\end{remark}

\section{Computing Lee homology in $\SSone$ via Hochschild homology}
\label{sec:Hochschild}

When $D$ is a diagram for a link $L$ in $M=\SSone$ and $\dKC{D}$ depends on inserting only a single copy of $\dKC{\infFT_n}$ (for even $n=2p$), there is an alternative way to view the construction of $\OdKC{D}$ via the Hochschild homology of certain complexes of bimodules over Khovanov's arc algebra \cite{KhTangles}.  Indeed, this version is how Rozansky defines $\OKC{D}$ in \cite{Roz}.  We begin by reviewing the definitions below.

Let $\caps_p$ denote the set of crossingless matchings on $n=2p$ endpoints (i.e., the set of $(0,2p)$-Temperley-Lieb diagrams).  After shifting quantum gradings by $p$, the arc algebra $H_p$ is generated by Khovanov complexes for concatenated diagrams of the form $a\caprefl{b}$ for $a,b\in\caps_p$, where $\caprefl{b}$ denotes the vertical reflection of $b$. We notate this visually as
\[H_p= \bigoplus_{a,b\in\caps_p}\q^{-p}\OKCp{ \capcap{a}{b}}.\]
Then any $(2p,2p)$-tangle $T$ in $I \times D^2$ determines a complex of $H_p$-bimodules \[\KCbm{T} := \bigoplus_{a,b\in\caps_p} \q^{-p}\OKC{aT\caprefl{b}},\] which we notate visually as
\[\KCbm{ \begin{tikzpicture}[baseline={([yshift=-.7ex]current bounding box.center)},x=1.7em,y=-1.7em] \Bboxst{$T$} \end{tikzpicture} }
 := \q^{-p}\bigoplus_{a,b\in\caps_p} \OKCp{ \capBcap{a}{T}{b} }.\]
(This is not an abuse of notation: The functor $\BartoKh$ was earlier defined only on $(0,0)$-Temperley-Lieb diagrams; here we are extending this to $(p,p)$-Temperley-Lieb diagrams by the above equation.)
The action of $c\caprefl{d}\in H_p$ on the left (respectively right) side of a summand is zero if $d \neq a$ (respectively $c \neq b$).  Meanwhile, $c\caprefl{a}$ acts by transforming $\OKC{aT\caprefl{b}}$ into $\OKC{cT\caprefl{b}}$ by a sequence of saddle maps $\caprefl{a}a\rightarrow I$, and similarly for the action on the other side.

\[ \OKCp{\capcap{c}{a}} \otimes \OKCp{\capBcap{a}{T}{b}} = 
\OKCp{
    \begin{tikzpicture}[baseline={([yshift=-.7ex]current bounding box.center)},x=1.7em,y=-1.7em]
      \node at (0,0) { \capcap{c}{a} };
      \node at (0,2) { \capBcap{a}{T}{b} };
    \end{tikzpicture}
  }
\xrightarrow{\text{saddles}}
\OKCp{
 \begin{tikzpicture}[baseline={([yshift=-.7ex]current bounding box.center)},x=1.7em,y=-1.7em]
  \BcapTopst[0]{$c$}
  \Bboxst[0]{$I$}
  \Bboxst[1]{$T$}
  \BcapBotst[2]{$\caprefl{b}$}
 \end{tikzpicture}
 }
\]
See \cite{KhTangles} for more details. (Note, Khovanov's quantum grading conventions in \cite{KhTangles} are opposite of his original conventions from \cite{Kh}; in this paper we are consistently following his original convention.) Then the key result is Theorem 6.7 in \cite{Roz}, where it is shown that the simplified full twist complexes stabilize to give a projective resolution of the identity bimodule $\KCbm{I}$; from this it is clear that for any $(2p,2p)$ tangle $T$, $\Kh{L}\simeq \HH{T}$ where $L$ is the closure of the $(2p,2p)$-tangle $T$ in $S^1 \times S^2$, and $\HH{T}$ denotes the Hochschild homology of the complex $\KCbm{T}$.

All of the definitions clearly generalize if we use the deformed complexes $\OdKC{aT\caprefl{b}}$ or the Lee complexes $\LCbm{aT\caprefl{b}}$ instead.

\begin{theorem}\label{thm:infFT is proj res}
Let $H_p'$ be the deformed Khovanov arc algebra.
Then, the limiting complex $\dKCbm{\infFT_{2p}}$ of $H_p'$-bimodules assigned to the infinite twist in the deformed category gives a projective resolution of the identity bimodule $\dKCbm{I}$, and similarly for the Lee complexes $\LCbm{\infFT_{2p}}$.
\end{theorem}
\begin{proof}
We summarize and expand upon the main points of Rozansky's argument in \cite{Roz}, but using our simplified complexes $\BartoKh \CSharp(\FT_{2p}^k)$ in the deformed category instead.  We know from Theorem \ref{thm:KC'(infFT)} that the complex $\dKC{\infFT_{2p}}$ is comprised of split diagrams; these correspond to projective bimodules after applying $\BartoKh$ (see \cite[Section 2.5]{KhTangles}).  This complex is determined through any finite homological degree by truncated versions of the simplified complexes $\BartoKh \CSharp(\FT_{2p}^k)$.  The key point is to note that each of these complexes is itself chain homotopy equivalent to the identity bimodule $\dKCbm{I}$.

Indeed, if we ignore the two algebra actions for a moment, it is clear that for any $a,b\in\caps$, 
\[\BartoKh \CSharp\left( \capBcap{a}{\FT_{2p}^k}{b} \right) \simeq \OdKCp{ \capBcap{a}{\FT_{2p}^k}{b} } \stackrel{u^*}{\simeq}
\OdKCp{ \capBcap{a}{I}{b} }\]
via a simple untwisting $u$ of the link $a\FT_{2p}^k \caprefl{b}$.  Depending on which way we untwist, one of the two actions clearly commutes with this equivalence.  Meanwhile, the other action involves saddles $s$ between pairs of points, one of which will be involved in the untwisting $u$.  The cobordism $s\circ u$ is isotopic to $u \circ s$ (see \cite[Lemma 2.25]{MW}).
\[
\begin{tikzpicture}[baseline={([yshift=-.7ex]current bounding box.center)},x=8em,y=5em]
\node(A) at (-1,0)
{
 \begin{tikzpicture}[baseline={([yshift=-.7ex]current bounding box.center)},x=1.7em,y=-1.7em]
 \node at(0,0){ \capBcap{a}{\FT_{2p}^k}{b} };
 \node at (0.3,1.3){
    \begin{tikzpicture}[x=.5em,y=-.5em]
      \draw[->,red]
        (0,0) to[out=-90,in=180] (1,.5) to[out=0,in=-90] (2,0) to[out=90,in=0] (1,-.5);
    \end{tikzpicture}
    };
 \node[right,red] at (.6,1.3){$u$};
 \node at (0,3) { \capcap{b}{c} };
 \draw[blue,thick] (-.3,1.3)--(-.3,2.2);
 \node[left,blue] at (-.3,1.75) {$s$};
 \end{tikzpicture}
};

\node(B) at (0,1)
{
 \begin{tikzpicture}[baseline={([yshift=-.7ex]current bounding box.center)},x=1.7em,y=-1.7em]
 \node at(0,0){ \capBcap{a}{I}{b} };
 \node at (0,2) { \capcap{b}{c} };
 \end{tikzpicture}
};

\node(C) at (0,-1)
{
 \begin{tikzpicture}[baseline={([yshift=-.7ex]current bounding box.center)},x=1.7em,y=-1.7em]
 \node at(0,0){ \capBcap{a}{\FT_{2p}^k}{c} };
 \end{tikzpicture}
};

\node(D) at (1,0)
{
 \begin{tikzpicture}[baseline={([yshift=-.7ex]current bounding box.center)},x=1.7em,y=-1.7em]
 \node at(0,0){ \capBcap{a}{I}{c} };
 \end{tikzpicture}
};

\draw[->]
(A) -- (B) node[midway,above,red]{$u$};

\draw[->]
(A) -- (C) node[midway,above,blue]{$s$};

\draw[->]
(B) -- (D) node[midway,above,blue]{$s$};

\draw[->]
(C) -- (D) node[midway,above,red]{$u$};

\node at (0,0) {$\simeq$};

\end{tikzpicture}
\]
This implies that the corresponding maps on complexes are chain homotopic up to a sign, which we can arrange to be consistently positive (see \cite[Theorem 7.17]{Roz} or \cite[Lemma 2.26]{MW}).  Meanwhile, since the \emph{simplified} complex $\BartoKh \CSharp(\FT_{2p}^k)$ is supported in non-positive homological degrees, and $\OdKC{I}$ is supported purely in degree zero, one can quickly conclude that the homotopy must be zero as well, and thus the maps commute on the nose.  Therefore we do indeed have an equivalence of bimodules, and the stable limiting complex $\dKCbm{\infFT_{2p}}$ gives a projective resolution of $\dKCbm{I}$.  All of the same arguments work equally well when we set the deformation parameter $t:=1$, for the case of Lee homology.
\end{proof}

\begin{corollary}\label{cor:Hochschild}
Let $T$ be a $(2p,2p)$ tangle in $I \times D^2$, and let $L$ be the corresponding link in $\SSone$ formed by collapsing the boundary of $D^2$ to a point, and identifying the boundary points of $I$.  Let $\dKCbm{T}$ (respectively $\LCbm{T}$) be the complex of bimodules using the deformed (respectively Lee) complexes as above, with corresponding Hochschild homology $\dHH{T}$ (respectively $\LHH{T}$).  Then $\dKh{L}\cong \dHH{T}$ (respectively $\Lh{L}\cong \LHH{T}$).
\end{corollary}
\begin{proof}
The stable limiting complex $\dKCbm{\infFT_{2p}}$ gives a projective resolution of $\dKCbm{I}$, which indicates that the homology of $\dKCbm{\infFT_{2p}\cdot T}$ is precisely the Hochschild homology $\dHH{T}$.  But this is equivalent to taking the closure of $\infFT_{2p} \cdot T$ in $S^3$ and computing the usual Khovanov homology, which is precisely how $\dKh{L}$ is defined.  The statements for Lee homology follow from setting $t:=1$.
\end{proof}

\section{The $s$-invariant of the link $\Fp \subset \SSone$}
\label{sec:Lpp}

We now aim to compute the $s$-invariant of the specific oriented link $\Fp \subset \SSone$ depicted on the left of Figure \ref{fig:Fp}.  In words, $\Fp$ is the image of the $n=2p$ strand identity braid in $I \times S^2$ under the identification of the two boundary spheres $\{0\}\times S^2$ and $\{1\}\times S^2$, with $p$ strands oriented `upwards' and $p$ strands oriented `downwards'.

{
\renewcommand{\thethm}{\ref{thm:s(Fp)}}
\begin{theorem}
For the link $\Fp \subset \SSone$, we have
\[s(\Fp)=1-2p.\]
\end{theorem}
\addtocounter{thm}{-1}
}
\begin{proof}
The Lee complex for the link $\Fp$ is determined by the complex $\LC{\infFT_{2p}}$.  Theorem \ref{thm:infFT is proj res} tells us that the deformed complex $\dKCbm{\infFT}$ is a projective resolution of $\dKCbm{I}$, which implies that $\dKCbm{\infFT}$ is chain homotopy equivalent to any projective resolution of the identity that we like.  We choose to use the bar construction $B$ illustrated below 
(note the $q$-degree shifts, which are determined by the fact that every map consists of $p$ saddles which are of degree -1).
\[B:=
\begin{tikzpicture}[baseline={([yshift=-.7ex]current bounding box.center)},x=12em,y=1.7em]


\node(A) at (3,0)
 {$\underline{
  \displaystyle\bigoplus_{a\in\caps_p} \q^{-p}\left(
  \begin{tikzpicture}[baseline={([yshift=-.7ex]current bounding box.center)},x=1.7em,y=-1.7em]
   \BcapBotst[0]{$\caprefl{a}$}
   \BcapTopst[1.5]{$a$}
  \end{tikzpicture}
  \right)
 }$};

\node(B) at (2,0)
 {$
  \displaystyle\bigoplus_{\substack{b_i\in\caps_p\\ i=1,2}} \q^{-2p}\left(
  \begin{tikzpicture}[baseline={([yshift=-.7ex]current bounding box.center)},x=1.7em,y=-1.7em]
   \BcapBotst[0]{$\caprefl{b_1}$}
   \BcapTopst[1.5]{$b_1$}
   \BcapBotst[1.5]{$\caprefl{b_2}$}
   \BcapTopst[3]{$b_2$}
  \end{tikzpicture}
  \right)
 $};

\node(C) at (1,0)
 {$
  \cdots
  \displaystyle\bigoplus_{\substack{c_i\in\caps_p\\ i =1,2,3}} \q^{-3p} \left(
  \begin{tikzpicture}[baseline={([yshift=-.7ex]current bounding box.center)},x=1.7em,y=-1.7em]
    \BcapBotst[0]{$\caprefl{c_1}$}
    \BcapTopst[1.5]{$c_1$}
    \BcapBotst[1.5]{$\caprefl{c_2}$}
    \BcapTopst[3]{$c_2$}
    \BcapBotst[3]{$\caprefl{c_3}$}
    \BcapTopst[4.5]{$c_3$}
   \end{tikzpicture}
   \right)
   $};
   

\draw[->] (2.35,1) -- node[midway,above]{$s_1$} (A);
\draw[->] (2.35,-1) -- node[midway,below]{$-s_2$} (A);
   
\draw[->] (1.45,1.7) to[out=0,in=180-30] node[midway,above]{$s_1$} (B) ;
\draw[->] (1.45,0) -- (B) node[midway,above]{$-s_2$};
\draw[->] (1.45,-1.7) to[out=0,in=180+30] node[midway,below]{$s_3$} (B);

\end{tikzpicture}
\]
The underlined term is in homological degree zero, where each summand comes equipped with a map to the identity braid via saddles $s:\caprefl{a}a \rightarrow I$ as in the multiplication maps of the arc-algebras.  Similarly, the differentials are matrices made of the various choices for saddles $\pm s_i : \caprefl{a_i}a_i \rightarrow I$ as illustrated.



 


 




We have a chain homotopy equivalence $f':\dKCbm{\infFT} \rightarrow B$ in the deformed category, which corresponds to a degree-zero filtered equivalence on Lee complexes when $t=1$.  When we close the diagrams of $\LC{\infFT}$ in $S^3$, we retrieve $\LC{\Fp}$; we will let $\overline{B}$ denote the corresponding complex associated to closing the diagrams of $B$.  Thus we find that we have a quasi-isomorphism
\[f:\OLC{\Fp} \rightarrow \overline{B}\]
which preserves the spectral sequence and, in particular, the filtration levels of homology.
Normally for links, to say that $f$ preserves the $s$-invariant, we would also need to know that $f$ `preserves' corresponding Lee generators.  Here though, because any null-homologous orientation of $\Fp$ is isotopic to any other (the strands in $\Fp$ can easily be permuted in $\SSone$), all of the corresponding Lee generators for $\Fp$ sit in one filtration level and we do not need to worry about specific generators at all.  Thus to calculate $s(\Fp)$, it is enough to compute the filtration level of the homology of $\overline{B}$, which sits entirely in homological degree zero.

To this end, we consider the complex $\overline{B}$.  In homological degree zero, this complex contains generators coming from diagrams of the form $\overline{\caprefl{a} a}$ consisting of $p$ circles in $D^2$.  Being a projective resolution of the identity bimodule $I$, we have the filtered chain map $\pi$ from such diagrams to the unlink $\overline{I}$ of $2p$ circles via $p$ saddles described above.  All of the saddles split one circle into two, and thus induce co-multiplication maps on the generators.

Let $\x_{a}\in\overline{B}^0$ denote the generator having a label of $x$ on each circle in $\overline{\caprefl{a} a}$.  Together with the $q$-degree shift of $-p$, this generator sits in filtration level $-2p$, which is the minimum (most inclusive) such level.  Being in homological degree zero of $\overline{B}$ means $\x_{a}$ is a cycle.  Because the map $\pi$ consists entirely of co-multiplications, we can compute
\[\pi(\x_a) = \x_I + \y\]
where $\x_I$ denotes the term having a label of $x$ on each circle of $\overline{I}$ (which also sits in filtration level $-2p$), and $\y$ consists of other terms in higher filtration level.
Now suppose $\x_a$ is homologous to some $\z\in\overline{B}^0$ (possibly zero) in higher filtration level; let $\w\in\overline{B}^{-1}$ such that 
\[d\w=\x_a - \z\]
Since $\pi$ is a chain map, we would then have
\[0 = \pi d (\w) = \x_I + \y - \pi(\z),\]
which in turn enforces
\[\x_I = \pi(\z)-\y,\]
which is a contradiction since $\x_I$ is in filtration level $-2p$ and
the right hand side is in higher filtration level.  Thus, the Lee
homology for $\Fp$ contains a generator in minimal filtration level
$-2p$. 
Since $s(L) = s_{\min}(L)+1$ (see Remark \ref{rem:new definition of s}), we find that $s(\Fp)$ is one
greater than this minimum level.
\end{proof}

As a corollary, we can compute the $s$-invariant for the link $\Fp(1) \subset S^3$, shown on the right hand side of Figure~\ref{fig:Fp}.

{
\renewcommand{\thethm}{\ref{thm:sLpp}}
\begin{theorem}
Let $\Fp(1) \subset S^3$ denote the torus link $T(2p,2p)$ with $p$ strands oriented one way, and $p$ strands oriented the other way. Then
\[
s(\Fp(1)) = 1-2p.
\]
\end{theorem}
\addtocounter{thm}{-1}
}
\begin{proof}
Combine Theorem \ref{thm:s(Fp)} and Theorem \ref{prop:s invt by finite approx}; for $L=\Fp$, $n^+_L=0$ and $k=1$ in the statement for the finite approximation.
\end{proof}

\section{Adjunction in $\#^r \bCP$}
\label{sec:adj}

\subsection{Orientation and positivity}

Recall the notion of generalized crossing change from \cite{Positive}, which was mentioned in the Introduction; cf.~Figure~\ref{fig:GCC}. 

\begin{remark}
\label{rem:GCC}
Suppose that a link $L^\tw$ is obtained from $L$ by adding a generalized negative crossing, and let $\ell = |L| = |L^\tw|$ be the number of link components. Then there is a cobordism $\Sigma$ in $\bCP\setminus(B^4 \sqcup B^4)$ from $L$ to $L^\tw$, which consists of $\ell$ disjoint cylinders, each of them connecting a component of $L$ to one of $L^\tw$, and which is homologically trivial rel boundary in $\bCP\setminus(B^4 \sqcup B^4)$.
The fact that $\Sigma$ is null-homologous follows from the fact that in the twisting region there are as many strands going up as there are going down.
\end{remark}

\begin{definition}
Given a closed 4-manifold $X$, we say that an $\ell$-component link $L \subseteq S^3$ is \emph{strongly H-slice in $X$} if $L = \de \Sigma$, where $\Sigma \subseteq X^{\circ} := X \sm B^4$ is an embedded surface consisting of $\ell$ disjoint discs, with $[\Sigma] = 0$ in $H_2(X^{\circ}, \de X^{\circ})$.
\end{definition}

\begin{remark}
\label{rem:Ltw}
Let $L^\tw$ be obtained from $L$ by adding a generalized negative crossing. If $L$ is strongly slice (i.e., strongly (H-)slice in $S^4$), then $L^\tw$ is strongly H-slice in $\bCP$. This is because, by Remark~\ref{rem:GCC}, each component of $L^\tw$ is connected by a cylinder to a component of $L$. The cylinder can be glued to the disc in $B^4$ bounded by this component of $L$.

Analogously, if instead $L^\tw$ is obtained from $L$ by adding a generalized \emph{positive} crossing and $L$ is strongly slice, then $L^\tw$ is strongly H-slice in $\CP$.
\end{remark}

\begin{example}
\label{ex:LHT}
The standard diagram of the left handed trefoil $T_{2,-3}$ is obtained from a diagram of the unknot by turning a positive crossing into a negative crossing. Thus, $T_{2,-3}$ is strongly H-slice in $\bCP$.
\end{example}

\begin{example}
The right handed trefoil $T_{2,3}$ is not strongly H-slice in $\bCP$. This is obstructed e.g.~by Theorem \ref{thm:adjunction-tau} and Corollary \ref{cor:adjunction}.
However, $T_{2,3}$ bounds a non-null-homologous disc in $\bCP$. In fact, the crossing change \emph{from negative to positive} can also be realized by adding a positive twist (same as for the crossing change \emph{from positive to negative}), but now both strands in the twisting region point in the same direction, so the surface one gets is no longer null-homologous.
\end{example}

\begin{remark}
$L$ is strongly H-slice in $\bCP$ if and only if $\m L$ is strongly slice in $\CP$. Thus, the right handed trefoil is strongly H-slice in $\CP$.
\end{remark}

\begin{example}
The link $\Fp(1)$ is obtained from the $(2p)$-component unlink in $S^3$ by adding a generalized negative crossing. Thus, $\Fp(1)$ is strongly H-slice in $\bCP$.
\end{example}

\subsection{An adjunction inequality in $\#^r\bCP$}

Ozsv\'ath and Szab\'o proved the following adjunction inequality for the $\tau$ invariant. (An alternative proof based on cobordism maps was given in~\cite{Zgradings}.)

\begin{theorem}[{\cite[Theorem 1.1]{os-tau}}]
\label{thm:adjunction-tau}
Let $W$ be a smooth, oriented four-manifold with $b_2^+(W) = b_1(W) = 0$, and $\de W = S^3$.
Let $K \subset \del W= S^3$ be a knot, and $\Sigma \subset W$ a properly, smoothly embedded oriented connected surface, such that $\del \Sigma =K$. Then
\[
2 \tau(K) \leq 1-\chi(\Sigma) - \left|[\Sigma]\right| - [\Sigma] \cdot [\Sigma].
\]
\end{theorem}

Here $|[\Sigma]|$ and $[\Sigma] \cdot [\Sigma]$ are the $L^1$-norm and the intersection form computed on the homology class $[\Sigma] \in H_2(W,\de W)$. They both vanish if $[\Sigma] = 0$.

\begin{remark}
Theorem~\ref{thm:adjunction-tau} remains true if we drop the assumption $b_1(W)=0$. Indeed, we can reduce this to the case $b_1(W) =0$ by doing surgery on the loops that generate $H_1(W; \Q)$. 
\end{remark}

As we will see, an analogous adjunction inequality holds for Rasmussen's $s$-invariant in $S^3$, in the case of null-homologous surfaces in $\#^t\bCP$ (see Corollary \ref{cor:adjunction}).

\begin{theorem}
\label{thm:strongadj}
Consider an oriented cobordism $\Sigma\subset Z = (I \times \SSr) \# (\#^t\bCP)$ from a null-homologous link $L_1$ to a second null-homologous link $L_2$.
Suppose that $[\Sigma] = 0$ in $H_2(Z, \de Z)$ and that every component of $\Sigma$ has a boundary component in $L_2$.
Then, we have an inequality of $s$-invariants
\[s(L_1) - s(L_2) \geq \chi(\Sigma),\]
where $\chi(\Sigma)$ denotes the Euler characteristic of $\Sigma$.
\end{theorem}

The special case $r=0$ of Theorem \ref{thm:strongadj} is Theorem \ref{thm:weakadj} from the Introduction.

{
\renewcommand{\thethm}{\ref{thm:weakadj}}
\begin{theorem}
Consider a null-homologous oriented cobordism $\Sigma\subset Z = (\#^t\bCP)\setminus(B^4 \sqcup B^4)$ from a link $L_1$ to a second link $L_2$.
Suppose that every component of $\Sigma$ has a boundary component in $L_2$. Then, we have an inequality of $s$-invariants
\[s(L_1) - s(L_2) \geq \chi(\Sigma),\]
where $\chi(\Sigma)$ denotes the Euler characteristic of $\Sigma$.
\end{theorem}
\addtocounter{thm}{-1}
}

By further specializing to the case $L_1 = \varnothing$ (for which $s = 1$), we obtain the following adjunction inequality for $s$.

{
\renewcommand{\thethm}{\ref{cor:adjunction}}
\begin{corollary}[Adjunction inequality for $s$]
Let $W = (\#^t \bCP) \setminus B^4$ for some $t \geq 0$. Let $L \subset \del W= S^3$ be a link, and $\Sigma \subset W$ a properly, smoothly embedded oriented surface with no closed components, such that $\del \Sigma =L$ and $[\Sigma]=0 \in H_2(W, \del W)$. Then 
\[
s(L) \leq 1-\chi(\Sigma).
\]
\end{corollary}
\addtocounter{thm}{-1}
}

The differences between Theorem \ref{thm:adjunction-tau} and Corollary \ref{cor:adjunction} are the following:
\begin{itemize}
\item The negative definite 4-manifold $W$ of Theorem \ref{thm:adjunction-tau} is replaced by the more restrictive case of $(\#^t \bCP) \setminus B^4$;
\item Corollary \ref{cor:adjunction} is only for $[\Sigma] = 0$;
\item Corollary \ref{cor:adjunction} works for \emph{links}, as opposed to knots. (However, Hedden and Raoux \cite{HR} have recently announced a generalization of Theorem~\ref{thm:adjunction-tau} to links.)
\end{itemize}

See Section~\ref{sec:GeneralSurfaces}, and in particular Conjecture~\ref{conj:AdjGeneral} for a discussion of what we expect without the hypothesis $[\Sigma]=0$.

\begin{proof}[{Proof of Theorem \ref{thm:strongadj}}]
By changing orientation, we get $(\overline{Z}, \overline{\Sigma})$, which is now a cobordism from $(\SSr, L_2)$ to $(\SSr, L_1)$, i.e.
\[
\de(\overline{Z}, \overline{\Sigma}) = (-(\SSr, L_2)) \sqcup (\SSr, L_1) = (\overline{\SSr}, \m L_2) \sqcup (\SSr, L_1).
\]
We can view $\overline{Z}$ as the connected sum of $I \times \SSr$ and $\#^t\CP$. Let $C_1, \ldots, C_t$ be the complex projective lines $\mathbb{CP}^1$ sitting in the various copies of $\CP$. We can assume that they intersect $\Sigma$ transversely in a finite number of points. The number of positive intersections matches that of negative intersections, because all algebraic intersection numbers $[\Sigma] \cdot [C_i]$ vanish, since $[\Sigma]=0$.

Consider $\mathbb{CP}^1$ sitting inside $\CP$. The normal bundle $N_{\mathbb{CP}^1|\CP}$ has Euler number $1$, so the sphere bundle $S(N_{\mathbb{CP}^1|\CP})$ is diffeomorphic to $S^3$, and the projection $S^3 \to \mathbb{CP}^1 = S^2$ is the Hopf fibration.

Now remove a neighborhood of each copy of $\mathbb{CP}^1$ from $\overline{Z}$. Let $W$ denote the resulting 4-manifold, and let ${\overline\Sigma}^\circ$ be the punctured surface obtained from $\overline\Sigma$ by removing a neighborhood of each intersection point with the various copies of $\mathbb{CP}^1$ (i.e., a disc near each intersection point). Since $\CP \setminus (N_{\mathbb{CP}^1|\CP}) = B^4$, we have that $W = (I \times \SSr) \# (\#^t B^4)$. Let $S^3_i$ denote the boundary of the $i$-th copy of $B^4$. If the geometric intersection $|\Sigma \cap C_i| = 2p_i$, then $\de{\overline\Sigma}^\circ \cap S^3_i$ consists of $2p_i$ fibers of the Hopf fibration in $S^3$; moreover, $p_i$ fibers must be oriented in one way and $p_i$ must be oriented the other way, since $[\Sigma]\cdot[C_i] = 0$. Thus, $\de{\overline\Sigma}^\circ \cap S^3_i$ is the link $\Fpp{p_i}(1) \subset S^3_i$.

Now choose properly embedded arcs $a_1, \ldots, a_r$ in $W$, such that $a_i$ connects $S^3_i$ to $\{1\} \times \SSr$, and the arcs are pairwise disjoint and disjoint from $\Sigma$. After removing a neighborhood of $a_1, \ldots, a_r$ from $W$, we obtain a 4-manifold $W' \cong I \times \SSr$, and ${\overline\Sigma}^\circ \subset W'$ is a cobordism from $L_2$ to $M:=L_1 \sqcup \Fpp{p_1}(1) \sqcup \cdots \sqcup \Fpp{p_t}(1)$.
Note that each component of ${\overline\Sigma}^\circ$ has boundary in $L_2$ and that
\[
\chi\left({\overline{\Sigma}}^\circ\right) = \chi(\Sigma) - 2(p_1 + \cdots + p_t).
\]
By Theorem \ref{thm:GenusBoundCylinders}, Proposition \ref{prop:disj union}, and Theorem \ref{thm:sLpp}, we have
\begin{align*}
s(L_2) \leq s(M) - \chi \left({\overline{\Sigma}}^\circ\right) &= \left(s(L_1) + s(\Fpp{p_1}) + \cdots + s(\Fpp{p_t}) - t\right) - \left(\chi(\Sigma) - 2(p_1 + \cdots + p_t)\right)\\
&= s(L_1) + \left((1-2p_1) + \cdots + (1-2p_t) - t\right) - \chi(\Sigma) + 2(p_1 + \cdots + p_t)\\
&=s(L_1)-\chi(\Sigma). \qedhere
\end{align*}
\end{proof}

A straightforward application of Corollary \ref{cor:adjunction} is the following.

\begin{corollary}
\label{cor:sHsCP}
If $L$ is strongly H-slice in $\#^t\bCP$, then $s(L) \leq 1-|L|$.
\end{corollary}

There are examples where the inequality of Corollary \ref{cor:sHsCP} is strict. For example, the left-handed trefoil knot $T_{2,-3}$ is strongly H-slice in $\bCP$ (see Example \ref{ex:LHT}) and $s(T_{2,-3}) = -2 < 0$.

If $L$ is a strongly H-slice link in $S^4$, then the inequality of Corollary \ref{cor:sHsCP} is actually an equality. This follows from \cite{beliakova-wehrli}, but for completeness we give a proof below.

\begin{theorem}
\label{thm:sHsS4}
If $L$ is strongly (H-)slice in $S^4$, then $s(L) = 1-|L|$. 
\end{theorem}
\begin{proof}
Let $|L|=\ell$. By hypothesis, $L$ bounds $\ell$ disjoint discs in $B^4$. By removing a small ball from $B^4$, one obtains a cobordism $\Sigma$ in $I \times S^3$ from $L$ to the unknot $U$, with $\chi(\Sigma) = \ell-1$. Every component of $\Sigma$ has boundary in $L$, so, by the case $r=0$ of Theorem \ref{thm:GenusBoundCylinders} or Corollary \ref{cor:adjunction} (originally proved in \cite[Equation (7)]{beliakova-wehrli}), we have
\begin{equation}
\label{eq:sHs1}
s(L) \leq 1-\ell.
\end{equation}
The analogous reasoning applied to $\m L$ shows that
\begin{equation}
\label{eq:sHs2}
s(\m L) \leq 1-\ell.
\end{equation}
By combining Equations \eqref{eq:sHs1} and \eqref{eq:sHs2}, we get $s(L) + s(\m L) \leq 2-2\ell$. The opposite inequality follows from \cite[Equation (10)]{beliakova-wehrli}. Thus, Equations \eqref{eq:sHs1} and \eqref{eq:sHs2} must actually by equalities.
\end{proof}

Using Corollary \ref{cor:sHsCP}, we can generalize the proof of Theorem \ref{thm:sHsS4} to show that the $s$-invariant cannot detect links that are strongly H-slice in some exotic Gluck twist, but not in $S^4$.

{
\renewcommand{\thethm}{\ref{cor:Glucktwist}}
\begin{corollary}
Let $X$ be a homotopy 4-sphere obtained by Gluck twist on $S^4$. If a link $L \subset S^3$ is strongly (H-)slice in $X$, then $s(L) = 1-|L|$.
\end{corollary}
\addtocounter{thm}{-1}
}
\begin{proof}
By a property of Gluck twists, $X \# \CP \cong \CP$ and $X \# \bCP \cong \bCP$; cf. \cite[Exercise 5.2.7(b)]{GS}. Thus, $L$ is strongly H-slice in $\CP$ and in $\bCP$. It follows that $L$ and $\m L$ are strongly H-slice in $\bCP$. By Corollary \ref{cor:sHsCP}, $s(L) \leq 1-\ell$ and $s(\m L) \leq 1-\ell$, where $\ell = |L|$. These are precisely Equations \ref{eq:sHs1} and \ref{eq:sHs2}. We then conclude the proof as in Theorem \ref{thm:sHsS4}.
\end{proof}

\begin{remark}
One could also ask if it is possible to disprove the smooth 4D Poincar\'e conjecture by using the $s$-invariant to find a link $L$ that bounds a surface of a given topological type in a Gluck twist $X$, but cannot bound a surface of the same type in $S^4$. Once again, the answer is negative, because the inequality involving $s$ and $\chi(\Sigma)$ in Corollary~\ref{cor:adjunction} is the same in both $B^4$ and $\bCP \setminus B^4$. 
\end{remark}

\section{$s$-invariants for links in $S^3$}
\label{sec:sS3}

\subsection{A connected sum formula}
Let $(L_1, p_1)$ and $(L_2, p_2)$ be oriented pointed links in $S^3$. Then the connected sum $L_1 \# L_2$ with respect to the basepoints $p_1$ and $p_2$ is a well-defined oriented link. In fact, the isotopy class of $L_1 \# L_2$ depends only on the components of $L_1$ and $L_2$ that are glued together.
For example, the connected sum of oriented knots is well-defined, with no need of specifying the basepoints.

Recall that the orientation is a fundamental piece of data needed to define the connected sum. For instance, even for knots, the isotopy class of the connected sum of two knots $K_1 \# K_2$ may change if one changes the orientation of either $K_1$ or $K_2$.

Beliakova and Wehrli~\cite[Lemma 6.1]{beliakova-wehrli} proved that for pointed links $(L_1, p_1)$ and $(L_2, p_2)$ in $S^3$
\[
s(L_1) + s(L_2) - 2 \leq s(L_1 \# L_2) \leq s(L_1) + s(L_2).
\]
The following theorem completely describes the behavior of the $s$-invariant under connected sum for links in $S^3$. This is a first step towards Theorem \ref{thm:ConnSumSSr}.

\begin{theorem}
\label{thm:cs}
Let $(L_1, p_1)$ and $(L_2, p_2)$ be oriented pointed links in $S^3$. Then
\[
s(L_1 \# L_2) = s(L_1) + s(L_2).
\]
\end{theorem}

\begin{figure}
\resizebox{0.8\textwidth}{!}{
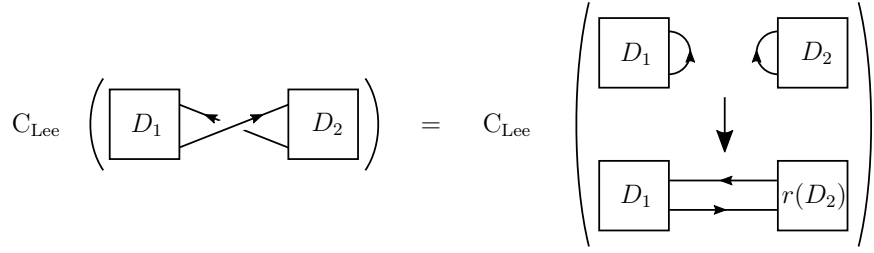
}
\caption{The Lee complex for $D_1 \# D_2$.}
\label{fig:Ras1}
\end{figure}

\begin{figure}
\resizebox{0.8\textwidth}{!}{
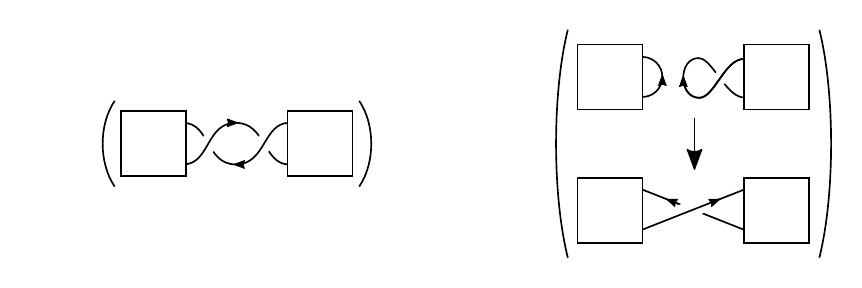
}
\caption{The Lee complex for $D_1 \# r(D_2)$.}
\label{fig:Ras2}
\end{figure}

\begin{proof}
We partially follow Rasmussen's proof for knots~\cite[Proposition 3.12]{Rasmussen}.
Over a field $\F$ (of characteristic other than two), as in \cite[Lemma 3.8]{Rasmussen}, given based links $(L_1,p_1)$ and $(L_2, p_2)$, there is a short exact sequence
\[
0 \longrightarrow \Lh{L_1 \# L_2} \stackrel{p_*}{\longrightarrow} \Lh{L_1} \otimes \Lh{L_2} \stackrel{\de}{\longrightarrow} \Lh{L_1\#r(L_2)} \to 0
\]
where $r(L_2)$ denotes the reverse of $L_2$. See Figure \ref{fig:Ras1}, where $D_1$ and $D_2$ denote diagrams for $L_1$ and $L_2$, respectively, and $r(D_2)$ denotes the diagram $D_2$ with orientation reversed on \emph{all} components. (This operation ensures that the sign of each crossing on $D_2$ matches the corresponding one on $r(D_2)$, so the degree shifts appearing in the definition of Khovanov homology are also the same for $D_2$ and $r(D_2)$.)
As in \cite[Lemma 3.8]{Rasmussen}, $p_*$ and $\de$ are filtered maps of $q$-degree $-1$.
Let $o$ be an orientation of $L_1 \# L_2$, and let $o_1$ and $o_2$ be the orientations $o|_{L_1}$ and $o|_{L_2}$ respectively. Since $o_1$ and $o_2$ are oriented in the same direction near the connected sum point, the Lee generators $\d_{o|{D_1}}$ and $\d_{o|{D_2}}$ have different labels $x+1$ and $x-1$ near the connected sum point.
As in \cite[Proposition 3.12]{Rasmussen}, one can check from the definition that $p_*([\d_{o}]) = [\d_{o|{D_1}}] \otimes [\d_{o|{D_2}}]$. In fact, this is true already on the chain level.
We know that the filtration levels of $[\d_o]$, $[\d_{o_1}]$ and $[\d_{o_2}]$ are $s_{\min}(L_1 \# L_2,o)$, $s_{\min}(L_1,o_1)$, and $s_{\min}(L_2,o_2)$, respectively.
Thus, by the fact that $p_*$ is filtered of degree $-1$, we get
\[
s_{\min}(L_1, o_1) + s_{\min}(L_2, o_2) = s_{\min}(L_1 \sqcup L_2, o_1 \sqcup o_2) \geq s_{\min}(L_1 \# L_2, o) - 1.
\]
From the fact that $s_{\min} = s -1$ we obtain the first inequality
\[
s(L_1, o_1) + s(L_2, o_2) \geq s(L_1 \# L_2, o).
\]
(This was already established \cite{beliakova-wehrli}.)

\begin{figure}
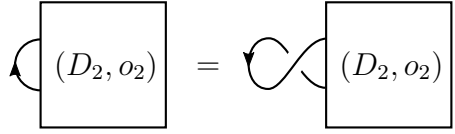
\caption{The same orientation $o_2$ in Figures \ref{fig:Ras1} and \ref{fig:Ras2} points in different directions near the connected sum point.}
\label{fig:Ras3}
\end{figure}
For the other inequality, consider the exact sequence
\[
0 \longrightarrow \Lh{L_1 \# r(L_2)} \stackrel{p_*}{\longrightarrow} \Lh{L_1} \otimes \Lh{L_2} \stackrel{\de}{\longrightarrow} \Lh{L_1\#L_2} \to 0
\]
coming from Figure \ref{fig:Ras2}. In Figure \ref{fig:Ras2}, the orientations $o_1$ and $o_2$ are oriented \emph{in the opposite direction} near the connected sum point, so the corresponding Lee generators are labeled with the same letter near the connected sum point, both $x+1$ or both $x-1$. This is because the clasp changes the local orientation, see Figure \ref{fig:Ras3}. Then, one can check from the definition that $\de([\d_{o_1}] \otimes [\d_{o_2}]) = [\d_o]$, which in fact holds also on the chain level. Note that we also have $\de([\d_{\bar{o_1}}] \otimes [\d_{\bar{o_2}}]) = [\d_{\bar o}]$,
$\de([\d_{\bar{o_1}}] \otimes [\d_{o_2}]) = 0$, and $\de([\d_{o_1}] \otimes [\d_{\bar{o_2}}]) = 0$.
For $i=1,2$, let $\e_i$ be the sign ($+$ or $-$) such that $[\d_{o_i}] + \e_i[\d_{\bar{o_i}}]$ is in filtration level $s_{\max}(L_i, o_i)$. Then, by using the fact that $\de$ is filtered of degree $-1$, we obtain that
\[
\q(\de([\d_{o_1} + \e_1\d_{\bar{o_1}}] \otimes [\d_{o_2} + \e_2\d_{\bar{o_2}}])) \geq s_{\max}(L_1, o_1) + s_{\max}(L_2, o_2) - 1.
\]
We can compute $\de([\d_{o_1} + \e_1\d_{\bar{o_1}}] \otimes [\d_{o_2} + \e_2\d_{\bar{o_2}}]) = [\d_{o} + \e_1 \e_2 \d_{\bar o}]$, so
\[
\q([\d_{o} + \e_1 \e_2 \d_{\bar o}]) \geq s_{\max}(L_1, o_1) + s_{\max}(L_2, o_2) - 1.
\]
The quantum filtration level of $[\d_{o} + \e_1 \e_2 \d_{\bar o}]$ is either $s_{\max}(L_1\# L_2,o)$ or $s_{\min}(L_1\# L_2,o)$. In any case,
\[
\q([\d_{o} + \e_1 \e_2 \d_{\bar o}]) \leq s_{\max}(L_1\# L_2,o).
\]
Thus,
\[
s_{\max}(L_1 \# L_2,o) \geq s_{\max}(L_1, o_1) + s_{\max}(L_2, o_2) - 1.
\]
By using the fact that $s_{\max} = s + 1$, we get the second inequality
\[
s(L_1 \# L_2,o) \geq s(L_1, o_1) + s(L_2, o_2). \qedhere
\]
\end{proof}

\subsection{$s_-$ and $s_+$ in $S^3$}
\begin{definition}
For a link in $S^3$, we define $s_-(L) = s(L)$ and $s_+(L) = -s(\m L)$.
\end{definition}

\begin{example}
\label{ex:knotsinS3}
For a knot $K$ in $S^3$, $s_-(K) = s_+(K)$ is the usual $s$-invariant.
\end{example}

\begin{example}
If $U_\ell$ is the $\ell$-component unlink in $S^3$, then $s_\pm(U_\ell) = \pm(\ell-1)$.
\end{example}

\begin{example}
For the positive Hopf link $\Hopf^+$, $s_\pm (\Hopf^+) = +1$.
\end{example}

We summarize all the properties of the $s$-invariants here. The new results of this paper are items \eqref{it:new1}, \eqref{it:s-ANDs+}, \eqref{it:GCC}, and \eqref{it:new4}.

\begin{proposition}
\label{prop:propertiesS3}
Let $L$ be an $\ell$-component link in $S^3$. Let $L_1$, $L_2$, $L^+$, and $L^-$ be links in $S^3$. Then
\begin{enumerate}
\item \label{it:old1} $s_\pm(L) = s_\pm(r(L)) = -s_\mp(m(L)) = -s_\mp(\m L)$;
\item \label{it:parity} $s_\pm(L) \equiv \ell-1 \pmod 2$;
\item \label{it:new1} $s_\pm(L_1 \# L_2) = s_\pm(L_1) + s_\pm(L_2)$;
\item \label{it:old2} $s_\pm(L_1 \sqcup L_2) = s_\pm(L_1) + s_\pm(L_2) \pm1$;
\item \label{it:s-ANDs+} if $L$ is non-empty, $s_+(L) - 2\ell+2 \leq s_-(L) \leq s_+(L)$;
\item \label{it:old3} if $\Sigma$ is a cobordism from $L_1$ to $L_2$ in $I \times S^3$ such that every component of $\Sigma$ has boundary in $L_1$, then
\[
\pm(s_\pm(L_1) - s_\pm(L_2)) \geq \chi(\Sigma);
\]
\item \label{it:CC} if $L^-$ is obtained from $L^+$ by a crossing change from positive to negative, then
\[
s_\pm(L^-)\leq s_\pm(L^+) \leq s_\pm(L^-) +2;
\]
\item \label{it:GCC} if $L^{\tw}$ is obtained from $L$ by a \emph{generalized} crossing change from positive to negative, then
\[
s_\pm(L^{\tw}) \leq s_\pm(L);
\]
\item \label{it:old4} if $L$ is strongly (H-)slice in $S^4$, then $s_\pm(L) = \pm(|L|-1)$;
\item \label{it:new4} if $L$ is strongly H-slice in both $\#^r\CP$ and $\#^r\bCP$, then $s_\pm(L) = \pm(|L|-1)$.
\end{enumerate}
\end{proposition}

\begin{example}
By taking connected sums of Hopf links and unlinks, one can easily show that $s_- - s_+$ can achieve all even values between $2-2|L|$ and $0$ (see item \eqref{it:s-ANDs+}).
\end{example}

\begin{proof} We discuss each statement separately.
\begin{enumerate}
\item[\eqref{it:old1}] This follows from Theorem \ref{thm:s invt diffeo and rev} and the definitions of $s_-$ and $s_+$.
\item[\eqref{it:parity}] This follows from the fact that the Khovanov complex is contained in odd (resp.~even) quantum degree if $\ell \equiv 0 \pmod2$ (resp.~$\ell \equiv 1 \pmod 2$); cf.~\cite[Proposition 24]{Kh}.
\item[\eqref{it:new1}] This is the content of Theorem \ref{thm:cs}.
\item[\eqref{it:old2}]
For $s_-$, this is the case $k_1=k_2=0$ of Proposition \ref{prop:disj union} (originally proved in \cite[Equation (8)]{beliakova-wehrli}). For $s_+$, use the same equation on the mirror links.
\item[\eqref{it:s-ANDs+}] The inequality $s_+(L) - 2\ell+2 \leq s_-(L)$ is the first part of \cite[Equation (10)]{beliakova-wehrli}. For the second inequality, consider the pointed links $(L, p)$ and $(\m L, p)$. Their connected sum bounds an embedded surface $\Sigma \subset B^4$ consisting of one disc (whose boundary is the connected sum of the components of $L$ and $\m L$ containing the basepoints), and of $\ell-1$ annuli (one for each non-based component of $L$ and the corresponding component of $\m L$). Note that $\chi(\Sigma) = 1$. After puncturing $B^4$ at a point on $\Sigma$, we get a cobordism ${\Sigma}^\circ$ in $I \times S^3$, with $\chi(\Sigma^\circ) = 0$, from the unknot $U$ to $L \# (\m L)$. By turning it upside down and changing its orientation, we can see it as a cobordism from $L\#(\m L)$ to $U$.
Every component of $\Sigma^\circ$ has a boundary component in $L \#(\m L)$, so, by the case $r=0$ of Theorem \ref{thm:GenusBoundCylinders} (originally proved in \cite[Equation (7)]{beliakova-wehrli}), we get
\[
s(L \# (\m L)) \leq s(U_1) - \chi({\Sigma}^\circ) = 0.
\]
By Theorem \ref{thm:cs}, we get $s(L) + s(\m L) \leq 0$, which means $s_-(L) \leq s_+(L)$.
\item[\eqref{it:old3}] This is the case $r=0$ of Theorem \ref{thm:GenusBoundCylinders} (originally proved in \cite[Equation (7)]{beliakova-wehrli}). Note that for the inequality involving $s_+$, one should apply it to $\overline{\Sigma}$ as a cobordism from $\m{L_1}$ to $\m{L_2}$.
\item[\eqref{it:CC}]
The first inequality can be proved by adapting Livingston's argument~\cite[Corollary 3]{Livingston}. Alternatively, it follows immediately from item \eqref{it:GCC} below.
For the second inequality, note that we can resolve the crossing change to obtain a genus-1 cobordism between $L^-$ and $L^+$.
Thus, by item \eqref{it:old3} (applied to the cobordism in both directions), we get $s_\pm(L^+) \leq s_\pm(L^-) +2$.
\item[\eqref{it:GCC}] This will be proved below (see Theorem \ref{thm:addFT}).
\item[\eqref{it:old4}] This is Theorem \ref{thm:sHsS4} applied to both $L$ and $\m L$.
\item[\eqref{it:new4}] The proof is analogous to the one of Corollary \ref{cor:Glucktwist}. \qedhere
\end{enumerate}
\end{proof}

\subsection{Cobordism properties}

The following theorem is a consequence of Corollary \ref{cor:adjunction}.

\begin{theorem}
\label{thm:positivity}
Consider $Z = (\#^t\bCP)\setminus(B^4 \sqcup B^4)$ as a cobordism from $S^3$ to $S^3$. Suppose that $(Z, \Sigma)$ is a cobordism from $(S^3, L_1)$ to $(S^3, L_2)$, such that
\begin{itemize}
\item $\Sigma$ has no closed components;
\item $[\Sigma] = 0$ in $H_2(Z, \de Z)$;
\item there is a component of $\Sigma$ with boundary both in $L_1$ and in $L_2$.
\end{itemize}
Then
\[
s_-(L_2) \leq s_+(L_1) - \chi(\Sigma).
\]
\end{theorem}

Note that the inequality of Theorem \ref{thm:positivity} mixes $s_-$ and $s_+$, but the hypothesis on connectedness is different from that of Theorem \ref{thm:strongadj}.

\begin{proof}[Proof of Theorem \ref{thm:positivity}]
Pick an arc on $\Sigma$ connecting $L_1$ and $L_2$, and remove a regular neighborhood of this arc from $Z$ and from $\Sigma$. What we get is the 4-manifold $W = (\#^t\bCP) \setminus B^4$, with a surface $\widetilde\Sigma$ embedded in it, with $\de\widetilde\Sigma = L_2 \# (-L_1)$. 
A Mayer-Vietoris argument implies that $[\widetilde\Sigma] = 0$, so we can apply Corollary \ref{cor:adjunction} to obtain
\[
s(L_2 \# (\m{L_1})) \leq 1-\chi(\widetilde\Sigma).
\]
To conclude, we note that, by Theorem \ref{thm:cs}, $s(L_2 \# (\m{L_1})) = s(L_2) + s(\m{L_1}) = s_-(L_2) - s_+(L_1)$, and by a simple computation $\chi(\widetilde\Sigma) = \chi(\Sigma)+1$.
\end{proof}

Under some hypothesis on connectedness, we can strengthen the inequality in Theorem \ref{thm:positivity} by having $s_-$ on both sides or $s_+$ on both sides.

\begin{theorem}
\label{thm:weakadjpm}
Consider a null-homologous oriented cobordism $\Sigma\subset Z = (\#^t\bCP) \setminus (B^4 \sqcup B^4)$ from a link $L_1$ to a second link $L_2$. Then:
\begin{itemize}
\item if every component of $\Sigma$ has a boundary component in $L_2$, we have $s_-(L_1) - s_-(L_2) \geq \chi(\Sigma)$;
\item if every component of $\Sigma$ has a boundary component in $L_1$, we have $s_+(L_1) - s_+(L_2) \geq \chi(\Sigma)$.
\end{itemize}
\end{theorem}
\begin{proof}
The inequality for $s_-$ is exactly the statement of Theorem \ref{thm:weakadj}.
For the second inequality, note that $(Z, \Sigma)$ (with the same orientation) can also be seen as a cobordism from $\m L_2$ to $\m L_1$. (This operation is usually referred to as `turning the cobordism upside down'.) Then such a cobordism satisfies the connectedness hypothesis of Theorem \ref{thm:weakadj}, which gives the inequality
\[
s(\m L_2) - s(\m L_1) \geq \chi(\Sigma),
\]
which is equivalent to $s_+(L_1) - s_+(L_2) \geq \chi(\Sigma)$.
\end{proof}

As a corollary of Theorem \ref{thm:weakadjpm} we can prove Theorem \ref{thm:addFT} from the Introduction (but now stated for both $s_-$ and $s_+$).
Recall that adding a generalized \emph{negative} crossing corresponds to adding a \emph{positive} full twist on $n$ strands going up and $n$ strands going down, see Figure \ref{fig:GCC}.

{
\renewcommand{\thethm}{\ref{thm:addFT}}
\begin{theorem}
Suppose that a link $L^\tw \subset S^3$ is obtained from $L$ by adding a generalized negative crossing. Then $s_\pm(L^\tw) \leq s_\pm(L)$.
\end{theorem}
\addtocounter{thm}{-1}
}

\begin{proof}
By Remark \ref{rem:GCC}, there is a connected cobordism $\Sigma$ in a doubly punctured $\bCP$ from $L$ to $L^\tw$, with $\chi(\Sigma) = 0$. Then, by Theorem \ref{thm:weakadjpm},
\[
s_\pm(L^\tw) \leq s_\pm(L). \qedhere
\]
\end{proof}

\subsection{Whitehead doubles}

\begin{figure}
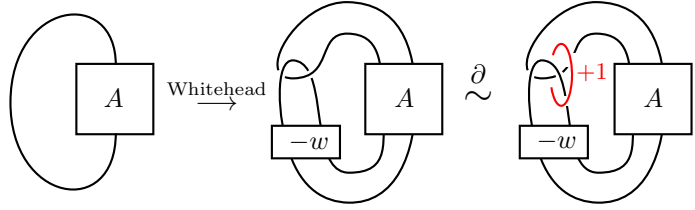
\caption{How to draw a diagram $B$ of $\Wh^+(K,0)$ from a diagram $A$ of $K$. A crossing change from positive to negative on $B$ turns $B$ into a diagram for the unknot.}
\label{fig:WhiteheadDoubleA}
\end{figure}

{
\renewcommand{\thethm}{\ref{thm:sWh+=0}}
\begin{theorem}
Let $K$ be a knot such that $n$ crossing changes from negative to positive turns $K$ into the unknot. Then $\Wh^+(K, 0)$ is $H$-slice in $\CP$ and in $\#^n\bCP$. Thus, $s(\Wh^+(K, 0)) = 0$.
\end{theorem}
\addtocounter{thm}{-1}
}
\begin{proof}
A single crossing change in the Whitehead clasp turns $\Wh^+(K,0)$ into the unknot. The crossing change is from positive to negative (since we are using the \emph{positive} Whitehead pattern), see Figure \ref{fig:WhiteheadDoubleA}. Thus, $\Wh^+(K,0)$ is obtained from a diagram of the unknot by a (generalized) positive crossing change, so, in view of Remark \ref{rem:Ltw}, $\Wh^+(K,0)$ is $H$-slice in $\CP$.

Given a diagram $A$ with writhe $w$ for $K$, then the $0$-framed Whitehead double is given by taking the blackboard-framed double of $A$, and adding $-w$ full twists, see Figure \ref{fig:WhiteheadDoubleA}. Call $B$ this diagram of $\Wh^+(K,0)$. Let $K^+$ (resp.~$A^+$) be obtained from $K$ (resp.~$A$) by changing a negative crossing into a positive crossing.
Then adding a generalized positive crossing on $B$, in correspondence to the crossing change from $A$ to $A^+$, turns $B$ into a diagram for $\Wh^+(K^+, 0)$: this is because the additional two negative twists coming from the generalized positive crossing are compensated by the change in the writhe: $\wr(A^+) = \wr(A)+2$. See Figure \ref{fig:WhiteheadDoubleB}.

By iterating this procedure on the $n$ given crossing changes, we can turn $\Wh^+(K,0)$ into $\Wh^+(U,0) = U$ with $n$ generalized positive crossings. In other words, a diagram of $\Wh^+(K,0)$ is obtained from one of $U$ by adding $n$ generalized \emph{negative} crossings. Thus, by Remark \ref{rem:Ltw}, $\Wh^+(K,0)$ is $H$-slice in $\#^n\bCP$.

Since $\Wh^+(K,0)$ is $H$-slice in both $\CP$ and $\#^n\bCP$, by part \eqref{it:new4} of Proposition \ref{prop:propertiesS3}, we have $s(\Wh^+(K,0)) = 0$.
\end{proof}

\begin{figure}
\[ \WhiteheadDoubleAandB \]
\caption{Starting from a diagram $A$ for $K$, we arrive at a diagram $A^+$ for $K^+$ by adding a generalized positive crossing (GPC) and cancelling two crossings, which then gives rise to a diagram for $\Wh^+(K^+,0)$.  On the other hand, we may apply Whitehead doubling (Wh) to $A$ to arrive at a diagram $B$ for $\Wh^+(K,0)$, to which we may add a generalized positive crossing to arrive at an equivalent diagram for $\Wh^+(K^+,0)$.}
\label{fig:WhiteheadDoubleB}
\end{figure}

\section{$s$-invariants in $\SSr$}
\label{sec:sMr}

\subsection{The invariants $s_-$ and $s_+$ in $\SSr$}

\begin{definition}
For a null-homologous link $L$ in $M = \#^r(\SSone)$, we define $s_-(L) = s(L)$ and $s_+(L) = -s(\m L)$.
\end{definition}

As a consequence of Theorem \ref{prop:s invt by finite approx}, the $s$-invariants $s_\pm$ can be computed via finite approximation.

\begin{proposition}
\label{prop:stab}
Let $D$ be a diagram for a null-homologous link $L$ in $M = \#^r(\SSone)$. Let $k^\pm = \ceiling{\frac{n^\pm_D+2}{2}}$ and let $\vec k = (k, \ldots, k)$. Then
\begin{enumerate}
\item \label{it:stab1} $s_-(L) = s_-(D(\vec k))$ for all $k \geq k^+$;
\item \label{it:stab2} $s_+(L) = s_+(D(-\vec k))$ for all $k \geq k^-$;
\item \label{it:stab3} $s_-(L) \leq s_-(D(\vec0))$;
\item \label{it:stab4} $s_+(D(\vec0)) \leq s_+(L)$.
\end{enumerate}
\end{proposition}

\begin{proof}
Item \eqref{it:stab1} is Theorem \ref{prop:s invt by finite approx}. For item \eqref{it:stab2}, consider the mirror diagram $\m D$, obtained by reversing all crossings in $D$. The number of positive crossings of $\m D$ is $n_D^-$. Moreover, it is immediate to check that $(\m D)(\vec k) = \m{(D(-\vec k))}$. Then \eqref{it:stab2} follows from Theorem \ref{prop:s invt by finite approx} applied to $\m D$.

For item \eqref{it:stab3}, note that there is a null-homologous cobordism in $\#^{kr}\bCP$ from $D(\vec0)$ to $D(\vec k)$, consisting of $|L|$ disjoint cylinders. By Theorem \ref{thm:weakadjpm} we have $s_-(D(\vec k))\leq s_-(D(\vec0))$. By item \eqref{it:stab1} we know that $s_-(D(\vec k)) = s_-(L)$ for $k \gg 0$, and thus conclude.

Lastly, item \eqref{it:stab4} follows in the same way as item \eqref{it:stab3}, except that the cobordism is now from $D(-\vec k)$ to $D(\vec 0)$ and we use item \eqref{it:stab2} to conclude.
\end{proof}

\begin{example}
Suppose that $L$ is a local link, i.e.~a link contained in a 3-ball: $L \subset B^3 \subset M$. Let $L_{S^3}$ be the corresponding link in $S^3$, obtained by viewing $B^3 \subset S^3$. Then $s_\pm(L) = s_\pm(L_{S^3})$.
To see this, note that any local link $L$ can be represented by a diagram $D$ that does not go through the handles. Thus, for all $k$, $D(\vec k) = D$ is always the same diagram, which could also be viewed as a diagram for $L_{S^3}$. Thus,
\[
s_\pm(L) = s_\pm(D) = s_\pm(L_{S^3}).
\]
\end{example}

\begin{example}
For the link $\Fp \subset \SSone$ discussed in Section \ref{sec:Lpp}, we have $s_\pm(\Fp) = \pm(2p-1)$. The computation of $s_-$ is done in Theorem \ref{thm:s(Fp)}. For $s_+$, note that $\m{\Fp} = \Fp \subset \SSone$, so $s_+(\Fp) = -s_-(\Fp)$.
\end{example}

\begin{figure}
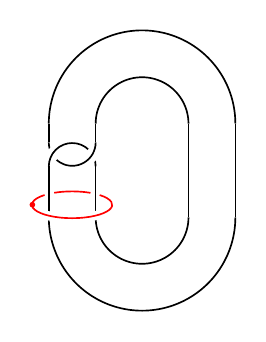
\caption{The positive Whitehead knot $\Wh^+$ in $\SSone$.}
\label{fig:Whitehead}
\end{figure}

\begin{example}
\label{ex:Whitehead}
Let $\Wh^+$ be the positive Whitehead knot in $\SSone$, with diagram $D_{\Wh^+}$ represented in Figure \ref{fig:Whitehead}. Then $s_-(\Wh^+) = 0$ and $s_+(\Wh^+) = 2$. Indeed, for $k \geq 0$, the finite approximation $D_{\Wh^+}(k) \subset S^3$ is the twist knot with an alternating diagram with $2k+2$ crossings. For example, $D_{\Wh^+}(1)$ is the figure-eight knot. All these knots have vanishing signature and, since they are alternating, vanishing $s$-invariant too. It follows that $s_-(\Wh^+) = 0$ by Proposition \ref{prop:stab}.\eqref{it:stab1}. If instead we consider $D_{\Wh^+}(-k)$, for $k<0$, we obtain the positive twist knot with an alternating diagram with $2k+1$ crossings. For instance, $D_{\Wh^+}(-1)$ is the right-handed trefoil knot $T_{2,3}$. All these knots have signature $-2$, thus their $s$-invariant is $2$. It follows by Proposition \ref{prop:stab}.\eqref{it:stab2} that
\[
s_+(\Wh^+) = s_+(D_{\Wh^+}(-k)) = s(D_{\Wh^+}(-k)) = 2,
\]
where in the second equality we used the fact that for knots in $S^3$ we have $s_-(K) = s_+(K) = s(K)$ (see Example \ref{ex:knotsinS3}).
\end{example}

\begin{example}
For the negative Whitehead knot $\Wh^-$ in $\SSone$, which is a mirror image of $\Wh^+$, we have $s_-(\Wh^-) = -2$ and $s_+(\Wh^-) = 0$. This follows from Example \ref{ex:Whitehead} and the definition of $s_{\pm}$.
\end{example}

\begin{remark}
Recall that for a knot $K$ in $S^3$ we have $s_-(K) = s_+(K)$ (see Example \ref{ex:knotsinS3}). This is no longer true for (null-homologous) knots in $\SSone$, since we saw in Example \ref{ex:Whitehead} that the positive Whitehead knot in $\SSone$ has $s_-(\Wh^+)=0$ and $s_+(\Wh^+)=2$.
\end{remark}

\subsection{Properties from $S^3$}

We saw in Section \ref{sec:disj union} that given oriented links $L_1$ and $L_2$ in $\#^{r_1}(\SSone)$ and $\#^{r_2}(\SSone)$ respectively, we can define $L_1 \sqcup L_2 \subset \#^{r_1+r_2}(\SSone)$. In a similar way, given $(L_1,p_1)$ and $(L_2, p_2)$ pointed oriented links in $\#^{r_1}(\SSone)$ and $\#^{r_2}(\SSone)$ respectively, we can define $L_1 \# L_2 \subset \#^{r_1+r_2}(\SSone)$ by performing connected sum near the basepoints.

Suppose that $L$ is a null-homologous link in $M = \SSr$, and $K \subset M$ is an unknot, with $\lk(L, K) = 0$. Then a \emph{generalized crossing change from positive to negative} (also called \emph{adding a generalized negative crossing}) on $L$ along $K$ consists of adding a positive full twist on $L$ on a disc $D$ with $\del D = K$. See Figure \ref{fig:GCC}. The resulting link $L^\tw$ is well-defined up to diffeomorphisms of $M$.

\begin{proposition}
\label{prop:propertiesS1S2}
Let $L$, $L^{\tw}$, $L^+$, and $L^-$ be null-homologous $\ell$-component links in $M = \#^r(\SSone)$.  Let $L_1$, $L_2$ be null-homologous links in $\#^{r_1}(\SSone)$ and $\#^{r_2}(\SSone)$, respectively. Then
\begin{enumerate}
\item \label{it:propertiesS1S2-1} $s_\pm(L) = s_\pm(r(L)) = -s_\mp(m(L)) = -s_\mp(\m L)$;
\item \label{it:parityS1S2} $s_\pm(L) \equiv \ell-1 \pmod 2$;
\item \label{it:propertiesS1S2-2} $s_\pm(L_1 \# L_2) = s_\pm(L_1) + s_\pm(L_2)$;
\item \label{it:propertiesS1S2-3} $s_\pm(L_1 \sqcup L_2) = s_\pm(L_1) + s_\pm(L_2) \pm1$;
\item \label{it:propertiesS1S2-4} if $L$ is non-empty, $s_-(L) \leq s_+(L)$;
\item \label{it:propertiesS1S2-CC} if $L^-$ is obtained from $L^+$ by a crossing change from positive to negative, then
\[
s_\pm(L^-)\leq s_\pm(L^+) \leq s_\pm(L^-) +2;
\]
\item \label{it:propertiesS1S2-GCC} if $L^{\tw}$ is obtained from $L$ by a \emph{generalized} crossing change from positive to negative, then
\[
s_\pm(L^{\tw}) \leq s_\pm(L).
\]
\end{enumerate}
\end{proposition}

Note that Property \eqref{it:propertiesS1S2-2} is Theorem \ref{thm:ConnSumSSr}.

\begin{proof}
Property \eqref{it:propertiesS1S2-1} is immediate from the definition.
Properties \eqref{it:parityS1S2}, \eqref{it:propertiesS1S2-2}, \eqref{it:propertiesS1S2-3}, \eqref{it:propertiesS1S2-CC}, and \eqref{it:propertiesS1S2-GCC} follow from the corresponding properties in $S^3$ (see Proposition \ref{prop:propertiesS3}) via the stabilization properties from Proposition \ref{prop:stab}. (Note that there is a numbering shift: properties of \eqref{it:propertiesS1S2-CC} and \eqref{it:propertiesS1S2-GCC} here correspond to properties \eqref{it:CC} and \eqref{it:GCC} in Proposition \ref{prop:propertiesS3}).

For Property \eqref{it:s-ANDs+}, take a diagram $D$ of $L \subset M$, and suppose that $\vec k$ is big enough so that, by Proposition \ref{prop:stab}, items \eqref{it:stab1} and \eqref{it:stab2}, we have
\begin{equation*}
s_-(L) = s_-(D(\vec k))\qquad \text{and} \qquad s_+(L) = s_+(D(-\vec k)).
\end{equation*}
Then, by Theorem \ref{thm:addFT} and Proposition \ref{prop:propertiesS3}.\eqref{it:s-ANDs+} we have
\[
s_-(L) = s_-(D(\vec k)) \leq s_-(D(-\vec k)) \leq s_+(D(-\vec k)) = s_+(L). \qedhere
\]
\end{proof}

\begin{remark}
Property \eqref{it:s-ANDs+} for links in $S^3$ contained two inequalities (see Proposition \ref{prop:propertiesS3}). The other inequality, namely $s_+(L) - 2\ell+2 \leq s_-(L)$ is no longer true in $\SSone$. For example, the positive Whitehead knot in $\SSone$ has $\ell=1$, $s_+ = 2$ and $s_-=0$. See Corollary \ref{cor:Hedden} for the relevant statement.
\end{remark}

\subsection{Cobordism properties}

The following result is an extension of Theorem \ref{thm:weakadjpm} to $\SSr$.

\begin{theorem}
\label{thm:strongadjpm}
Consider an oriented cobordism $\Sigma\subset Z = (I \times \SSr) \#(\#^t\bCP)$ from a null-homologous link $L_1$ to a second null-homologous link $L_2$, with $[\Sigma] = 0$ in $H_2(Z, \de Z)$.
\begin{itemize}
\item If every component of $\Sigma$ has a boundary component in $L_2$, we have $s_-(L_1) - s_-(L_2) \geq \chi(\Sigma)$;
\item if every component of $\Sigma$ has a boundary component in $L_1$, we have $s_+(L_1) - s_+(L_2) \geq \chi(\Sigma)$.
\end{itemize}
\end{theorem}

\begin{proof}
The proof is exactly the same as that of Theorem \ref{thm:weakadjpm}, except that we now use Theorem \ref{thm:strongadj} instead of Theorem \ref{thm:weakadj}.
\end{proof}


\begin{corollary}
\label{cor:cobS1S2}
Consider an oriented cobordism $\Sigma\subset Z = (I \times \SSr)$ from a null-homologous link $L_1$ to a second null-homologous link $L_2$.
\begin{itemize}
\item If every component of $\Sigma$ has a boundary component in $L_2$, we have $\mp(s_\pm(L_1) - s_\pm(L_2)) \geq \chi(\Sigma)$;
\item if every component of $\Sigma$ has a boundary component in $L_1$, we have $\pm(s_\pm(L_1) - s_\pm(L_2)) \geq \chi(\Sigma)$;
\item if every component of $\Sigma$ has a boundary component in $L_1$ as well as a boundary component in $L_2$, we have $|s_\pm(L_1) - s_\pm(L_2)| \leq -\chi(\Sigma)$.
\end{itemize}
\end{corollary}

\begin{proof}
First, let's prove that $[\Sigma] = 0$ in $H_2(Z, \de Z)$. Consider the long exact sequence
\[
H_2(\de Z) \xrightarrow{\iota} H_2(Z) \xrightarrow{} H_2(Z, \de Z) \xrightarrow{\de_*} H_1(\de Z)
\]
The map $\iota$ is surjective, since $Z$ is a strong deformation retract onto any of its boundary components. Thus, by exactness, $\de_*$ is injective. We know that $\de\Sigma$ consists of null-homologous links, so $[\de\Sigma] = \de_*([\Sigma]) = 0$, and by injectivity $[\Sigma] = 0$ in $H_2(Z, \de Z)$.

We now prove the first bullet point. Once we know that $[\Sigma] = 0$, the inequality $s_-(L_1) - s_-(L_2) \geq \chi(\Sigma)$ is exactly the case $t=0$ of Theorem \ref{thm:strongadjpm}. For the other inequality, we change orientation to the cobordism: $(\overline{Z}, \overline{\Sigma})$ is a cobordism from $\m L_1$ to $\m L_2$. By applying Theorem \ref{thm:strongadjpm} to this new cobordism we get the other inequality.

The second bullet point is proved in the same way. The third bullet point follows immediately from the first two.
\end{proof}

\begin{remark}
The extra inequalities of Corollary \ref{cor:cobS1S2} do not hold under the hypotheses of Theorem \ref{thm:strongadjpm}. The reason is that $\bCP$ is not sent to a diffeomorphic copy of itself under orientation reversal.
For example, by Remark \ref{rem:GCC}, in $\bCP$ there is a connected cobordism $\Sigma$ from the unknot $U$ to the left handed trefoil $T_{2,-3}$, with $\chi(\Sigma) = 0$. Since $s_\pm(U) = 0$ and $s_\pm(T_{2,-3}) = -2$, it is immediate to see that the extra inequalities do not hold in this case.
\end{remark}

The following corollary is due to Matthew Hedden.

\begin{corollary}
\label{cor:Hedden}
Let $L \subset \SSr$ be an $\ell$-component null-homologous link.
If the total geometric intersection of $L$ with the co-cores $\coprod \{p\} \times S^2$ is $k$, then
\[
(s_+-s_-)(L) \leq 2(\ell+k-1).
\]
\end{corollary}

\begin{proof}
We argue by induction on $k$. If $k=0$, the link is local, and the statement is Proposition \ref{prop:propertiesS3}.\eqref{it:s-ANDs+}. For the induction step, a saddle cobordism brings the link $L$ to a link $L'$ with geometric intersection $\leq k-2$, and the number of components of $L'$ is $\ell \pm 1$. By Corollary \ref{cor:cobS1S2} and by the induction step applied to $L'$ we get
\begin{align*}
(s_+-s_-)(L) &\leq (s_+-s_-)(L') + 2 \\
&\leq 2((\ell+1)+(k-2)-1) + 2 \\
&= 2(\ell + k - 1). \qedhere
\end{align*}
\end{proof}

The manifold $\SSone$ has two natural fillings, namely $S^1 \times B^3$ and $B^2 \times S^2$. For $M = \#^r(\SSone)$, we consider the fillings $\natural^r(S^1 \times B^3)$ and $\natural^r (B^2 \times S^2)$.

\begin{definition}
\label{def:gSD}
Let $L$ be a null-homologous link in $M = \#^r(\SSone)$. The \emph{$(S^1 \times B^3)$-genus} $\gSD(L)$ is the minimum genus of a properly embedded oriented surface $\Sigma \subset \natural^r(S^1 \times B^3)$ with $\de \Sigma = L \subset M$.
\end{definition}

Note that Definition \ref{def:gSD} does not make sense for links that are homologically essential, since $H_1(M) \cong H_1(\natural^r(S^1 \times B^3))$, so a homologically essential link does not bound any surface in $\natural^r(S^1 \times B^3)$.

\begin{definition}
\label{def:gDS}
Let $L$ be a link in $M = \#^r(\SSone)$. The \emph{$(B^2 \times S^2)$-genus} $\gDS(L)$ is the minimum genus of a properly embedded surface $\Sigma \subset \natural^r(B^2 \times S^2)$ with $\de \Sigma = L \subset M$.
\end{definition}
Note that in Definition \ref{def:gDS} we do not require $[\Sigma] = 0$, even when $L$ is null-homologous. 

\begin{remark}
\label{rem:almostnullhomologous}
Suppose that $L \subset M$ is null-homologous and that $\Sigma \subset X:=\natural^r(B^2 \times S^2)$ is a properly embedded surface with $\de\Sigma = L$. By considering the long exact sequence of the pair $(X, \de X)$
\[
\cdots \to H_2(X) \to H_2(X, \de X) \to H_1(\de X) \to \cdots
\]
it is immediate to check that $[\Sigma] \in H_2(X, \de X)$ is a linear combination of the elements of the kind $(\set{0} \times S^2)_i$, where $(\set{0} \times S^2)_i$ denotes the core of the $i$-th $B^2 \times S^2$ summand of $X$. As a consequence, for every $i=1, \ldots, r$,
\[
[\Sigma] \cdot [(\set{0} \times S^2)_i] = 0.
\]
\end{remark}

Having set up the definition, we now show that $s_-$ gives a lower bound to $\gDS$ and $s_+$ gives a lower bound to $\gSD$.

{
\renewcommand{\thethm}{\ref{thm:genus bounds}}
\begin{theorem}
Let $L \neq \varnothing$ be a null-homologous $\ell$-component link in $M = \#^r(\SSone)$. Then
\begin{align*}
s_-(L) &\leq 2\gDS(L) + \ell - 1\\
\rotatebox[origin=c]{-90}{$\leq$} \hspace{2ex} & \hspace{10ex} \rotatebox[origin=c]{-90}{$\leq$} \\
s_+(L) &\leq 2\gSD(L) + \ell - 1
\end{align*}
\end{theorem}
\addtocounter{thm}{-1}
}

Hedden and Raoux~\cite{HR} proved an analogue of Theorem \ref{thm:genus bounds} for the $\tau$ invariants from Heegaard Floer homology.

\begin{example}
\label{ex:Whitehead-genus}
The positive Whitehead knot $\Wh^+$ in Figure \ref{fig:Whitehead} has $s_-(\Wh^+) = 0$ and $s_+(\Wh^+) = 2$. Thus, $\gDS(\Wh^+) \geq 0$ and $\gSD(\Wh^+) \geq 1$. It is not difficult to see that both bounds are sharp.
\end{example}

\begin{example}
\label{ex:Lpp-genus}
The link $\Fp$ in Figure \ref{fig:Fp} has $2p$ components and $s_\pm(\Fp) = \pm(2p-1)$. Thus, Theorem \ref{thm:genus bounds} does not tell anything. Note that $\Fp = S^1 \times \coprod \{x_i\} $ is strongly slice in $B^2 \times S^2$, because it bounds an embedded surface consisting of the $2p$ discs $B^2 \times \{x_i\}$. As for $S^1 \times B^3$, after attaching $p$ bands one get the $p$-component unlink $U_p$. Thus, $\Fp$ bounds an embedded surface consisting of $p$ disjoint annuli in $S^1 \times B^3$, hence it is weakly slice.
\end{example}

We break the proof of Theorem \ref{thm:genus bounds} into some lemmas.

\begin{lemma}
\label{lem:genus bound DS}
Let $L$ be a null-homologous $\ell$-component link in $M = \#^r(\SSone)$, and let $\Sigma$ be a properly embedded surface in $\natural^r(B^2 \times S^2)$, with $\sigma$ components, none of which is closed, and with $\de \Sigma = L$. Then
\[
s_-(L) \leq 1- \chi(\Sigma) = 2g(\Sigma) + \ell - 2\sigma + 1.
\]
\end{lemma}

\begin{proof}
We consider the following model for $\natural^r(B^2 \times S^2)$. Suppose that we have $r$ copies of $B^2 \times S^2$, denoted by $(B^2 \times S^2)_1, \ldots, (B^2 \times S^2)_r$. For $i=1, \ldots, r-1$, identify the boundary of $(B^2 \times S^2)_i$ near $(1, N)$ with that of $(B^2 \times S^2)_{i+1}$ near $(1, S)$. Here $S$ and $N$ denote the north and south poles of $S^2$, respectively, and $1 \in B^2$ is the point defined from seeing $B^2 \subset \C$.
We now describe a CW complex $A$ inside the aforementioned model of $\natural^r(B^2 \times S^2)$:
\[
A = \bigcup_{i=1}^r (\set0 \times S^2)_i \cup \bigcup_{i=1}^{r-1} ([0,1] \times \set{N})_i \cup \bigcup_{i=2}^r ([0,1] \times \set{S})_i
\]
Loosely speaking, $A$ is the union of the $S^2$-cores of all the copies of $B^2 \times S^2$, joined together with straight arcs. The complement of a neighborhood $\nbd(A)$ of $A$ in $\natural^r(B^2 \times S^2)$ is $I \times \#^r(\SSone)$.

Now consider $\Sigma$ as in the statement of the lemma. By transversality we can assume that $\Sigma$ is disjoint from the 1-skeleton of $A$, and that therefore it avoids all the arcs connecting the various copies of $\set0 \times S^2$. Moreover, by Remark \ref{rem:almostnullhomologous}, the number of positive and negative intersection points between $\Sigma$ and $(\set0 \times S^2)$ is the same. Thus, after removing $\nbd(A)$, we get a cobordism $\Sigma^\circ$ in $I \times M$ from $L$ to a disjoint union of links $\Fpp{p_i} \subset (\SSone)_i$.
Every component of $\Sigma^\circ$ has a boundary component in $L$, so by Theorems \ref{cor:cob inequality} and \ref{thm:s(Fp)}, and Proposition \ref{prop:propertiesS1S2}.\eqref{it:propertiesS1S2-3}, we have that
\begin{equation}
\label{eq:u8Gt2c3<4Na} 
s_-(L) \leq s_-\left(\coprod_{i=1}^r \Fpp{p_i}\right) - \chi(\Sigma^\circ) = 1-\sum_{i=1}^r 2 p_i - \chi(\Sigma^\circ).
\end{equation}
The cobordism $\Sigma^\circ$ has $\sigma$ components, genus $g$, and $\ell+\sum_{i=1}^r 2 p_i$ punctures. Thus,
\[
\chi(\Sigma^\circ) = 2\sigma - 2g - \ell - \sum_{i=1}^r 2 p_i.
\]
We conclude by substituting the formula for $\chi(\Sigma^\circ)$ in Equation \eqref{eq:u8Gt2c3<4Na}.
\end{proof}

\begin{lemma}
\label{lem:isotopeSigma}
Let $\Sigma$ be a properly embedded surface in $\natural^r(S^1 \times B^3)$. Then it can be isotoped rel boundary to a properly embedded surface in $I \times \#^r(\SSone) \subset \natural^r(S^1 \times B^3)$.
\end{lemma}
\begin{proof}
We consider the following model for $\natural^r(S^1 \times B^3)$. Suppose that we have $r$ copies of $S^1 \times B^3$, denoted by $(S^1 \times B^3)_1, \ldots, (S^1 \times B^3)_r$. For $i=1, \ldots, r-1$, identify the boundary of $(S^1 \times B^3)_i$ near $(1, N)$ with that of $(S^1 \times B^3)_{i+1}$ near $(-1, N)$. Here $S$ and $N$ denotes the north pole of $S^2 =\de B^3$. Let $a_N$ be the straight arc in $B^3$ connecting $0$ to $N$.
We now describe a CW complex $B$ inside the aforementioned model of $\natural^r(S^1 \times B^3)$:
\[
B = \bigcup_{i=1}^r (S^1 \times \set0)_i \cup \bigcup_{i=1}^{r-1} (\set1 \times a_N)_i \cup \bigcup_{i=2}^r (\set{-1} \times a_N)_i
\]
Loosely speaking, $B$ is the union of the $S^1$-cores of all the copies of $S^1 \times B^3$, joined together with straight arcs. The complement of a neighborhood $\nbd(B)$ of $B$ in $\natural^r(S^1 \times B^3)$ is $I \times \#^r(\SSone)$.

Observe that $B$ is a CW complex consisting of $0$- and $1$-cells, so by transversality we can assume that $\Sigma$ misses $B$. Thus, $\Sigma$ is isotopic to a surface in the complement of $\nbd(B)$, that is $I \times \#^r(\SSone)$.
\end{proof}

\begin{proof}[Proof of Theorem \ref{thm:genus bounds}]
The fact that $s_-(L) \leq s_+(L)$ is the content of Proposition \ref{prop:propertiesS1S2}\eqref{it:propertiesS1S2-4}.

The inequality $\gDS(L) \leq \gSD(L)$ follows from Lemma \ref{lem:isotopeSigma}. If $\Sigma$ is an embedded surface in $\natural^r(S^1 \times B^3)$, then we can isotope it to an embedded surface in $I \times M \subset \natural^r(S^1 \times B^3)$ without changing its boundary. However, $I \times M$ also sits in $\natural^r(B^2 \times S^2)$, so the same surface $\Sigma$ can also be embedded in $\natural^r(B^2 \times S^2)$.

Let's now consider the inequality $s_-(L) \leq 2\gDS(L) + \ell - 1$. Let $\Sigma$ be a properly embedded surface in $\natural^r(B^2 \times S^2)$ with $\de \Sigma = L$. If we discard all closed components of $\Sigma$, then the genus does not increase, so without loss of generality we can suppose that $\Sigma$ has no closed components. Then, by Lemma \ref{lem:genus bound DS},
\[
s_-(L) \leq 2g(\Sigma) + \ell - 2\sigma + 1 \leq 2g(\Sigma) + \ell - 1,
\]
since $\Sigma$ has at least one component.

Lastly, let's focus on the inequality $s_+(L) \leq 2\gSD(L) + \ell - 1$. Let $\Sigma$ be a properly embedded surface in $\natural^r(S^1 \times B^3)$ with $\de \Sigma = L$, and let $\sigma$ be the number of its components. By Lemma \ref{lem:isotopeSigma} we can suppose that $\Sigma$ sits in $I \times M$. After puncturing each component of $\Sigma$ we get a cobordism $\Sigma^\circ$ from the unlink $U_\sigma$ to $L$. By construction, each component of $\Sigma^\circ$ has a boundary component in $U_\sigma$. The Euler characteristic of $\Sigma^\circ$ is
\[
\chi(\Sigma^\circ) = 2\sigma - 2g(\Sigma) - \ell - \sigma.
\]
By Corollary \ref{cor:cobS1S2} (second bullet point, inequality involving $s_+$) we have
\[
(\sigma-1) - s_+(L) \geq 2\sigma - 2g(\Sigma) - \ell - \sigma.
\]
By simplifying we get the desired conclusion.
\end{proof}

We conclude by noticing that $s_- + s_+$ shares the same ``homomorphism'' properties of $s$ for knots in $S^3$.

\begin{corollary}
\label{cor:shomomorphism}
For a link $L \subset M$, define $\sav(L) := \frac{s_-(L)+s_+(L)}2$. Then:
\begin{itemize}
\item $\sav(L)$ is invariant under strong concordance of links in $M$;
\item $\sav(L_1 \# L_2) = \sav(L_1) + \sav(L_2)$;
\item $\sav(\m L) = -\sav(L)$.
\end{itemize}
\end{corollary}

\begin{proof}
This follows from Theorem~\ref{thm:GenusBoundCylinders} and Proposition~\ref{prop:propertiesS1S2}.
\end{proof}

\subsection{Knotification}

For an $\ell$-component non-empty link $L \subset S^3$, Hedden-Raoux defined $\taut(L)$ and $\taub(L)$ as the top and bottom $\tau$ invariants of the knotification $\kappa(L) \subset \#^{\ell-1} (\SSone)$. As such, by definition, $\taut(L) = \taut(\kappa(L))$ and $\taub(L) = \taub(\kappa(L))$. 

By constrast, the $s$-invariants $s_\pm$ of a $\ell$-component link $L$ in $S^3$ do not seem to be related to the $s$-invariants of $\kappa(L) \subset \#^{\ell-1} (\SSone)$. We give some examples below.

\begin{example}
Let $U_\ell$ be the $\ell$-component unlink. Its knotification is the unknot in $\#^{\ell-1}(\SSone)$ Then $s_\pm(U_\ell) = \pm(\ell-1)$ and $s_\pm(\kappa(U_\ell)) = 0$.
\end{example}

\begin{example}
Let $\Hopf^+$ be the positive Hopf link in $S^3$. Its knotification $\kappa(\Hopf^+)$ is the positive Whitehead knot $\Wh^+ \subset \SSone$. We know that $s_\pm(\Hopf^+) = 1$, $s_-(\Wh^+) = 0$ and $s_+(\Wh^+) = 2$.
\end{example}

\begin{example}
Let $\Hopf^-$ be the negative Hopf link in $S^3$. Its knotification $\kappa(\Hopf^-)$ is the negative Whitehead knot $\Wh^- \subset \SSone$. We know that $s_\pm(\Hopf^-) = -1$, $s_-(\Wh^-) = -2$ and $s_+(\Wh^-) = 0$.
\end{example}

Based on the examples above, we ask the following question.

\begin{question}
Let $L \neq \varnothing$ be an $\ell$-component link in $\SSr$, and let $\kappa(L) \subset \#^{r+\ell-1}(\SSone)$ denote its knotification. Then is
\[
\sav(L) = \sav(\kappa(L)),
\]
where $\sav(L) = \frac{s_-(L) + s_+(L)}2$, as defined in Corollary \ref{cor:shomomorphism}?
\end{question}

\subsection{Positivity}

\begin{definition}
\label{def:positivityM}
We say that a link $L$ in $M=\#^r(\SSone)$ is \emph{positive} if there exists a \emph{positive diagram}, i.e.~a link diagram $D_0 \sqcup K_1 \sqcup \cdots \sqcup K_r$ in $S^3$ such that
\begin{itemize}
\item $K_1 \sqcup \cdots \sqcup K_r$ is a diagram for an $r$-component unlink;
\item $0$-surgery on each $K_i$ turns $D_0$ into a diagram for $L$;
\item $D_0$ is a positive link diagram in $S^3$ (i.e., all crossings of $D_0$ are positive).
\end{itemize}
Given a positive diagram $D = D_0 \sqcup K_1 \sqcup \cdots \sqcup K_r$ for $L$, we define $n^+_D$ to be $n^+_{D_0}$ and $\Seif(D)$ to be the number $\text{Seif}(D_0)$ of Seifert circles of $D_0$, i.e.~ the number of components of the link obtained from $D_0$ by resolving all crossings in an oriented way.
\end{definition}

\begin{remark}
When $K_1 \sqcup \cdots \sqcup K_r$ is the usual diagram for the $r$-component unlink (consisting of $r$ disjoint circles), we recover the kind of standard diagrams used to represent links in $\#^r(\SSone)$, as on the left of Figure~\ref{fig:InsertTwists}; we could then impose a positivity condition on such diagrams. However, Definition~\ref{def:positivityM} is more general, and makes positivity a less restrictive condition.
\end{remark}

\begin{theorem}
\label{thm:positivity2}
Suppose that $D$ is a positive diagram for a null-homologous non-empty link $L$ in $\SSr$. Then
\[
s_-(L) \leq n^+_D - \Seif(D) + 1\leq s_+(L).
\]
\end{theorem}
\begin{proof}
Let $D = D_0 \sqcup K_1 \sqcup \cdots \sqcup K_r$. Since $K_1 \sqcup \cdots \sqcup K_r$ is an $r$-component unlink, after performing some Reidemeister moves we can move $D$ to a standard diagram $D'$ in $\SSr$ for the same link $L$. Then $D'(\vec 0)$ and $D_0$ are diagrams for the same link $L_0$ in $S^3$. Thus, $s_-(D'(\vec 0)) = s_-(D_0)$.
By Proposition \ref{prop:stab}.\eqref{it:stab3} we know that
\[
s_-(L) \leq s_-(D'(\vec 0)) = s_-(D_0),
\]
and by Propositions \ref{prop:stab}.\eqref{it:stab4} and \ref{prop:propertiesS3}.\eqref{it:s-ANDs+} we know that
\[
s_+(L) \geq s_+(D'(\vec 0)) \geq s_-(D'(\vec 0)) = s_-(D_0).
\]
To conclude the proof, we use the fact that for a positive diagram $D_0$ of a link $L_0$ in $S^3$ we have
\begin{equation}
\label{eq:s-positive}
s_-(D_0) = n^+_{D_0} - \text{Seif}(D_0) + 1.
\end{equation}
Equation \eqref{eq:s-positive} can be computed explicitly as follows.
Since the diagram $D_0$ has only positive crossing, from the definition it is immediate to check that the Lee complex $\OLC{D_0}$ is concentrated in non-negative homological grading. Thus, if two cycles in $\OLC[0]{D_0}$ are homologous they must be equal, and therefore the $\q$-filtration level of a cycle $\xi \in \OLC[0]{D_0}$ is also the $\q$-filtration level of its homology class $[\xi]$.
The Lee generator $\s_{D_0}$ is a cycle in homological grading $0$, so
\[
s(D_0) = s_{\min}(D_0) + 1 = \q([\s_{D_0}]) + 1 = \q(\s_{D_0}) + 1 = n^+_{D_0} - \text{Seif}(D_0) + 1,
\]
where $\q(\s_{D_0})$ is computed directly from the definition: for each Seifert circle, either label $x+1$ or $x-1$ contributes $-1$ to the $\q$-filtration level, and there is a global $\q$-shift by $n^+_{D_0}-2n^-_{D_0} = n^+_{D_0}$.
\end{proof}

\begin{remark}
An alternative way to prove Equation \eqref{eq:s-positive} is by using \cite[Theorem 1.3]{AT} (first proved for knots in \cite[Section 5.2]{Rasmussen}) and \cite[Corollary 4.1]{Cromwell}.
\end{remark}

\begin{example}
\label{ex:Wh-positive}
For the positive Whitehead knot $\Wh^+ \subset \SSone$, which is a positive link (see Figure \ref{fig:Whitehead} for a positive diagram $D_{\Wh^+}$), $s_-(\Wh^+) = 0$, $s_+(\Wh^+) = 2$, $n^+_{D_{\Wh^+}} = 2$, and $\Seif(D_{\Wh^+}) = 3$. Thus, for $\Wh^+$, the inequality of Theorem \ref{thm:positivity2} involving $s_+$ is strict, while the one involving $s_-$ is actually an equality.
\end{example}

In fact, based on Example \ref{ex:Wh-positive} and the examples below, we formulate the following conjecture.

\begin{conjecture}
Suppose that $D$ is a positive diagram for a null-homologous link $L$ in $\SSr$. Then
\begin{equation}
\label{eq:positivity}
s_-(L) = n^+_D - \Seif(D) + 1.
\end{equation}
\end{conjecture}

\begin{example}
When $r=0$, i.e.~when the link $L$ is in $S^3$, Equation \eqref{eq:positivity} holds by \cite[Theorem 1.3]{AT} (or the computation in the proof of Theorem \ref{thm:positivity2}).
\end{example}

\begin{example}
More generally, let $L$ be a local link in $M$ with a positive diagram. Then we can find a diagram $D$ that is simultaneously local and positive. This is because we can suppose that the $S^2$'s given by the co-cores of the 1-handles can be chosen disjoint from the 3-ball in which $L$ is contained. This implies that, given any diagram of $L$, the unknotted components $K_1, \ldots, K_r$ bound discs disjoint from $L$. Thus, given a positive diagram for $L$, we move the unknotted components by shrinking them following these discs, while leaving the projection of $L$ unchanged, until we get a diagram $D$ for $L \subset M$ where the unknotted components $K_1, \ldots, K_r$ are disjoint from $L$ and form the standard surgery presentation of $M=\SSr$. Thus, $D$ is positive and local, so we can view it as a diagram of a link $L_{S^3}$ in $S^3$, and $s(L_{S^3}) = s(D) = s(L)$. Thus, Equation \eqref{eq:positivity} follows again from \cite[Theorem 1.3]{AT} (or the computation in the proof of Theorem \ref{thm:positivity2}).
\end{example}

\begin{example}
The link $\Fp$ in $\SSone$ from Section \ref{sec:Lpp} has a crossingless diagram $D$ (see Figure \ref{fig:Fp}), which therefore is positive. We have $s_-(\Fp) = 1-2p$ by Theorem \ref{thm:s(Fp)}. It is immediate to check that $n^+_D=0$ and $\Seif(D) = 2p$. Thus, Equation \eqref{eq:positivity} holds.
\end{example}

\subsection{The slice Bennequin inequality}
The slice Bennequin inequality in $S^3$ was first proved by Rudolph~\cite{SB1} using gauge theory.  

\begin{theorem}[{Slice Bennequin inequality in $S^3$~\cite{SB1}}]
\label{thm:SB1}
Let $L$ be a transverse link in $S^3$, and let $\Sigma \subset B^4$ be a smoothly embedded oriented surface with no closed components, such that $\de \Sigma = L$. Then
\[
\sl(L) \leq -\chi(\Sigma).
\]
\end{theorem}

Lisca-Mati\'c extended the proof to transverse knots in any contact 3-manifold with a Stein filling. They formulated their result in terms of Legendrian knots, but one can re-state it in terms of transverse knots (as we do here), by considering push-offs.

\begin{theorem}[{Slice Bennequin inequality~\cite{SB2}}]
Let $W$ be a Stein 4-manifold with boundary and $F \hookrightarrow W$ a smoothly embedded oriented connected surface transversal to $\de W$ with $K = \de F \subset \de W$ connected and transverse. Then
\[
\sl(K) \leq -\chi(\Sigma).
\]
\end{theorem}

In \cite{SB3}, Shumakovitch gave a combinatorial proof of Theorem \ref{thm:SB1}, based upon the following result.

\begin{theorem}[{\cite[Theorem 1.C]{SB3}}]
\label{thm:SB3}
Let $L$ be a transverse link in $S^3$. Then
\[
\sl(L) + 1 \leq s(L).
\]
\end{theorem}

From here one can prove Theorem \ref{thm:SB1} by using the slice genus bound given by the $s$-invariant~\cite[Equation (7)]{beliakova-wehrli}.

In this section we give a combinatorial proof of the slice Bennequin inequality in the case of null-homologous transverse links in $\SSr$, endowed with the standard tight contact structure $\xistd$, relative to the Stein filling $\natural^r (S^1 \times B^3)$. Other proofs of the slice Bennequin inequality in $\#^r(\SSone)$ are given by Lambert-Cole~\cite{Lambert-Cole} and Hedden-Raoux~\cite{HR}. (Lambert-Cole's proof does not use gauge theory or Heegaard Floer homology, but rather reduces it to the slice Bennequin inequality in $S^3$ using some techniques from contact geometry.)

Let us start by describing an open book decomposition for the standard contact structure on $\SSone$. The binding is given by $S^1 \times \set{N,S}$ and the pages are annuli given by $\set{S^1 \times \mi \,\middle|\, \text{$\mi$ meridian of $S^2$}}$. Here by $N$ and $S$ we mean the north and south poles of $S^2$, respectively, and by ``meridian'' of $S^2$ we mean an arc of a great circle from $N$ to $S$.

This open book decomposition can be envisioned as follows. 
First, consider $S^3$ as the compactification of $\R^3$, with a standard coordinate system $(x,y,z)$. The sphere $S^3$ admits a (standard) open book decomposition where the binding is the circle
\[
C = \set{(x,y,z) \,\middle|\, x^2+y^2=1, z=0}
\]
and the pages are discs. See Figure \ref{fig:OBD}. The compactification in $S^3$ of the $z$-axis is an unknot $U$.
Note that the pages of the open book decomposition are perpendicular to $U$, so the open book decomposition can be approximated near $U$ with horizontal planes (i.e., planes parallel to the $xy$-plane). We let $\nbd(U)$ denote a tubular neighborhood of $U$ in $S^3$.
Performing $0$-surgery on $U$ yields $S^1 \times S^2$.

\begin{figure}
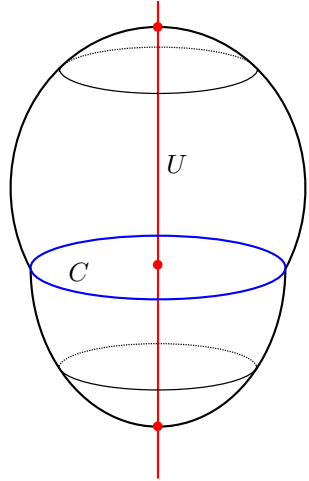
\caption{The standard open book decomposition of $S^3$. The blue circle $C$ is the binding. Two pages (discs) are sketched. The unknot $U$ is shown in red.}
\label{fig:OBD}
\end{figure}

Alternatively, one can see $S^1 \times S^2$ as obtained by taking two copies of $S^3 \sm \nbd(U)$ (oriented in opposite ways) and by gluing them along the boundary via the identity map $\de\nbd(U) \to \de \nbd(U)$. The aforementioned open book decomposition of $S^3$ extends to an open book decomposition of $S^1 \times S^2$. In $S^1 \times S^2$ there are two binding components (the two copies of the circle $C$), and the pages are annuli (each annulus is obtained by gluing two punctured discs together along the punctures). 

Suppose now that $L$ is a transverse link in $\SSone$, and that $\SSone$ is envisioned with the open book decomposition described above. Up to transverse isotopy, we can assume that $L$ misses the binding of the open book and is transverse to the pages. Indeed, Mitsumatsu and Mori \cite{Mitsumatsu} and, independently, Pavelescu \cite{Pavelescu} proved that, given a transverse link in a contact $3$-manifold, and an open book decomposition of the manifold compatible supporting the contact structure, the link can be transversely braided with respect to that open book decomposition. 

By pushing $L$ further, we can assume that it is contained in a single copy of $S^3 \sm \nbd(U)$, and away from $C$. Recall that the pages away from $C$ can be approximated with horizontal planes. Hence, in this picture, transverse links appear as closures of braids oriented from the bottom to the top. We can further suppose that the braid is always on the right of $U$, except when it winds with it. See Figure \ref{fig:transverseS1S2}.

\begin{figure}
\begingroup%
  \makeatletter%
  \providecommand\color[2][]{%
    \errmessage{(Inkscape) Color is used for the text in Inkscape, but the package 'color.sty' is not loaded}%
    \renewcommand\color[2][]{}%
  }%
  \providecommand\transparent[1]{%
    \errmessage{(Inkscape) Transparency is used (non-zero) for the text in Inkscape, but the package 'transparent.sty' is not loaded}%
    \renewcommand\transparent[1]{}%
  }%
  \providecommand\rotatebox[2]{#2}%
  \newcommand*\fsize{\dimexpr\f@size pt\relax}%
  \newcommand*\lineheight[1]{\fontsize{\fsize}{#1\fsize}\selectfont}%
  \ifx\svgwidth\undefined%
    \setlength{\unitlength}{344.98916339bp}%
    \ifx\svgscale\undefined%
      \relax%
    \else%
      \setlength{\unitlength}{\unitlength * \real{\svgscale}}%
    \fi%
  \else%
    \setlength{\unitlength}{\svgwidth}%
  \fi%
  \global\let\svgwidth\undefined%
  \global\let\svgscale\undefined%
  \makeatother%
  \begin{picture}(1,0.41677265)%
    \lineheight{1}%
    \setlength\tabcolsep{0pt}%
    \put(0,0){\includegraphics[width=\unitlength,page=1]{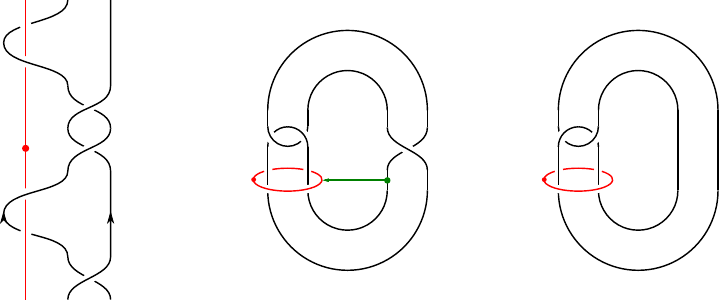}}%
    \put(0.24524002,0.19279287){\color[rgb]{0,0,0}\makebox(0,0)[lt]{\lineheight{0}\smash{\begin{tabular}[t]{l}\Large $\simeq$\end{tabular}}}}%
    \put(0.66699266,0.19279287){\color[rgb]{0,0,0}\makebox(0,0)[lt]{\lineheight{0}\smash{\begin{tabular}[t]{l}\Large $\simeq$\end{tabular}}}}%
  \end{picture}%
\endgroup%

\caption{On the left hand side is a braid $\beta$ in $\SSone$. After sliding the knot over the handle as illustrated above, its closure is seen to be $\Wh^+$, the positive Whitehead knot.}
\label{fig:transverseS1S2}
\end{figure}

Next, consider the connected sum $\SSr$. The standard contact structure on $\SSr$ is obtained from those on $\SSone$ by taking their connected sums. In general, an open book decomposition for a connected sum of contact $3$-manifolds can be constructed by the Murasugi sum of their open books along boundary parallel arcs; see for example \cite[Lemma 2.2]{Etnyre}.
Concretely, by summing up the open books for $\SSone$ we get an open book for $\SSr$, where the binding consists of $r+1$ circles, the pages are $r$-punctured disks, and the monodromy is the identity. We picture transverse links in $\SSr$ as braids just as in Figure~\ref{fig:transverseS1S2}, except that now they can wrap around $r$ vertical red lines. 

Kawamuro~\cite[Equation (1.6)]{Kawamuro} computed the self-linking number of any null-homologous transverse link $L \subset (\SSr, \xistd)$ to be
\begin{equation}
\label{eq:KP}
\sl(L) = \sl(L(\vec 0)) = \wr(\beta) - \br(\beta),
\end{equation}
where $\beta$ is a braid whose closure is $L$. Here $L(\vec 0)$ denotes the transverse link obtained by viewing $L$ in $S^3$, without performing dot-surgeries on $\SSr$. Note that $L(\vec 0)$ is the closure of the same braid $\beta$, this time viewed in $S^3$. By $\wr(\beta)$ and $\br(\beta)$ we denote the writhe and the braid index of $\beta$, respectively.

{
\renewcommand{\thethm}{\ref{thm:sliceBennequin}}
\begin{theorem}
Let $L \neq \varnothing$ be a null-homologous transverse link in $(\SSr, \xistd)$. Then
\[
\sl(L)+1 \leq s_+(L).
\]
\end{theorem}
\addtocounter{thm}{-1}
}
\begin{proof}
We represent $L$ as the closure of a braid $\beta$ in $S^3$. By Equation \eqref{eq:KP}, Theorem \ref{thm:SB3}, Proposition \ref{prop:propertiesS3}.\eqref{it:s-ANDs+}, and Proposition \ref{prop:stab}.\eqref{it:stab4}, we have
\[
\sl(L) + 1 = \sl(L(\vec 0)) + 1 \leq s_-(L(\vec 0)) \leq s_+(L(\vec 0)) \leq s_+(L). \qedhere
\]
\end{proof}

There is another combinatorial proof, more similar in spirit to Shumakovitch's one~\cite{SB3}.

\begin{proof}[Alternative proof]
Let $L$ be obtained as the closure of a braid $\beta$ as in Figure \ref{fig:transverseS1S2} (winding around $r$ red lines). Let $\beta^+$ denote the braid obtained by replacing every negative crossing of $\beta$ with a positive crossing. The link $L^+$ defined as the closure of $\beta^+$ is positive, so, by Equation \eqref{eq:KP} and Theorem \ref{thm:positivity2} we have
\[
\sl(L^+)+1 = \wr(\beta^+) - \br(\beta^+) + 1 = n^+_{D_+} - \Seif(D_+) + 1 \leq s_+(L^+),
\]
where $D_+$ denotes the (positive) diagram obtained by taking the braid closure of $\beta^+$.
Whenever we change a crossing of $\beta^+$ from positive to negative, the left hand side of the inequality decreases by exactly $2$, whereas Proposition~\ref{prop:propertiesS1S2}.\eqref{it:propertiesS1S2-CC} tells us that the right hand side decreases by at most $2$.
\end{proof}

\begin{remark}
Theorem \ref{thm:sliceBennequin} is not true if one replaces $s_+$ with $s_-$.
A counterexample is given by the closure of the braid $\beta$ in Figure \ref{fig:transverseS1S2}.
It is straightforward to check that $\wr(\beta) = 3$ and $\br(\beta) = 2$, thus $\sl(\beta) = 1$.
As JungHwan Park pointed out to us, the knot type of the closure of $\beta$ is $\Wh^+$, the positive Whitehead knot.
From Example~\ref{ex:Whitehead} we see that $s_-(\Wh^+) = 0$. However, $1+1\not\leq0$.
\end{remark}

We finally deduce the slice Bennequin inequality in $\SSr$.

{
\renewcommand{\thethm}{\ref{cor:sliceBennequin}}
\begin{corollary}[Slice Bennequin inequality in $\SSr$]
Let $L$ be a null-homologous transverse link in $(\SSr, \xistd)$, and let $\Sigma$ be a properly embedded surface in $\natural^r(S^1 \times B^3)$ with no closed components and $\de \Sigma = L$. Then
\[
\sl(L) \leq -\chi(\Sigma).
\]
\end{corollary}
\addtocounter{thm}{-1}
}

\begin{proof}
Following Shumakovitch's strategy~\cite{SB3}, we first suppose that $L$ is a knot. In such a case by Theorems \ref{thm:sliceBennequin} and \ref{thm:genus bounds} we have
\[
\sl(L) + 1 \leq s_+(L) \leq 2g(\Sigma) = 1-\chi(\Sigma),
\]
which concludes the proof in the case of knots.

If $L$ is an $\ell$-component transverse link represented by a braid $\beta$, then we can change it into a knot $K$ by attaching $\ell-1$ positively half-twisted bands on the braid $\beta$. Moreover, $K$ bounds a surface $\Sigma'$ in $\natural^r(S^1 \times B^3)$ with $\chi(\Sigma') = \chi(\Sigma) - (\ell-1)$, obtained by attaching the bands to $\Sigma$. Since $K$ is a knot we know that
\begin{equation}
\label{eq:slK1}
\sl(K) \leq -\chi(\Sigma') = -\chi(\Sigma)+(\ell-1).
\end{equation}
The braid $\beta'$ defining $K$ has $\br(\beta') = \br(\beta)$ and $\wr(\beta') = \wr(\beta) + (\ell-1)$. Thus, by Equation \eqref{eq:KP} we have
\begin{equation}
\label{eq:slK2}
\sl(K) = \sl(L) + (\ell-1).
\end{equation}
By combining Equations \eqref{eq:slK1} and \eqref{eq:slK2} we get the conclusion.
\end{proof}

\begin{question}
Is there a counterexample to Corollary \ref{cor:sliceBennequin} if $\natural^r(S^1 \times B^3)$ is replaced by $\natural^r(B^2 \times S^2)$?
\end{question}

\subsection{Genus bounds and Dehn twists}

Let $\DT \colon \SSone \to \SSone$ denote the Dehn twist.

\begin{proposition}
\label{prop:afterD}
Given a null-homologous link $L \subset \SSone$, $\gSD(L) = \gSD(\DT(L))$.
\end{proposition}
\begin{proof}
Note that $\DT$ extends to a diffeomorphism on $S^1 \times B^3$. Thus, given a minimum genus surface $\Sigma$ for $L$, $\DT(\Sigma)$ is an embedded surface for $\DT(L)$. Thus, $\gSD(\DT(L)) \leq \gSD(L)$. The other inequality is analogous.
\end{proof}

Proposition \ref{prop:afterD} does not seem to hold for $\gDS$, although we could not find a counterexample.
JungHwan Park noticed that if such a counterexample exists, the geometric intersection number between the offending link $L$ and any $S^2$ fiber must be at least 4. This is because if $L$ intersects an $S^2$ fiber in exactly two points, then $L$ and $\sigma(L)$ are isotopic in $S^1 \times S^2$, where the isotopy is obtained by taking one local strand of $L$ and dragging it all around the $S^2$ fiber.

%
%
%
\section{Open problems} \label{sec:Problems}

\subsection{Inserting full twists}
In this paper we defined and studied the Rasmussen-type invariant $s$ for null-homologous links  $L \subset M = \SSr$. According to Theorem~\ref{thm:fda}, if we represent $L$ by a diagram $D$, then $s(D)$ can be understood as the ordinary $s$ invariant for the link in $S^3$ represented by the diagram $\Lk$ for $k \gg 0$.

Consider now an arbitrary link $L \subset M$ (represented by a diagram $D$), and let $\eta_i \in \Z$ be its algebraic intersection number with the two-sphere $\{*\} \times S^2$ in the $i$th copy of $\SSone$, for $i=1, \dots, r$. The homology class $[L] \in H_1(M; \Z)$ is determined by the numbers $\eta_i$.

In general, it is not true that $s(\Lk)$ stabilizes as $k \to \infty$. However, we could ask the following:
\begin{question}
\label{question:shifteds}
Let $\vec{k}=(k,\dots,k)$. Is the quantity
\begin{equation}
\label{eq:shifteds}
 s(\Lk) - k  \sum_{i=1}^r |\eta_i| (|\eta_i|-1)
 \end{equation}
independent of $k$ for $k \gg 0$? 
\end{question}

The examples below are what led us to Question~\ref{question:shifteds}. 

\begin{example}
When $[L]=0$, all the $\eta_i$ vanish and we know that $s(\Lk)$ stabilizes by Theorem~\ref{thm:fda}.
\end{example}

\begin{example}
If the answer to Question \ref{question:shifteds} is affirmative for diagrams $D_1$ and $D_2$ for links $L_1 \subset \#^{r_1}(S^1 \times S^2)$ and $L_2 \subset \#^{r_2}(S^1 \times S^2)$, then the answer is also affirmative for $D_1 \sqcup D_2$ and $D_1 \# D_2$ (for any choice of basepoints on $D_1$ and $D_2$), using Proposition \ref{prop:propertiesS1S2}, items \eqref{it:propertiesS1S2-2} and \eqref{it:propertiesS1S2-3}.
\end{example}

\begin{example}
\label{ex:Fpq}
Consider the link $F_{p,q} \subset \SSone$, consisting of $p+q$ fibers $S^1 \times \{x_i\}$, with $p$ of them oriented one way and $q$ oriented the other way. We represent it by a diagram with no crossings, just as the one in Figure~\ref{fig:Fp}, but with different orientations. (By a slight abuse of notation, we will also use $F_{p,q}$ to denote this diagram.) Observe that, if we forget the orientations, then $F_{p,q}(k)$ is the torus link $T_{p+q, k(p+q)}$. Assume $p \geq q$. By computer calculations using the program UniversalKh from the {\em Knot Atlas} \cite{KAT, UniversalKh}, we found that 
$$ s(F_{p,q}(k)) - k(p-q)(p-q-1) =  1-p-q$$
for small values of $p, q$ and $k \geq 1$. Thus, the expression \eqref{eq:shifteds} appears to already stabilize at $k=1$.
\end{example}

\begin{example}
\label{ex:cable}
Consider the $(p,q)$-cable pattern $C_{p,q} \subset \SSone$ given by 
$$ C_{p,q} = \{(e^{ipt}, e^{iqt}) \mid t \in \R\} \subset S^1 \times S^1 \subset \SSone.$$
We represent this by a diagram (denoted the same way) with $p$ strands going through the handle,  all in the same direction, and wrapping around $q$ times. Then $C_{p,q}(k)$ is the torus link $T_{p, q+kp}$, with all strands oriented in one direction. Since $T_{p, q+kp}$ is a positive link, its $s$-invariant can be computed using the formula in \cite[Theorem 1.3]{AT}; see Equation~\eqref{eq:s-positive}. We have
$$ s(C_{p,q}(k)) - kp(p-1)= (p-1)(q-1).$$
\end{example}

\begin{remark}
When $[L] \neq 0$, the expression~\eqref{eq:shifteds} does not give an invariant of the link $L \subset \SSr$. For example, in Example~\ref{ex:cable}, the links $C_{p,0}$ and $C_{p, 2p}$ are isotopic in $\SSone$, because they differ by the square $\sigma^2 \cong \id$ of the Dehn twist diffeomorphism $\sigma$. However, the expressions~\eqref{eq:shifteds} for the corresponding diagrams are $1-p$ and $(p-1)(2p-1)$, respectively.

In fact, recall from Theorem~\ref{thm:deformedKh} that, when $[L] \neq 0$, the deformed Khovanov-Lee homology is only an invariant of $L$ up to grading shifts. Thus, we could not expect the $s$-invariant (which captures some grading information) to be well-defined. 
\end{remark}

\subsection{Surfaces in negative definite four-manifolds}
\label{sec:GeneralSurfaces}
Let $W$ be a smooth, oriented four-manifold with $b^+_2(W) =0$ and $\del W = S^3$. Let $K \subset \del W$ be a knot, and $\Sigma \subset W$ a smoothly embedded oriented surface with no closed components, such that $\del \Sigma =K $.  As noted in Section~\ref{sec:adj} (cf. Theorem~\ref{thm:adjunction-tau}), in \cite{os-tau}, Ozsv\'ath and Szab\'o proved an adjunction inequality for their $\tau$ invariant:
\begin{equation}
\label{eq:tauadj}
 2 \tau(K) \leq 1-\chi(\Sigma) - \left|[\Sigma]\right| - [\Sigma] \cdot [\Sigma].
 \end{equation}
\begin{question}
\label{q:sadj}
Does the inequality~\eqref{eq:tauadj} continue to hold with $2\tau(K)$ replaced by the Rasmussen invariant $s(K)$? 
\end{question}
Of course, one could also ask this in the case where $\del \Sigma$ is a link $L$, and we use $s(L)$.

Less speculatively, let us focus on the manifold $W = (\#^n \bCP) \setminus B^4$, where we proved the  adjunction inequality for $s$ in Corollary~\ref{cor:adjunction}, in the case where $\Sigma$ is null-homologous. For arbitrary surfaces, we expect the following.
\begin{conjecture}
\label{conj:AdjGeneral}
Let $W = (\#^r \bCP) \setminus B^4$ for some $r \geq 0$. Let $L \subset \del W= S^3$ be a link, and $\Sigma \subset W$ a properly, smoothly embedded oriented surface with no closed components, such that $\del \Sigma =L$. Then:
\[
s(L) \leq 1-\chi(\Sigma) - \left|[\Sigma]\right| - [\Sigma] \cdot [\Sigma].
\]
\end{conjecture}

Just as Corollary~\ref{cor:adjunction} was derived (through Theorem \ref{thm:strongadj}) from calculating the $s$ invariant of the cable link $\Fp(1)$, by the same argument Conjecture~\ref{conj:AdjGeneral} boils down to showing that
  \begin{equation}
 \label{eq:sab}
  s(F_{p,q}(1)) = (p-q)^2 - 2p+1,
  \end{equation}
 where $F_{p,q} \subset \SSone$ is the link from Example~\ref{ex:Fpq}, and $p \geq q \geq 0$.
   
 We have checked by computer that \eqref{eq:sab}  holds for small values of $p$ and $q$ (namely, when $p+q \leq 4$). We can also prove the following bounds.
 \begin{proposition}
 For the link $F_{p,q}(1) \subset S^3$, with $p \geq q \geq 0$, we have
 $$ (p-q)^2 - 2p+ 1 \leq s(F_{p,q}(1)) \leq (p-q)^2 -2(p-q)+1.$$
 \end{proposition}
 \begin{proof}
 When $q=0$, we have that $F_{p,0}(1)$ is the positive torus link $T_{p,p}$. By Equation~\eqref{eq:s-positive}, we have
 $$ s(T_{p,p}) = (p-1)^2.$$

When $q\geq 1$, consider a band $B$ joining two oppositely oriented strands of $F_{p,q}$. This gives a cobordism between $F_{p-1, q-1}(1) \sqcup U$ and $F_{p,q}(1)$, where $U$ is the unknot. Topologically, the cobordism is a disjoint union of a pair-of-pants and some annuli, with Euler characteristic $-1$. Moreover, note that every component of this cobordism has boundary components both in $F_{p-1, q-1}(1) \sqcup U$ and in $F_{p,q}(1)$. Applying Theorem~\ref{thm:GenusBoundCylinders} in both directions, we find that
 \begin{equation}
   \label{eq:spq1'}
     s(F_{p,q}(1))+1 \geq s(F_{p-1,q-1}(1) \sqcup U)= s(F_{p-1,q-1}(1)) -1\geq s(F_{p,q}(1))-1.
 \end{equation}
 
Proceeding inductively in $q \geq 0$, with $p-q$ fixed, and using \eqref{eq:spq1'} at the inductive step, we obtain the desired bounds for $s(F_{p,q}(1))$.
 \end{proof}
 
 \begin{remark}
 In Heegaard Floer theory, the analogue of $s$ for links is $2\tau(L) + 1- |L|$. Using the generalization of  \eqref{eq:tauadj} to links in \cite{HR}, we can prove that
 $$ \tau(F_{p,q}(1)) = \frac{(p-q)(p-q-1)}{2},$$
 that is,
 $$2\tau(F_{p,q}(1)) + 1-p-q= (p-q)^2 - 2p + 1,$$
 which is the analogue of \eqref{eq:sab}.
 \end{remark}
 
 \subsection{Homotopy $4$-spheres}
 Recall from Corollary~\ref{cor:Glucktwist} that the Freedman-Gompf-Morrison-Walker strategy for disproving the smooth 4D Poincar\'e conjecture \cite{mnm} cannot work for homotopy $4$-spheres obtained from $S^4$ by Gluck twists. It remains an open question whether the strategy can work for other homotopy $4$-spheres. This can be viewed as a particular case of Question~\ref{q:sadj}, when $W$ is a homotopy $4$-ball. 
 \begin{question}
 Let $X$ be a homotopy 4-sphere. If a link $L \subset S^3$ is strongly slice in $X$, do we have $s(L) = 1-|L|$?
 \end{question}
 
When $X$ was obtained by a Gluck twist from $S^4$, the key facts we used to answer this question in the affirmative were that $X  \# \CP \cong \CP$ and $X \# \bCP \cong \bCP$. The same result would hold for any homotopy 4-sphere $X$ satisfying $X\#(\#^r\CP)=\#^r\CP$ and $X\#(\#^r\bCP)=\#^r\bCP$ for some $r$.

Apart from Gluck twists, there are various other potential counterexamples to the smooth 4D Poincar\'e conjecture in the literature. For example, a large family is obtained by considering finite balanced presentations $P$ of the trivial group, constructing a four-manifold $Z(P)$ by adding one-handles and two-handles according to the generators and relations in $P$, and then taking the double $D(P)$ of $Z(P)$. Then, $D(P)$ is a homotopy $4$-sphere, which is easily seen to be $S^4$ when the presentation $P$ is Andrews-Curtis trivial. See \cite[Exercises 4.6.4(b) and 5.1.10(b)]{GS}.

\begin{question}
Let $P$ be any finite balanced presentation of the trivial group. Does the homotopy four-sphere $D(P)$ satisfy $D(P)  \#(\#^r\CP) \cong \#^r\CP$ and $D(P) \#(\#^r\bCP) \cong \#^r\bCP$ for some $r$?
\end{question}

\subsection{Spectra}
In \cite{LS}, Lipshitz and Sarkar constructed a stable homotopy refinement of Khovanov homology; see also \cite{HKK}, \cite{LLS}. Furthermore, in \cite{LS2}, they showed that stable cohomology operations (such as the Steenrod squares) produce new Rasmussen-type invariants for knots in $S^3$, and in some cases these give bounds on the slice genus that are strictly stronger than those from $s$.

\begin{question}
Can one extend the constructions in \cite{LS}, \cite{HKK} or \cite{LLS} to produce a stable homotopy refinement of the Khovanov homology for links in $\SSr$ defined in \cite{Roz}, \cite{MW}?
\end{question}

\begin{question}
For null-homologous links in $\SSr$, can one define and use the Steenrod squares on Khovanov homology to construct new Rasmussen-type invariants, which would improve the bounds on $\gDS$ and $\gSD$ from Theorem~\ref{thm:genus bounds}?
\end{question}



\bibliographystyle{alpha}
\bibliography{biblio}

\end{document}